\documentclass[a4paper,11pt,oneside,reqno]{amsart}

\usepackage[english]{babel}
\usepackage[utf8x]{inputenc}    
\usepackage[dvipsnames]{xcolor}            
\usepackage{xspace}
\usepackage[pdftex]{graphicx}
\usepackage{epstopdf}
\usepackage{mathrsfs}
\usepackage{stmaryrd}
\usepackage{enumitem}	
\usepackage{verbatim}
\usepackage{amssymb}	
\usepackage{csquotes}	

\usepackage{subcaption}	

\makeatletter
\newcommand{\leqnos}{\tagsleft@true\let\veqno\@@leqno}
\newcommand{\reqnos}{\tagsleft@false\let\veqno\@@eqno}
\reqnos
\makeatother

\numberwithin{equation}{section}

\usepackage{caption}           
\usepackage{url}

\usepackage[a4paper,scale={0.72,0.74},marginratio={1:1},footskip=7mm,headsep=10mm]{geometry}

\usepackage{hyperref}
\hypersetup{					
	breaklinks=true,			
	colorlinks,
    citecolor=blue,
    filecolor=black,
    linkcolor=blue,
    urlcolor=black
}

\newcommand{\ind}{{\sf 1}}

\newcommand{\bP}{\mathbf{P}}

\newcommand{\bE}{\mathbf{E}}
\newcommand{\bbP}{\mathbb{P}}
\newcommand{\Pb}{\mathbb{P}}

\newcommand{\bbE}{\mathbb{E}}

\newcommand{\Pbspin}{\mathbb{P}_{\pm1}}		
\newcommand{\Pspinx}{\mathbb{P}_{\pm x}}		
\newcommand{\Pbspinx}{\mathbb{P}_{\pm x}}		
\newcommand{\bbPspinx}{\mathbb{P}_{\pm x}}	
\newcommand{\bbEspinx}{\mathbb{E}_{\pm x}}	
\newcommand{\tP}{\tilde\bbP}			
\newcommand{\tE}{\tilde\bbE}			
\newcommand{\bbR}{\mathbb{R}}
\newcommand{\R}{\mathbb{R}}
\newcommand{\bbN}{\mathbb{N}}
\newcommand{\N}{\mathbb{N}}
\newcommand{\bbZ}{\mathbb{Z}}

\newcommand{\Qb}{\overline Q}			

\newcommand{\C}{\mathcal{C}}

\newcommand{\G}{\mathcal{G}}

\newcommand{\A}{\mathcal{A}}

\newcommand{\cA}{{\ensuremath{\mathcal A}} }
\newcommand{\cF}{{\ensuremath{\mathcal F}} }

\newcommand{\cC}{{\ensuremath{\mathcal C}} }
\newcommand{\cN}{{\ensuremath{\mathcal N}} }

\newcommand{\cG}{{\ensuremath{\mathcal G}} }

\newcommand{\dd}{\mathrm{d}}

\renewcommand{\epsilon}{\varepsilon}
\renewcommand{\phi}{\varphi}

\renewcommand{\tilde}{\widetilde}

\newcommand{\ol}{\overline}

\newcounter{cst}[section]
\newcounter{svf}[section]
\newcommand{\cntc}{{\stepcounter{cst}\arabic{cst}}}
\newcommand{\cntf}{{\stepcounter{svf}\arabic{svf}}}
\newcommand{\rcntc}[1]{{\refstepcounter{cst}\arabic{cst}\label{#1}}}

\newtheorem{theorem}{Theorem}[section]

\newtheorem{assumption}[theorem]{Assumption}

\newtheorem{proposition}[theorem]{Proposition}

\newtheorem{lemma}[theorem]{Lemma}

\theoremstyle{definition}

\newtheorem{remark}[theorem]{Remark}

\newcommand{\ga}{\alpha}
\newcommand{\gb}{\beta}
\newcommand{\gd}{\delta}
\newcommand{\gep}{\varepsilon}  
\newcommand{\eps}{\varepsilon}      

\newcommand{\gr}{\rho}

\newcommand{\gz}{\zeta}

\newcommand{\go}{\omega}

\newcommand{\hgo}{\hat\omega}
\newcommand{\bgo}{\bar\omega}

\newcommand{\gl}{\lambda}

\newcommand{\gs}{\sigma}

\newcommand{\gh}{\eta}

\renewcommand{\tilde}{\widetilde}
\renewcommand{\hat}{\widehat}

\newcommand{\tf}{\mathtt{F}}
\newcommand{\ann}{\mathrm{a}}
\newcommand{\quen}{\mathrm{q}}
\newcommand{\btau}{\boldsymbol{\tau}}

\newcommand{\bgs}{\boldsymbol{\sigma}}

\newcommand{\bfrakS}{\boldsymbol{\mathfrak{S}}}
\newcommand{\bnu}{\boldsymbol{\nu}}
\newcommand{\bgn}{\boldsymbol{\nu}}

\newcommand{\bgr}{\boldsymbol{\rho}}

\newcommand{\gsb}{\bar \sigma}
\newcommand{\bi}{\textbf{\textit{i}}}
\newcommand{\bj}{\textbf{\textit{j}}}
\newcommand{\bn}{\textbf{\textit{n}}}
\newcommand{\bm}{\textbf{\textit{m}}}

\newcommand{\ba}{\textbf{\textit{a}}}
\newcommand{\bb}{\textbf{\textit{b}}}
\newcommand{\bc}{\textbf{\textit{c}}}
\newcommand{\bd}{\textbf{\textit{d}}}
\newcommand{\bk}{\textbf{\textit{k}}}

\newcommand{\bl}{\textbf{\textit{l}}}
\newcommand{\bone}{{\boldsymbol{1}}}
\newcommand{\bzero}{{\boldsymbol{0}}}
\newcommand{\binfty}{{\boldsymbol{\infty}}}
\newcommand{\Var}{\mathbb{V}\mathrm{ar}}

\newcommand{\llbbrack}{{\boldsymbol{\llbracket}}}
\newcommand{\rrbbrack}{{\boldsymbol{\rrbracket}}}
\newcommand{\bnsquare}{\llbbrack \bone,\bn\rrbbrack}			
\newcommand{\intsquaretwo}[2]{\llbbrack #1,#2\rrbbrack}			
\newcommand{\boldsquare}[1][\bn]{\llbbrack \bone,#1\rrbbrack}		
\newcommand{\bsquare}[1][\bn]{\llbbrack \bone,#1\rrbbrack} 		
\newcommand{\bsquaretwo}[2]{\llbbrack #1,#2\rrbbrack}			

\newcommand{\sumtwo}[2]{\sum_{\substack{#1 \\ #2}}}
\newcommand{\prodtwo}[2]{\prod_{\substack{#1 \\ #2}}}
\newcommand{\limtwo}[2]{\lim_{\substack{#1 \\ #2}}}

\renewcommand{\preceq}{\preccurlyeq}		
\renewcommand{\succeq}{\succcurlyeq}		
\newcommand{\ligned}{\leftrightarrow}

\newcommand{\blue}{}

\title[The disordered gPS model for DNA denaturation]{Influence of disorder on DNA denaturation:\\ the disordered generalized Poland-Scheraga model}

\author[A. Legrand]{Alexandre Legrand} 
\address{Universit\'e de Nantes, Laboratoire Jean Leray, UFR Sciences et Techniques, 2 rue de la Houssini\`ere, BP 92208 F-44322 Nantes Cedex 3, France
}
\email{alexandre.legrand@univ-nantes.fr}

\date{}

\begin{document}
\begin{abstract}
The Poland-Scheraga model is a celebrated model for the denaturation transition of DNA, which has been widely used in the bio-physical literature to study, and  investigated by mathematicians. In the original model,  only opposite bases of the two strands can be paired together, but a generalized version of this model has recently been introduced, and allows for mismatches in the pairing of the two strands, and for different strand lengths. This generalized Poland-Scheraga (gPS) model has only been studied recently in the case of homogeneous interactions, then with disordered interactions perturbed by an i.i.d. field.
The present paper considers a disordered version of the gPS model which is more  appropriate to depict the inhomogeneous composition of the two strands (in particular interactions are perturbed in a strongly dependent manner): we study the question of the influence of disorder on the denaturation transition,
and our main results provide criteria for disorder (ir)-relevance, both in terms of critical points and of order of the phase transition.
Surprisingly, we find that criteria for disorder relevance depend on the law of the disorder field.
We discuss this with regards to Harris' prediction for disordered systems.
\end{abstract}

\maketitle
\tableofcontents

\section{Introduction}
The Poland-Scheraga (PS) model has been introduced in \cite{PS70} to formally study the DNA denaturation phenomenon, that is the unbinding of two strands of DNA as temperature increases. It has proven to be relevant from a quantitative point of view (see e.g. \cite{BBBDDKMS99, BD98}) and has been subject to much interest from the mathematical, physical and biophysical communities (see e.g.  \cite{Fish84, Giac07, FdH07, KMP01}). In the \emph{homogeneous} version of the model, \textit{i.e.}\ when bases in each strand are all the same (for instance AAA$\ldots$ and TTT$\ldots$), an interesting feature is that the model is solvable: 
it is proven to undergo a denaturation (or delocalization) phase transition, and its critical behavior can be described precisely, cf. \cite[Ch.~2]{Giac07}.

In the PS model, it is assumed that the two strands are of equal length, and that only bases from each strand with the same index can be paired. To depict DNA denaturation more accurately, the generalized Poland-Scheraga (gPS) model has been introduced more recently, where those assumptions are relaxed, see \cite{GO03, GO04, NG06}. From the mathematical point of view, the gPS model can be described as a pinning model based on a two-dimensional renewal process, see  \cite{GK17}. Interestingly, the homogeneous version of the model remains solvable, despite having a much more complex behavior ---in particular it has (in general) other critical points, corresponding to ``condensation'' phase transitions, see \cite{GK17} and~\cite{BGK18}.

The PS and gPS models can naturally embody the inhomogeneous character of DNA.  In the PS model, one introduces a sequence of random variables---referred to as \emph{disorder} in statistical mechanics--- describing the inhomogeneous binding energies of successive pairs. A disordered version of the gPS model has been studied recently in \cite{BGK}, with the introduction of a two-dimensional disorder field: the random variable of index $(i,j)\in\N^2$ corresponds to the binding energy of the $i$-th base of the first strand with the $j$-th base of the second strand. In \cite{BGK}, the authors chose the disorder field to be i.i.d.: this assumption is relevant when using the gPS model to portray the pinning of a polymer on a inhomogeneous surface, or a directed (stretched) polymer in a random environment (in the spirit of \cite{Com06,Wei16}). However this choice is not satisfactory when describing the denaturation phenomenon between two inhomogeneous chains: the binding energy of a pair $(i,j)$ should be a function of the $i$-th and $j$-th bases of each strand ---in particular the binding energies of two pairs $(i,j)$ and $(i,k)$ are not independent because they share a common base.

The purpose of this paper is twofold:
\begin{itemize}
\item  study the gPS model in a setting which portrays more faithfully the pinning of two inhomogeneous polymers, as in DNA denaturation;
\item  make progress on the understanding of disordered systems when disorder/random-ness is slightly elaborate, in particular not i.i.d.
\end{itemize}

\begin{figure}[h]
 \centering
  \begin{subfigure}[b]{0.48\linewidth}
    \includegraphics[width=\linewidth]{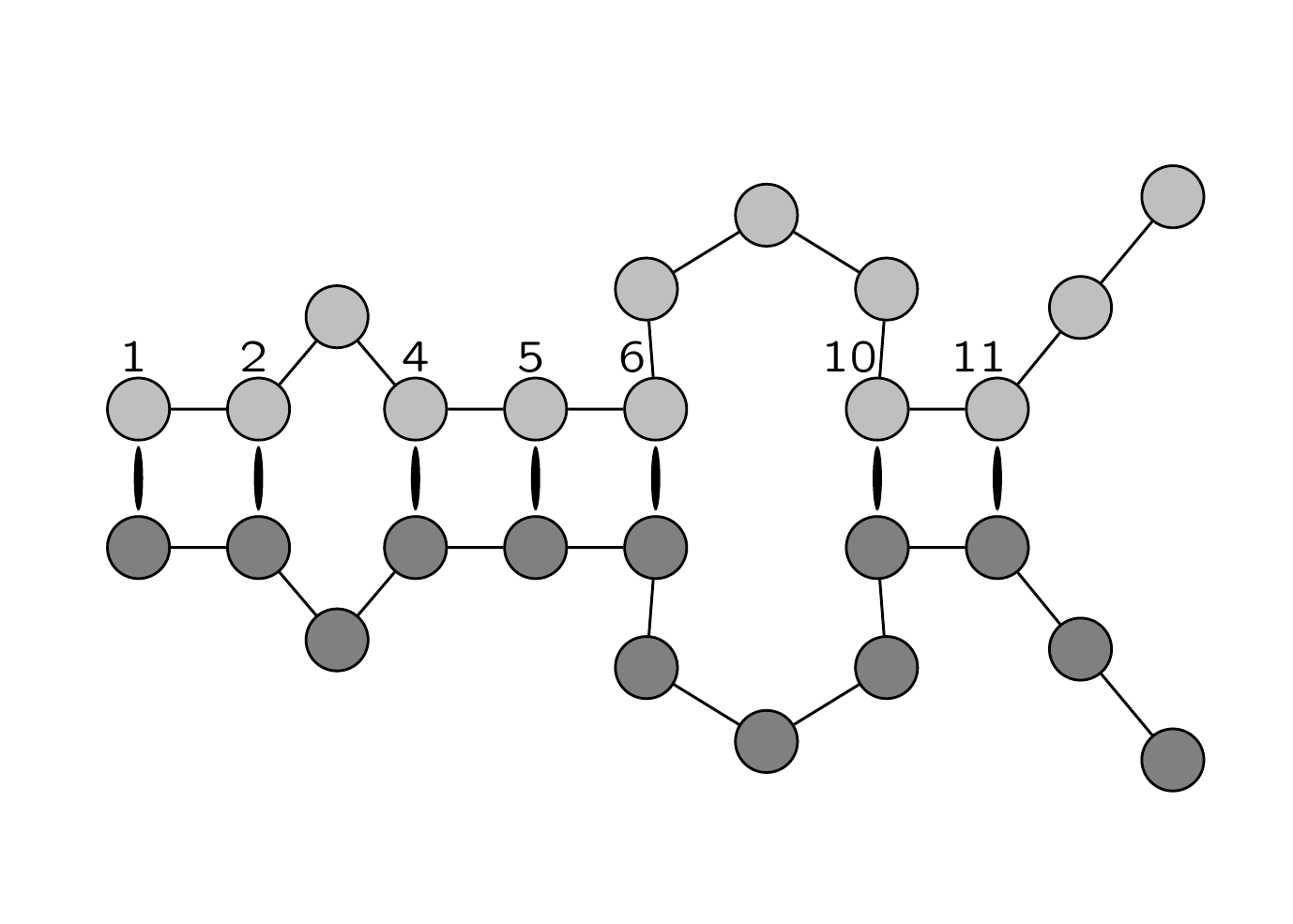}
    \caption{Poland-Scheraga.}
  \end{subfigure}
  \begin{subfigure}[b]{0.48\linewidth}
    \includegraphics[width=\linewidth]{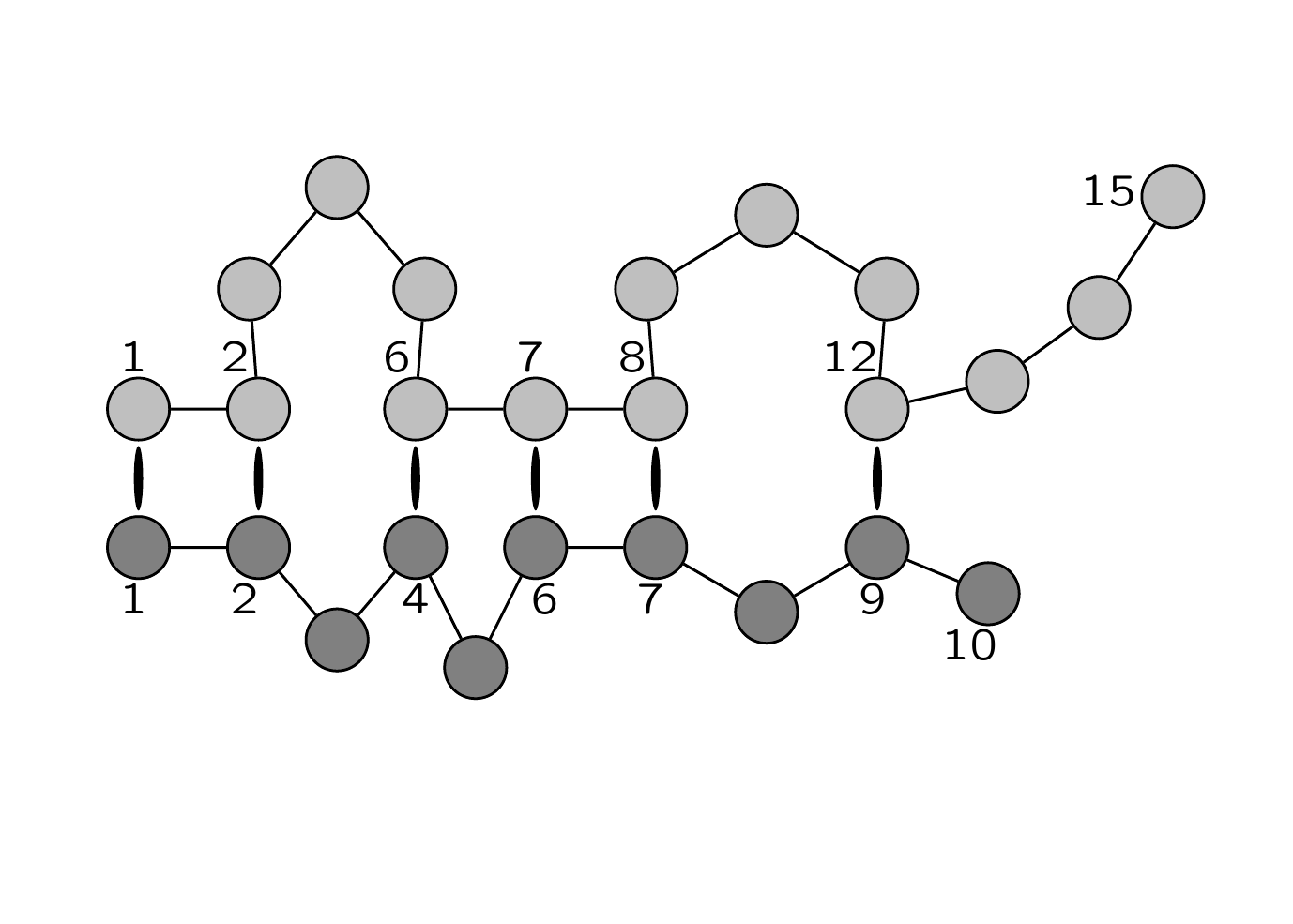}
    \caption{Generalized Poland-Scheraga.}
  \end{subfigure}
  \caption{\footnotesize Representations of the PS and gPS models. In the first figure, the two strands are of length 13 and are bounded symmetrically, on pairs 1, 2, 4, 5, 6, 10 and 11. In the other one, the two strands are of different lengths and are bound on pairs $(1,1)$, $(2,2)$, $(6,4)$, $(7,6)$, $(8,7)$ and $(12,9)$.}
  \label{fig:PSmodel}
\end{figure}

\noindent
{\it Acknowledgment.} I am very grateful to Quentin Berger for the support, prolific discussions and sound advice throughout my first steps researching in the domain of polymer models.
I am also very thankful to the anonymous referees for their constructive comments which substantially improved the quality of this paper.

\smallskip \noindent
{\it Some notation.} In the remainder of the paper, bold characters $\bn,\bi,\bj,\ldots$ will denote elements of $\N^2$ (or $\bbZ^2$), and plain characters elements of $\N$ or $\R$. In particular we denote $\bone:=(1,1)$, $\bzero:=(0,0)$. For $r\in\{1,2\}$, the projection of any element $\bn\in\N^2$ on its $r$-th coordinate will be denoted $\bn^{(r)}\in\N$.

\subsection{The generalized Poland-Scheraga model}
Let $\btau=(\btau_i)_{i\geq0}$ be a bivariate renewal process: $\btau_0 =(0,0)$, and $(\btau_i-\btau_{i-1})_{i\geq 1}$ are i.i.d., $\bbN^2$-valued random variables. We denote $\bP$ its law, and we assume that the inter-arrival distribution satisfies for all $a,b\in\N$,
\begin{equation}\label{eq:interarrival:tau}
\bP(\btau_1= (a,b)) \,=\, K(a+b) \,:=\, \frac{L(a+b)}{(a+b)^{2+\ga}}\; ,
\end{equation}
where $\ga> 0$, and $L(\cdot)$ is a slowly varying function (that is $L(ux)/L(x)\to1$ as $x\to\infty$ for any $u>0$, see \cite{BGT87}).
We also assume that $\btau$ is persistent, \textit{i.e.}\ $\sum_{n,m\geq 1} K(n+m)=1$.
With a slight abuse of notation, we write $\btau := \{\btau_0,\btau_1, \btau_2, \ldots\} \subset \N^2$ the set of renewal points, and from now on we will omit the point $\btau_0=(0,0)$.
Notice that $\btau^{(r)}:=\{\btau_1^{(r)}, \btau_2^{(r)}, \ldots\}$ is a univariate renewal process with inter-arrival distribution $\bP(\btau_1^{(r)}=a)=\tilde L(a) a^{-(1+\ga)}$, with $\tilde L(n)\sim (1+\ga)^{-1} L(n)$ some slowly varying function.

Let $\go=(\go_\bi)_{\bi\in\N^2}$ be a field of real random variables indexed in $\N^2$, whose law is denoted~$\bbP$
($\go_{(i,j)}$ represents the binding potential between the $i$-th and $j$-th bases of the first and second strand respectively). We assume that they all have the same law, and that there exists some $\gb_0\in(0,\infty]$ such that for all $\gb \in [0,\gb_0)$,
\begin{equation}\label{def:lambda}
\lambda(\gb)\,:=\, \log \bbE[ e^{\gb \go_{\boldsymbol{1}}}]\,<\infty\;,
\end{equation}
(this is satisfied by bounded laws and by many unbounded laws, notably Gaussian or the product of two independent Gaussian variables). 

\smallskip
For a fixed realization of $\go$ (\emph{quenched} disorder), we define, for $\gb\in [0,\gb_0)$ (the disorder strength) and $h\in \bbR$ (the pinning potential), the following \emph{polymer (Gibbs) measure}: for any renewal set $\btau\subset \N^2$ and $\bn=(\bn^{(1)},\bn^{(2)})\in\N^2$,
\begin{equation}
\label{def:polymermeasure}
\frac{\dd \bP_{\bn,h}^{\gb,\go,\quen}}{\dd \bP} (\btau)\,:=\, \frac{1}{Z_{\bn,h}^{\gb,\go,\quen}}  \exp\bigg( \sum_{\bi\in \bnsquare} \big( \gb \go_\bi -\lambda(\gb)+h\big) \ind_{\{\bi \in \btau\}} \bigg)  \ind_{\{ \bn\in \btau \}} \; ,
\end{equation}
where $Z_{\bn,h}^{\gb,\go,\quen}$ is the \emph{partition function} with \emph{quenched disorder},
\begin{equation}
Z_{\bn,h}^{\gb,\go,\quen}\,:=\, \bE\bigg[  \exp\bigg( \sum_{\bi\in\bnsquare } \big( \gb \go_\bi -\lambda(\gb)+h\big) \ind_{\{\bi \in \btau\}} \bigg)  \ind_{\{ \bn\in \btau \}} \bigg]\;,
\end{equation}
and $\bnsquare$ denotes $\llbracket1,\bn^{(1)}\rrbracket\times \llbracket1,\bn^{(2)}\rrbracket\subset\N^2$. This represents the binding of two strands with respective lengths $\bn^{(1)}$ and $\bn^{(2)}$, and $\bi\in\btau$ if and only if the base $\bi^{(1)}$ of the first strand is paired with the base $\bi^{(2)}$ of the second strand. The polymers are constrained to be bound on the last pair $\bn$, and we give a reward (or a penalty if negative) $\gb \go_\bi-\lambda(\gb)+h$ for each bound pair $\bi\in\btau$. Notice that the term $-\gl(\gb)$ in the reward is present only for renormalization purposes, see for instance~\eqref{eq:annealed} below. From now on, we will drop the superscript $\go$ in the quenched partition function and Gibbs measure to lighten notations, even though they are functions of $\go$.

Let us now precise our choice of disorder field. In \cite{BGK}, the authors studied the gPS model under an i.i.d. disorder field $\go=(\go_\bi)_{\bi\in\N^2}$. In this paper we want the disorder field to depict the inhomogeneous composition of the two strands: we pick two independent sequences $\hat \go = (\hat \go_{i_1})_{i_1\in\N}$ and $\bar \go = (\bar \go_{i_2})_{i_2\in\N}$ of i.i.d. random variables, whose distributions are denoted $\hat \bbP$ and $\bar \bbP$ respectively. These random variables are thought as being charges attached to the two strands. For each $\bi\in\N^2$, we fix
\begin{equation}
\go_{\bi}\;:=\; f(\hat \go_{\bi^{(1)}} ,\, \bar \go_{\bi^{(2)}})\;,
\end{equation}
where $f(\cdot, \cdot)$ is a function describing the interactions between the monomers. We will write $\bbP:=\hat \bbP \otimes \bar \bbP$ with an abuse of notation. We stress right away that $\go:=(\go_\bi)_{\bi\in\bbN^2}$ is a strongly correlated field, but that $\go_{\bi}$ and $\go_{\bj}$ are independent as soon as $\bi,\bj\in\N^2$ are not aligned, \textit{i.e.} are not on the same line or column: $\bi^{(1)}\neq \bj^{(1)}$ and $\bi^{(2)}\neq \bj^{(2)}$.

\subsection{The free energy and the denaturation transition}

A physical quantity central to the study of the model is the free energy, defined in the following proposition.
\begin{proposition}
\label{prop:existF}
For all $\gamma>0$, $h \in \bbR$, $\beta \ge 0$ and every sequence $\{m(n)\}_{n=1,2, \ldots}$ such that $\lim_{n\to\infty}  m(n)/n=\gamma$,
the following limit exists:
\begin{equation}
\label{eq:saf}
\lim_{n\to\infty} \frac 1n \log Z^{\gb,\quen}_{\bn,h} \, =\,  \lim_{n\to\infty} \frac 1n \bbE \log Z^{\gb,\quen}_{\bn,h}\, =:\, \tf_{\gamma}(\beta,h) \;,
\end{equation}
where $\bn:=(n,m(n))$, both $\bbP(\dd \go)$-almost surely and in $L^1(\bbP)$. 
Also, $(\gb,h) \mapsto \tf_{\gamma}(\gb,h+\gl(\gb))$ is non-negative and convex (therefore continuous on $(0,\infty)\times\R$), $h\mapsto \tf_{\gamma}(\gb,h)$ and $\gb\mapsto \tf_{\gamma}(\gb,h+\gl(\gb))$ are non-decreasing, and $\gamma \mapsto \tf_{\gamma}(\gb,h)$ is non-decreasing and 
 continuous. Moreover, we have, for any $0<\gamma_1 \leq \gamma_2$,
 \begin{equation}
 \label{eq:boundsF}
  \tf_{\gamma_1} (\beta,h) \,\leq\, \tf_{\gamma_2} (\gb,h) \,\leq\, \frac{\gamma_2}{\gamma_1}\, \tf_{\gamma_1}(\gb,h) \; .
 \end{equation}
\end{proposition}

This proposition is analogous to \cite[Thm.~1.1 and Prop~2.1]{BGK} : the proof of these results is not affected by our choice of (correlated) disorder in any way whatsoever, because any trajectory of $\btau$ contributing to $Z_{\bn,h}^{\gb,\quen}$ only involves an i.i.d. subfamily of the field $\go$: indeed, if $\bi,\bj \in \btau$ and $\bi\neq \bj$, then necessarily they are not aligned because the inter-arrivals of $\btau$ are in $\bbN^2$, hence $\go_{\bi}$ and $\go_{\bj}$ are independent. Therefore the proof of Proposition~\ref{prop:existF} is an immediate replica of that of  \cite[Thm.~1.1 and Prop~2.1]{BGK}.\smallskip

Proposition~\ref{prop:existF} allows us to define the \emph{(quenched) critical point}:
\begin{equation}
h_c^{\quen}(\gb)\,=\,h_c(\beta)\,:=\, \inf \lbrace h \,: \,  \tf_\gamma (\beta,h) >0 \rbrace \, .
\end{equation}
We stress that $h_c(\gb)$ does not depend on $\gamma >0$, thanks to \eqref{eq:boundsF}. 
 

\smallskip
The critical point $h_c(\gb)$ marks the transition between a \emph{localized}  and a \emph{delocalized} phase: this is the so-called \emph{denaturation} (or \emph{(de)-localization}) transition.
Indeed, a standard calculation gives that $\partial_h \log Z^{\gb,\quen}_{\bn,h}  = \bE_{\bn,h}^{\gb,\quen}\big[ \sum_{\bi\in\bnsquare}  \ind_{\{\bi \in \btau\}}  \big]$: by exploiting the convexity of the free energy and Proposition~\ref{prop:existF}, we get that
\begin{equation}
\label{contactfraction}
\partial_h \tf_\gamma (\beta,h)  \,=\, \limtwo{n\to+\infty,}{m(n)/n\to \gamma} \bE_{\bn,h}^{\gb,\quen} \Big[ \frac{1}{n} \sum_{\bi\in\llbracket\bone,\bn\rrbracket}  \ind_{\{\bi \in \btau\}} \Big]\;,
\end{equation}
whenever $\partial_h \tf_\gamma (\beta,h)$ exists.
Therefore, for $h>h_c(\gb)$ we have $\partial_h \tf_\gamma (\beta,h) >0$, and in view of \eqref{contactfraction}, there is a positive density of contacts between the two strands: they stick to each other. On the other hand, for $h<h_c(\gb)$ we have $\partial_h \tf_\gamma (\beta,h) = 0$, and there is a zero density of contacts: the two strands wander away from one another.

\subsection{The homogeneous and annealed models}

The \emph{homogeneous} model corresponds to the case when there is no disorder, \textit{i.e.}\ $\beta = 0$.
This model has been proven to be {\sl exactly solvable}, and a fine analysis of $\tf_\gamma (0,h)$ has been performed in~\cite{GK17}.

\begin{theorem}[Thm.~1.2 in \cite{GK17}]
\label{thm:hom}
For any $\gamma \ge 1$, $h_c(0):= \inf \{ h\,:\, \tf_{\gamma}(0,h)>0\}= 0$. Moreover there are a slowly varying function $L_{\ga}(\cdot)$ and a constant $c_{\alpha,\gamma}$ such that
\begin{equation}
\tf_\gamma (0,h) \,\sim\,  c_{\ga,\gamma} L_{\ga} ( 1/h) \, h^{1/ \min (1, \ga)} \;,\quad \text{as }h \searrow 0\,.
\end{equation}
\end{theorem}
The exponent $1/ \min (1, \ga)$ is often referred to as the \emph{critical exponent}: it is the main quantification of the behavior of the model around its phase transition. Explicit expressions of $L_\ga$ are given in \cite{GK17}, in particular it is some constant if $\ga>1$.

\smallskip
The \emph{annealed} model, on the other hand, corresponds to averaging the partition function over the disorder: the annealed partition function is, for $\gb\in[0,\gb_0)$,
\begin{align}
Z_{\bn,h}^{\gb,\ann}\,:=\,\bbE \big[Z^{\gb,\quen}_{\bn,h}\big] \,&=\, \bbE \, \bE\Big[ \exp\Big( \sum_{\bi\in\bnsquare} \big( \gb \go_\bi- \lambda(\gb)+h\big) \ind_{\{\bi \in \btau\}} \Big)  \ind_{\{ \bn\in \btau \}}  \Big] \notag\\
&=\, \bE\Big[ \exp\Big( \sum_{\bi\in\bnsquare} h\, \ind_{\{\bi\in \btau\}} \Big)  \ind_{\{ \bn\in \btau \}}  \Big]\,=\,Z_{\bn,h}^{0,q}    \; .
\label{eq:annealed}
\end{align}
Here we used that for any fixed trajectory of $\btau$ the non-zero terms $( \gb \go_\bi- \lambda(\gb)+h ) \ind_{\{\bi \in \btau\}}$ are independent and that $\lambda(\gb)=\log \bbE[e^{\gb \go_\bone}] <+\infty$ for $\gb\in[0,\gb_0)$ (in particular this implies $Z^{\gb,\quen}_{\bn,h}\in L^1(\bbP)$).

Notice that the annealead model matches exactly the homogeneous model. Recalling Proposition~\ref{prop:existF}, the annealed free energy is therefore
\begin{equation}\label{eq:fa}
\tf_\gamma^{\ann}(\beta,h)\,:=\, \limtwo{n\to\infty,}{m(n)/n \to \gamma} \frac 1n \log Z_{\bn,h}^{\gb,\ann} = \tf_\gamma (0, h ) \;.
\end{equation}
We also directly have that the annealed critical point is 
$h_c^{\ann}(\beta):= \min \lbrace h \,: \,  \tf^{\ann}_\gamma (\beta,h) >0 \rbrace = 0$ (recall Theorem~\ref{thm:hom}).

\smallskip
Now, a simple use of Jensen's inequality in \eqref{eq:saf} gives that $\tf_\gamma(\beta,h) \le \tf_\gamma^{\ann}(\beta,h)$. Moreover, 
we have that $\tf_\gamma (0,h) \le \tf_\gamma (\beta ,h + \gl(\gb))$ (recall that $\beta \mapsto \tf_\gamma (\beta,h+\gl(\gb))$ is non-decreasing). As a conclusion, we obtain the following bounds for the quenched critical point: for every~$\beta$ we have
\begin{equation}
\label{eq:inh}
0 \,=\, h_c^{\ann}(\beta)\,  \le\,  h_c(\beta) \,\le\, h_c(0) +\gl(\gb)\,=\,\gl(\gb)\;.
\end{equation} 
An adaptation of the proof of
\cite[Th.~5.2]{Giac07} would easily give that the second inequality is strict for every $\gb > 0$.
The first inequality may or may not be strict and this is an important issue which is directly linked to disorder relevance or irrelevance.

\smallskip
In the rest of the paper, we will work in the case $\gamma=1$: recall that having $\gamma \neq 1$ changes neither the value of the critical point $h_c(\gb)$, nor the homogeneous critical behavior (up to a constant factor, see inequality~\eqref{eq:boundsF} and Theorem~\ref{thm:hom}). To simplify notations, we will drop the dependence on $\gamma$ in the free energy.

\section{Presentation of the results: the question of disorder relevance}
In general, going from a homogeneous model to a disordered one is a complex matter in statistical mechanics (even in the PS model, see \cite[Ch. 5]{Giac07}). A first issue is wether the phase transition ---in this paper we focus on the denaturation transition--- {\blue remains when we introduce a small disorder}
; if so, at what critical value and with what critical behavior compared to the homogeneous model.
If any disorder with any strength ---parametrized by $\gb$ in our setting--- changes the critical behavior (notably the critical exponent) of the model from the homogeneous case, disorder is said to be \emph{relevant}; if a disorder of small strength does not change the critical behavior, it is said to be \emph{irrelevant}. 

The physicist Harris~\cite{H74} predicts that disorder (ir)-relevance for a $d$-dimensional system can be determined from the correlation length exponent $\nu$ in the homogeneous model.
If we admit that the correlation length is given by the reciprocal of the free energy, we obtain from Theorem~\ref{thm:hom} that $\nu=1/\min(1,\ga)$. Then Harris' criterion predicts that when $\nu>2/d$ disorder should be irrelevant, and when $\nu<2/d$ it should be relevant (the case $\nu=2/d$, dubbed \emph{marginal}, is much harder to treat, even with heuristic methods).

Notice that in our setting, determining the dimension $d$ of the system is a more delicate issue than it seems: even though the disorder field $\go$ is indexed in $\N^2$, it is constructed from two sequences $\hgo$ and $\bgo$, therefore has a 1-dimensional degree of freedom. It is not obvious if one should pick $d=1$ or $d=2$ in Harris' criterion. Actually we prove that there are two possible criteria for disorder (ir)-relevance depending on the law $\bbP$, which correspond to Harris' prediction for each value $d\in\{1,2\}$. We prove disorder irrelevance (same critical point and exponent), and disorder relevance (shift of the critical point and smoothing of the phase transition) for small disorder intensity in both cases.

\smallskip {\blue
We will study disorder (ir)-relevance in the case of $\hat \go$ and $\bar \go$ having the same distribution $\hat \bbP = \bar \bbP$, and with a product interaction function $f(\cdot,\cdot)$:
\begin{equation}\label{eq:disorder:product}
\go_{\bi}\,:=\, f(\hgo_{\bi^{(1)}} , \bgo_{\bi^{(2)}} ) \,=\,\hgo_{\bi^{(1)}} \times \bgo_{\bi^{(2)}} \; .
\end{equation}
We assume that $\hgo_1,\bgo_1$ are not constant a.s. (otherwise the model is homogeneous). The condition that $\bbE[e^{\gb \go_\bone}]$ is finite for $\beta< \gb_0$ can be guaranteed simply by asking that $\bbE[ e^{\frac{1}{2} \gb \hgo_1^2 }] < +\infty$ for $\beta< \gb_0$, using that $xy \leq (x^2+ y^2)/2$.
This is verified for example when $\hgo, \bgo$ are sequences of Gaussian variables, or when $\hgo, \bgo$ are bounded.
Let us denote the moments of $\hgo_1$ by $m_k:=\bbE[\hgo_{1}^k]=\bbE[\bgo_{1}^k]$ for all $k\in\N$. In particular $\bbE[\go_\bone^k]=m_k^2$ for all $k\in\N$.

} 


\subsection{Main results I: disorder irrelevance}

Our first result is the following theorem, showing disorder irrelevance for $\ga<1/2$, regardless of the law $\bbP$.

\begin{theorem}\label{thm:irrel}
Let $\btau$ and $\btau'$ be two independent copies of a renewal process with law $\bP$. If $\btau^{(1)}\cap\btau'^{(1)}$ (or equivalently $\btau^{(2)}\cap\btau'^{(2)}$) is terminating  (in particular if $\ga<1/2$), then there exists $\beta_1>0$ such that for every $\beta\in [0,\beta_1)$, one has:
(i) $h_c(\beta)=h_c^{\ann}(\beta)=0$; (ii) for any $h\in [0,1]$,
\begin{equation}\label{eq:thm:irrel}
L_\cntf(1/h) \, h^{1/\ga} \,\leq\, \tf(\gb,h) \,\leq\, L_\ga(1/h) \, h^{1/\ga}\;,
\end{equation}
for some (explicit) slowly varying functions $L_{\arabic{svf}}, L_\ga$.
\end{theorem}

Recall that $\btau^{(1)}$ is a univariate renewal process such that $\bP(\btau^{(1)}=a)=\tilde L(a) a^{-(1+\ga)}$, where $\tilde L\sim (1+\ga)^{-1}L$. Proposition~\ref{prop:intertau} gives a necessary and sufficient condition for $\btau^{(1)}\cap\btau'^{(1)}$ to be terminating ---in particular Theorem~\ref{thm:irrel} holds for $\ga<1/2$. The upper bound in \eqref{eq:thm:irrel} is a direct consequence of Jensen's inequality, \eqref{eq:fa} and Theorem~\ref{thm:hom}: so the interesting features are (i) and the lower bound in \eqref{eq:thm:irrel}. This result implies that, provided that $\gb$ is small enough, the quenched critical point and the quenched critical exponent are the same as those given by Theorem~\ref{thm:hom} for the homogeneous and annealed models---which means that disorder is irrelevant.

{\blue When $\ga>1/2$, we will state below that disorder is relevant for \emph{almost} all disorder laws. Indeed, let us define for any $x>0$ the distribution
\begin{equation}\label{def:Pbspinx}
\Pbspinx(\hgo_1=x)\,=\,\Pbspinx(\hgo_1=-x)\,=\,1/2\;,
\end{equation}
(and $\bgo_1$ has same law). Note that this family of distributions $(\Pbspinx)_{x>0}$ is characterized by the identities $\bbEspinx[\hgo_1]=m_1=0$, and $\Var_{\pm x}(\hgo_1^2)=m_4-m_2^2=0$ ---that is $\hgo_1^2=x^2$ a.s..}


\begin{theorem}\label{thm:irrel:spin}
Assume that $\bbP=\Pspinx$ for some $x>0$. Then the results of Theorem~\ref{thm:irrel} hold as soon as $\btau\cap\btau'$ is terminating (in particular if $\ga<1$). That is, there exists $\beta_1>0$ such that for every $\beta\in [0,\beta_1)$, one has (i) $h_c(\beta)=0$; (ii) \eqref{eq:thm:irrel} holds.
\end{theorem}
The fact that $\ga<1$ is a sufficient condition for having $\btau\cap\btau'$ terminating (while $\ga\leq1$ is necessary)
is ensured by Proposition~\ref{prop:intertau} below. In particular,
Theorem~\ref{thm:irrel:spin} shows that disorder with distribution $\Pbspinx$ is irrelevant for all $\ga\in(0,1)$.

\subsection{Main results II: disorder relevance}
In the two previous theorems, we stated that if $\Pb\neq\Pbspinx$ for all $x>0$, disorder is irrelevant when $\ga<1/2$; and if $\Pb=\Pbspinx$ for some $x>0$, it is irrelevant when $\ga<1$. We now prove disorder relevance in each case for $\ga>1/2$ and $\ga>1$ respectively,

We first focus on the shift of the critical point, starting with the case $\Pb\neq\Pbspinx$.

\begin{theorem}\label{thm:rel:shift}
Assume $\alpha>1/2$ and $\Pb\neq\Pbspinx$ for all $x>0$. Then for any fixed $\eps>0$ (small), there exists $\gb_{\eps}>0$ such that for every $\gb\in(0,\gb_\eps)$, one has
{\blue \begin{equation}
h_c(\gb)\;\geq\;
\left\{\begin{aligned} \gb^{\max(\frac{2\ga}{2\ga-1},2)\,+\,\eps} \qquad & \text{if}\quad m_1\neq 0, \\
\gb^{\max(\frac{4\ga}{2\ga-1},4)\,+\,\eps} \qquad& \text{if}\quad m_1= 0,\, m_4>m_2^2. \end{aligned}\right. 
\end{equation}}
In particular $h_c(\gb)>h_c^\ann(\gb)=0$ for all $\gb\in(0,\gb_\eps)$.
\end{theorem}

{\blue Notice that all (non-constant) distributions other than $\Pbspinx$, $x>0$ are covered by the assumptions $m_1\neq0$ or $m_4>m_2^2$.} To complete this result, we provide an upper bound on~$h_c(\gb)$.

\begin{proposition}\label{prop:hcub:nonopt}
Assume $\alpha>1/2$ and $\Pb\neq\Pbspinx$ for all $x>0$. Then there exist $\gb_1>0$ and some slowly varying function $L_\cntf$, such that for any $\gb\in[0,\gb_1)$,
\begin{equation}\label{eq:hcub:nonopt}
h_c(\gb)\,\leq\, L_{\arabic{svf}}(1/\gb) \, \gb^{\max(\frac{2\ga}{2\ga-1},2)} \;.
\end{equation}
\end{proposition}

{\blue Let us stress that this upper bound (almost) matches the lower bound found in Theorem~\ref{thm:rel:shift} when $m_1\neq 0$; however it is not satisfactory when $m_1=0$.}
 We will discuss in Section~\ref{sec:hcub:nonopt} why we strongly believe that this upper bound can be improved to (almost) match the lower bound when $m_1=0$, see Remark~\ref{rem:bornesup} below ---actually, when $\hgo,\bgo$ are two sequences of i.i.d., centered Gaussian variables, computations of Section~\ref{subsection:secondmoment} can be carried out exactly and give an upper bound on the shift of order $L_{\arabic{svf}}(1/\gb) \gb^{\max(\frac{4\ga}{2\ga-1},4)}$, which (almost) matches the lower bound from Theorem~\ref{thm:rel:shift}.

\smallskip
{\blue When $m_1=0$, $m_4=m_2^2$ ---that is $\bbP=\Pbspinx$ for some $x>0$--- and $\ga>1$, we prove (almost) optimal bounds on the critical point shift. }

\begin{theorem}\label{thm:rel:shift:m4=1}
Assume $\bbP=\Pbspinx$ for some $x>0$, and $\ga>1$. Then for every $\eps>0$, there exist $\gb_{\eps}>0$ such that for any $\gb\in[0,\gb_{\eps})$,
\begin{equation}\label{eq:rel:shift:m4=1}
\gb^{\max(\frac{2\ga}{\ga-1},4) + \eps} \,\leq\, h_c(\gb)\,\leq\, L_{\cntf}(1/\gb) \, \gb^{\max(\frac{2\ga}{\ga-1},4)}\;,
\end{equation}
with $L_{\arabic{svf}}$ a slowly varying function. In particular $h_c(\gb)>0$ for all $\gb\in(0,\gb_\eps)$.
\end{theorem}

This fully covers the shift of the critical point for all disorder laws (except for the marginal cases: $\ga=1/2$ when $\Pb\neq\Pbspinx$ and $\ga=1$ when $\Pb=\Pbspinx$; we will discuss them at the end of this section).

\smallskip
With regards to the critical exponent, in the case $\Pb\neq\Pbspinx$ and $\ga>1/2$ we prove that the phase transition is smoother in the disordered model than in the homogeneous one. This smoothing phenomenon has been first highlighted for the disordered pinning model by \cite{GTsmooth}, and our proof follows the same lines. Unfortunately we could not prove a smoothing phenomena in the case $\Pb=\Pbspinx$ ---this will be further discussed in Section~\ref{subsection:comments}.

As in \cite{GTsmooth} we need an additional assumption on the disorder law $\Pb$ (mostly for technical reasons).

\begin{assumption}\label{hyp:entropy}
Let $\tP_\delta$ denote the law of $(1+\delta)\hgo_1$ for any $\gd\in\R$. There are $c>0$ and $\gd_0>0$ such that for all $\gd \in (-\gd_0,\gd_0)$,
\begin{equation}\label{eq:hyp:entropy}
H(\tP_\gd|\Pb)\,:=\, \tE_\delta\Big[\log\frac{\dd \tP_\delta}{\dd \bbP}\Big]
\,\leq\, c\, \delta^2 \;,
\end{equation}
(Recall that $H(\tP_\gd|\Pb)$ is well-defined as soon as $\frac{\dd \tP_\delta}{\dd \bbP}$ exists, and it is non-negative).
\end{assumption}
Of course we make the same assumption regarding $\bgo$ (we assumed $\hat \bbP = \bar \bbP$). Notice that Assumption~\ref{hyp:entropy} is verified when $\hgo,\bgo$ are Gaussian sequences, and for many unbounded laws; however it does not hold for bounded disorder, in particular it does not hold for $\bbP_{\pm x}$.

\begin{theorem}\label{thm:rel:smoothing}
Suppose that Assumption~\ref{hyp:entropy} holds for $\Pb$ {\blue (in particular $\Pb\neq\Pbspinx$ for all $x>0$)}. Then {\blue for any $\gb\in(0,\gb_0)$}, there are constants $c_\gb$ and $t_\gb>0$ such that for any $t \in (0,t_\gb)$, one has
\begin{equation}\label{eq:thm:rel}
\tf (\beta,h_c(\gb) +t) \,\leq\, c_\gb \, t^2\;.
\end{equation}
\end{theorem}

When $\ga>1/2$, Theorem~\ref{thm:rel:smoothing} shows that disorder has a smoothing effect  on the phase transition. Indeed Theorem~\ref{thm:hom} claims that that the homogeneous model has a critical exponent $1/\min(1,\alpha)$, which is stricly smaller than $2$ if $\alpha>1/2$. {\blue Notice also that Theorem~\ref{thm:rel:smoothing} is not restricted to small values of $\gb$ (as opposed to other results in this paper).}

\subsection{Some comments on the results and the techniques of the proofs}
\label{subsection:comments}
\subsubsection*{Criteria for disorder (ir)-relevance: dependence on $\bbP$.}
An interesting feature of our setting is that, unlike the PS model or gPS with i.i.d. disorder, criteria on $\bP$ for disorder (ir)-relevance are not the same for all disorder distributions $\bbP$. We may foresee this peculiarity by looking at the correlation between rewards given by two different indices, that is $\bbE[e^{\gb(\go_\bi+\go_\bj)}]-\bbE[e^{\gb\go_\bi}]\bbE[e^{\gb\go_\bj}]$, $\bi,\bj\in\N^2$ (in particular those correlations appear in the proof of Theorems~\ref{thm:irrel} and \ref{thm:irrel:spin} in Section~\ref{subsection:secondmoment}, when computing the second moment of the partition function). It is obviously 0 if $\bi,\bj$ are on different lines and columns, and it is greater than $0$ if $\bi=\bj$. However if $\bi\neq\bj$ are on the same line or column (that is $\bi^{(1)}=\bj^{(1)}$ or $\bi^{(2)} =\bj^{(2)}$) then this correlation is $0$ \emph{if and only if} $\bbP=\Pbspinx$ for some $x>0$ ---otherwise it is positive for $\gb$ small (this follows from a Taylor expansion).

Therefore, when $\bbP\neq\Pbspinx$ for all $x>0$, the field $(e^{\gb \go_\bi})_{\bi\in\N^2}$ of rewards given by the disorder has strong correlations on each line and column, and the criteria for disorder (ir)-relevance in that case correspond to Harris' prediction for one-dimensional systems, \textit{i.e.} the marginal regime is at the value $\ga=1/2$. Whereas when $\bbP=\Pbspinx$ for some $x>0$, that field is much less correlated (all its two-point correlations are 0), and it matches Harris' prediction for two-dimensional systems, \textit{i.e.} the marginal regime is at $\ga=1$.

\subsubsection*{Comparison with the existing literature and other models.}
Let us compare our results to~\cite{BGK}, where the authors studied the gPS model with an i.i.d. disorder field $(\go_\bi)_{\bi\in\N^2}$. In this setting, they prove that disorder is irrelevant as soon as $\btau\cap\btau'$ terminates, and that the critical point is shifted for all $\ga>1$ by $\gb^{\max(\frac{2\ga}{\ga-1},4)+\eps}$. Those results are the same as our Theorems~\ref{thm:irrel:spin} and \ref{thm:rel:shift:m4=1}. This supports the idea that in our setting, when $\bbP=\Pbspinx$ for some $x>0$, the field $(e^{\gb \go_\bi})_{\bi\in\N^2}$ is poorly correlated (even though it is not i.i.d.) and has the same influence on the system as an i.i.d. field. However, when disorder has another distribution, the field becomes highly correlated and the comparison with i.i.d. disorder falls appart ---in particular the criteria for disorder (ir)-relevance are not the same.

Another interesting comparison is to the standard PS model (or the pinning model). The question of disorder relevance has been studied extensively in that context, see \cite{A08, BL18, DGLT09, Giac07,GTsmooth,Ton08} among others. For that model, it has been proven that disorder is irrelevant if and only if the process $\tau\cap\tau'$ terminates, with $\tau,\tau'$ two independent copies of the univariate renewal process ---in particular it is irrelevant for all $\ga<1/2$ and relevant for all $\ga>1/2$, where we assume that the inter-arrival distribution is $\bP(\tau_1=a)=L(a) a^{-(1+\ga)}$. With regards to Theorems~\ref{thm:irrel} and \ref{thm:rel:shift}, and noticing that the processes $\btau^{(r)}$ in our setting and $\tau$ in the PS model are very similar, we observe analogous criteria for disorder (ir)-relevance in both systems ---the standard PS model and our setting with $\bbP\neq\Pbspinx$--- with a marginal value $\ga=1/2$.

{\blue 
Finally, the order of the shift of the critical point when $m_1\neq0$ in our setting matches the standard PS model ---that is $\gb^{\max(\frac{2\ga}{2\ga-1},2)}$--- but not when $m_1=0$. Rewriting the disorder field $\go_{(i,j)} = \gz_i\xi_j + m_1 \gz_i + m_1 \xi_j + m_1^2$ with $\gz_i$, $\xi_j$ independent, centered variables, this seems to imply that when $m_1\neq0$, the main contribution to the disorder are the terms $m_1 \gz_i + m_1 \xi_j$, which are very similar to the disorder of the PS model ; whereas if $m_1=0$, the remaining term $\gz_i\xi_j$ carries a lower disorder intensity, leading to a shift of the critical point of smaller order.
}

\subsubsection{About the marginal cases.} 
In our paper we did not fully treat the marginal cases, that is $\ga=1/2$ when $\bbP\neq\Pbspinx$ and $\ga=1$ when $\bbP=\Pbspinx$. By comparing our setting to the PS model, it is expected that the assumptions of Theorems~\ref{thm:irrel} and \ref{thm:irrel:spin} are optimal, and that a shift of the critical point can be proven whenever $\btau^{(r)}\cap\btau'^{(r)}$ is persistent for $\bbP\neq\Pbspinx$, respectively whenever $\btau\cap\btau'$ is persistent for $\bbP=\Pbspinx$ (notice that in \cite{BGK} the authors make the same conjecture for the gPS model with i.i.d. disorder). This is only conjectural for now, and we expect that a great amount of technical work is needed in both cases (for instance, bivariate renewal estimates are very complex in the case $\ga=1$, see~\cite{B18}).

{\blue
\subsubsection{About higher disorder intensities.} 
To answer the issue of disorder (ir)-relevance, we only had to consider small disorder intensities $\gb$ (similarly to \cite{BGK}). It would still be interesting to study the disordered model for any $\gb\in(0,\gb_0)$, yet it seems particularly challenging. First, there is no clear reason for $h_c(\gb)$ to be non-decreasing in $\gb$, hence for $h_c(\gb)$ to be positive for all $\gb>0$ under the assumptions of Theorems~\ref{thm:rel:shift} or \ref{thm:rel:shift:m4=1} (the proof from \cite{GLT11} for the PS model does not apply here due to the correlations of $\go$). Second, apart for the smoothing of the phase transition, our methods rely heavily on taking small values of $\gb$: regarding disorder irrelevance in Section~\ref{section:irrel}, we require that $\gl(2\gb)-2\gl(\gb)$ is sufficiently small to ensure that some exponential moments are finite; and regarding the shift of the critical point, we require an upper bound like \eqref{eq:Afrachom:taylorexp} in the change of measure procedure, which we obtained with a Taylor expansion argument. Both those estimate are not guaranteed for larger values of $\gb$, and may even depend heavily on the choice of disorder distribution $\bbP$.

Nonetheless when $\gb$ is very large (close to $\gb_0$), one could adapt \cite[Theorem~2.1]{Ton08annealed} to prove that there is a strong disorder regime even when $\ga<1/2$ (see \cite[Corollary~3.2, Remark~3.3]{Ton08annealed}): this is based on a quite straightforward fractional moment estimate.
}

\subsubsection{About the proof of disorder relevance} 
Let us make some technical comments on our results for disorder relevance, starting with the shift of the critical point. 
Our lower bounds on the critical point shift are obtained with a coarse graining procedure, together with estimates on the fractional moments of the partition function obtained via a change of measure argument. This method has first been applied in \cite{DGLT09} for the PS model, and was adapted to the i.i.d. gPS model in \cite{BGK}. In this paper we use the same coarse-graining procedure as \cite{BGK}, but the change of measure argument ---which in \cite{BGK} relies on an i.i.d. tilt of the field $\go$--- cannot be replicated straightforwardly in the general case, because of the non-independent structure of $\go$ (we only do it for distributions $\Pbspinx$ in Section~\ref{section:shift:m4=1}, when proving the lower bound of Theorem~\ref{thm:rel:shift:m4=1}). In the case $\bbP\neq\Pbspinx$, we introduce another change of measure ---a simultaneous tilt of both sequences $\hgo$ and $\bgo$--- to prove Theorem~\ref{thm:rel:shift}. Interestingly enough, this change of measure relies on the correlated structure of the disorder, and does not lead to pertinent estimates in the case $\bbP=\Pbspinx$. Moreover, it is less costly than the tilt of the field $\go$ ---for a system of size $\bn=(n,n)$, we tilt $2n$ variables instead of $n^2$. In comparison to \cite{BGK} or the case $\bbP=\Pbspinx$, this induces the appearance of a shift of the critical point when $\ga\in(1/2,1]$, and a greater shift when $\ga\in(1,2]$.

Let us stress that we have not pursued optimal upper and lower bounds on the critical point shift: the $\gb^\eps$ in Theorems~\ref{thm:rel:shift}-\ref{thm:rel:shift:m4=1} can certainly be replaced by slowly varying functions with a more sophisticated coarse graining technique, (see \cite{BL18,CTT15} and \cite[Ch. 6]{Giac10} for the PS model). However the proofs of Theorems~\ref{thm:rel:shift} and \ref{thm:rel:shift:m4=1} are already rather technical, and getting sharper bounds would have been even more laborious; so we decided to stick with these ``almost'' optimal lower bounds for the sake of clarity.

\smallskip
As far as Theorem~\ref{thm:rel:smoothing} is concerned, a small disappointment in our result is that our proof relies heavily on Assumption~\ref{hyp:entropy}, whereas one could expect to prove the same smoothing inequality for any disorder law other than $\Pbspinx$---the same way smoothing inequalities have been more recently proven for all laws in some disordered polymer models (including PS and copolymer models, see \cite{CdH}). Because the disorder we consider is highly correlated, many technical difficulties appear when trying to prove a smoothing inequality, even under strong technical assumptions (\textit{e.g.} bounded and symmetric disorder). We do not try to expand Theorem~\ref{thm:rel:smoothing} under other assumptions in this paper, because we could not handle the computations without very restrictive assumptions, and even so the computations remain extremely cumbersome.

In the case $\bbP=\Pbspinx$, we do not have any result on the smoothing of the phase transition when $\ga>1$. Actually, proving a smoothing inequality is also an issue in the gPS model with i.i.d.\ disorder. It is conjectured in \cite[Conj. 1.5]{BGK} that the gPS model model with (say Gaussian) i.i.d.\ disorder undergoes a smoothing phenomenon with exponent $\min(\frac{2\ga}{\ga+1},\frac{4}{3})$ when $\ga>1$. Regarding our previous comments about the similarity between the i.i.d. gPS model and our setting with $\bbP=\Pbspinx$, and comments in \cite{BGK} leading the authors to their conjecture, it is reasonable to expect that a similar smoothing phenomenon occurs in our setting for disorder with distribution $\Pbspinx$, but it is purely conjectural for now.

\subsubsection{Open questions and perspectives}
Let us list here some questions on this model that are not answered yet, and perspectives for further study:

(i) Improving some of our results which are not fully satisfactory, namely: proving a smoothing inequality for $\bbP=\Pbspinx$ (this matter seems highly related to the gPS model with i.i.d. disorder), proving a smoothing without additional assumptions when $\bbP\neq\Pbspinx$ (which seems very technical), improving the upper bound from Proposition~\ref{prop:hcub:nonopt} to match Theorem~\ref{thm:rel:shift} when $m_1=0$.

(ii) Dealing with the marginal cases, in particular $\bbP=\Pbspinx$ and $\ga=1$, because bivariate renewal processes with $\ga=1$ are not extensively understood yet. 

(iii) We also omitted the case $\ga=0$, for two reasons:
the case $\ga=0$ should be ``strongly irrelevant'', in the sens that the quenched and annealed critical points should \emph{always} be equal (in the same spirit as in \cite{AZ10}), so it should somehow be easier to prove disorder irrelevance;
 from a technical point of view, there are no estimates for bivariate renewals in the case $\ga=0$ (there is no $\alpha$-stable domain of attraction), and this should therefore require a separate analysis.

{\blue
(iv) Dealing with higher values of $\gb$. As mentioned above, proving a strong disorder regime when $\gb$ approaches $\gb_0$ should be feasible even when $\ga<1/2$ by adapting \cite[Theorem~2.1]{Ton08annealed}. However handling intermediate values of $\gb$ may require much more technicalities or even different methods.
}


(v) As stated in \cite{GK17}, the gPS model undergoes other phase transitions. What is the effect of disorder on those? Do they still occur in the disordered model and are their features modified?
These questions are addressed in \cite{NG06}, but only via heuristic and numerical arguments.
{\blue
More recently, in~\cite{GH20}, 
the authors treat generalized (one-dimensional) pinning models whose homogeneous versions exhibit a \emph{big-jump} phase transition ---analogous to the other phase transitions in the homogeneous gPS model, see~\cite{BGK18}. They show that 
the presence of disorder is incompatible with the presence of a big jump, and therefore completely smears out the \emph{big-jump} phase transition (contradicting the conclusions of~\cite{NG06}). This suggests in our case that only the localization/delocalization phase transition survives in the disordered gPS model.
}

(vi) About the choice of disorder : with regards to the interaction function $f$, we chose a product function in \eqref{eq:disorder:product}, but we can conjecture that a generic (symmetric) function should lead to the same criteria for disorder (ir)-relevance, depending on the correlations of the field $(e^{\gb\go_\bi})_{\bi\in\N^2}$. However we assumed the two sequences $\hgo$, $\bgo$ ({\it i.e.} the two strands) to be independent, while it is known that two DNA strands have a strong symmetry (an A-base on one strand faces a T-base on the other, same for C- and G-bases). Therefore, a more pertinent choice of disorder field would be
\begin{equation}
\go_{\bi}\,:=\, f(\hgo_{\bi^{(1)}} ,\, \hgo_{\bi^{(2)}})\;,
\end{equation}
with $\hgo$ an (i.i.d.) sequence of random variables. This setting represents even more faithfully the denaturation of DNA, and it has been considered in numerical studies (see \cite{GO03, GO04}). But it is much more difficult to handle than the one in this paper: the argument of Propostition~\ref{prop:existF} does not apply (even if the sequence $\hgo$ is still assumed i.i.d.), hence even the convergence of the quenched free energy is not clear; moreover the annealed model does not match the homogeneous one anymore, and its analysis seems very challenging.


\subsection{Notation and organisation of the paper}
Let us introduce some notations for the subsequent sections. $C_1, C_2,\ldots$ will denote some constants, and $L_1, L_2,\ldots$ some slowly varying functions. Unless otherwise specified, they may depend on $\ga$ but will not depend on any other parameter $\bn, h, \gb,\ldots$. Moreover $L$ will always denote the function of the inter-arrival distribution (see \eqref{eq:interarrival:tau}), and $L_\ga$ the function introduced in Theorem~\ref{thm:hom} for the critical behavior of the homogeneous model.

For any $\bi, \bj\in\N^2$, we will note $\bi\leftrightarrow\bj$ if $\bi=\bj$ or $\bi$ and $\bj$ are on the same line or column (\textit{i.e.}\ if $\bi^{(r)}=\bj^{(r)}$ for some $r\in\{1,2\}$, recall the notation $\bi = (\bi^{(1)},\bi^{(2)})$ for all $\bi\in\N^2$). We also introduce an order on $\N^2$:
\begin{equation}\label{eq:defprec}
\begin{aligned}
\bi \prec \bj &\quad\text{if}\quad \bi^{(1)} < \bj^{(1)},\,\bi^{(2)} < \bj^{(2)}\;,\\
\bi \preceq \bj &\quad\text{if}\quad \bi^{(1)} \leq \bj^{(1)},\,\bi^{(2)} \leq \bj^{(2)}\;.
\end{aligned}
\end{equation}
For all $a<b\in\N$, we define $\intsquaretwo{a}{b}:=[a,b]\cap\N$. For all $\ba\preceq\bb\in\N^2$, $\bsquaretwo{\ba}{\bb}$ denotes the rectangle in $\N^2$ with bottom-left corner $\ba$ and top-right corner $\bb$: $\bsquaretwo{\ba}{\bb}:=\{\bm\in\N^2\,; \ba\preceq\bm\preceq\bb\}$. 
Finally we will note $\|\cdot\|$ the $L^1$ norm on $\N^2$: $\|\bi\|:=\bi^{(1)}+\bi^{(2)}$ for any $\bi\in\N^2$.

\smallskip
The remainder of the paper is organized into four sections. In Section~\ref{section:irrel} we show how to use a second moment method together with second moment estimates to prove both disorder irrelevance (Theorems~\ref{thm:irrel} and \ref{thm:irrel:spin}) and upper bounds on the shift of the critical point (Proposition~\ref{prop:hcub:nonopt} and right-hand-side inequality in Theorem~\ref{thm:rel:shift:m4=1}).
In Section~\ref{section:smoothing} we prove the smoothing inequality of Theorem~\ref{thm:rel:smoothing} under Assumption~\ref{hyp:entropy}, via a rare-stretch strategy.
In Section~\ref{section:shift}, we display the coarse-graining method, and compute estimates on the fractional moments to obtain lower bounds on the shift of the critical point when $\Pb\neq\Pbspinx$ for all $x>0$ (Theorem~\ref{thm:rel:shift})
Finally in Section~\ref{section:shift:m4=1}, we adapt the change of measure argument from \cite{BGK} to distributions $\Pbspinx$ to prove the same estimates on the fractional moments, thereby proving the left-hand side inequality in Theorem~\ref{thm:rel:shift:m4=1}. Section~\ref{section:shift:m4=1} relies on the coarse-graining procedure introduced in Section~\ref{section:shift}, the other sections are independent.

Additionally, we provide in Appendix~\ref{section:appendixA} some computations on the partition function of the homogeneous gPS model that are needed in Section~\ref{section:shift}, in Appendix~\ref{section:appendixB} we collect useful estimates on bivariate renewals, and in Appendix~\ref{section:appendixC} we compute the upper bound claimed below Proposition~\ref{prop:hcub:nonopt} for Gaussian disorder.

\section{Disorder irrelevance}\label{section:irrel}
In this section, we choose $\bn=(n,n)$ where $n\in\N$ (recall that we assume $\gamma=1$). 
Let us define the \emph{free} version of the gPS model, where the constraint $\ind_{\{\bn\in \btau\}}$ is removed in \eqref{def:polymermeasure}, \textit{i.e.}\ the endpoints are free: its partition function is defined as
\begin{equation} Z_{\bn,h}^{\gb, \quen, \textit{free}}\,:=\, \bE\Big[  \exp\Big( \sum_{\bi\in\bnsquare} \big( \gb \go_\bi - \lambda(\gb)+h\big) \ind_{\{\bi \in \btau\}} \Big)  \Big] \; .\end{equation}
We claim that the constrained and free partition functions are comparable: more precisely for any $\ga_+>\ga$ and $\bm\in\N^2$,
\begin{equation}\label{eq:compZ}
Z_{\bm,h}^{\gb,\quen}\,\leq\, Z_{\bm,h}^{\gb,\quen,\textit{free}} \,\leq\, Z_{\bm,h}^{\gb,\quen} \big(1+C_{\cntc}\|\bm\|^{3+\ga_+}\!\sup_{\bl\leq\bm,\,\bl\ligned\bm}\!(e^{\gb(\go_\bl-\go_\bm)})\big)\;,
\end{equation}
where $C_{\arabic{cst}}$ is a uniform constant (see \cite[Lem. 2.2]{BGK}: here again the proof is not altered by our setting of disorder). In particular, the free energy is not modified with respect to \eqref{eq:saf}: we have that $\tf(\beta,h) = \lim_{n\to\infty} \frac {1}{n} \log Z^{\gb,\quen,\textit{free}}_{\bn,h}$, both $\bbP(\dd \go)$-a.s.\  and in $L^1(\bbP)$.

Note that, as in \cite{BGK18}, we defined the free partition function simply by removing the constraint $\{\bn\in \btau\}$, but let us mention that in the literature the free ends may be assigned some entropic cost: for instance see \cite{GK17}, where the authors consider the (homogeneous) free partition function $\sum_{i,j\leq n} K_f(i)K_f(j) Z^{0,\quen}_{\bn,h}$, with $K_f$ some regularly varying function. However this would not affect our results, since it has no effect at the level of the free energy.


\subsection{Second moment method}
Our proof of disorder irrelevance relies on the idea that, if $\sup_{n\in \bbN} \bbE\big[  (Z^{\gb,\quen,\textit{free}}_{\bn,0})^2 \big] <+\infty$, then $Z^{\gb,\quen,\textit{free}}_{\bn,h}$ should remain concentrated around its mean, at least in a certain regime for $n,h$; in particular the quenched and annealed free energy should remain close to each other (this idea has been exploited and turned into a proof in \cite{A08,Ton08} for the PS model).
In \cite{L10} (for the PS model again), it is roughly showed that as long as $\bbE\big[  (Z^{\gb,\quen,\textit{free}}_{\bn,0})^2 \big]$ is of order~$1$, the measure $\bP_{\bn,0}^{\gb,\quen,\textit{free}}$ does not differ much from $\bP$---this has also been exploited in \cite{BL18}. We use this idea to obtain the following statement.

\begin{proposition}\label{prop:irrel}
Fix some constant $C >1$, and define
\begin{equation}\label{eq:def:ngb}
n_\gb\,:=\,\sup\big\{n\in\N\,;\,\bbE\big[(Z_{\bn,0}^{\gb,\quen, \textit{free}})^2\big]\leq C\big\}\;.
\end{equation}
Then there is some (explicit) slowly varying function $L_\cntf$ such that the critical point satisfies
\begin{equation}\label{eq:hcupperbound}
0\,\leq\, h_c(\gb)\,\leq\, L_{\arabic{svf}}(n) \, n^{-\min(\ga,1)}\;,
\end{equation}
for any $n\leq n_\gb$.\\
If $(Z_{\bn,0}^{\gb,\quen,\textit{free}})_{n\in\N}$ is bounded in $L^2(\Pb)$, then $n_{\gb}=+\infty$ (provided that $C$ had been fixed large), so $h_c(\gb)=0$; moreover there exists a slowly varying function $L_2$ such that for all $h \in (0,1)$,
\begin{equation}\label{eq:flowerbound}
\tf(\gb,h) \,\geq\, L_2(1/h) \, h^{1/\min(1,\ga)}\;.
\end{equation}
\end{proposition}

Before we prove this proposition, notice that it fully implies the non-relevance of the disorder as soon as $(Z_{\bn,0}^{\gb,\quen,\textit{free}})_{n\in\N}$ is bounded in $L^2(\Pb)$. We show in Section~\ref{subsection:secondmoment} that this holds under the assumptions of either Theorem~\ref{thm:irrel} or \ref{thm:irrel:spin} with $\gb$ small, proving both theorems. Otherwise one can use \eqref{eq:hcupperbound} and an estimate of $n_\gb$ to obtain an upper bound for the shift of the critical point, which we do in Section~\ref{sec:hcub:nonopt}.

(Note that it is not self-evident that $\bbE\big[(Z_{\bn,0}^{\gb,\quen, \textit{free}})^2\big]<\infty$ for any $n\in\N$. We actually prove it in Section~\ref{subsection:secondmoment} for $\gb<\gb_0/2$).

\begin{proof} We already stated the left inequality of \eqref{eq:hcupperbound} with Jensen's inequality in \eqref{eq:inh}. Let us prove the right inequality. Fix $h\in\R$. 
We note that the sequence $(\bbE \log Z_{\bn,h}^{\gb,\quen})_{n\in\N}$ is super-additive, so \eqref{eq:saf} and \eqref{eq:compZ} give for any $n\in\N$,
\begin{equation}\label{eq:tech:irrel:1}
\tf(\gb,h) \,=\, \sup_{n\in\N}\frac{1}{n}\bbE\log Z^{\gb,\quen}_{\bn,h} \,\geq\, \frac{1}{n}\bbE\log Z^{\gb,\quen, \textit{free}}_{\bn,h} - C_\cntc \frac{\log n}{n}\;.
\end{equation}
Moreover we have:
\begin{equation}\label{eq:tech:irrel:1bis}\begin{aligned}
\bbE\log Z^{\gb,\quen, \textit{free}}_{\bn,h} \,&\geq\, h\,\bbE\big[\partial_h \log Z_{\bn,h}^{\gb,\quen, \textit{free}}\big|_{h=0}\big] + \bbE\big[\log Z_{\bn,0}^{\gb,\quen, \textit{free}}\big]\\
&\geq\, h\,\bbE\,\bE_{\bn,0}^{\gb,\quen, \textit{free}}\Big[\sum_{\bi\in\bnsquare}\ind_{\{\bi\in\btau\}}\Big] - C_\cntc \log n ,
\end{aligned}\end{equation}
where we used the convexity of $h\mapsto \log Z_{\bn,h}^{\gb,\quen, \textit{free}}$, and the obvious bound $Z_{\bn,h}^{\gb,\quen, \textit{free}} \geq \bP(\|\btau_1\|>2n)$. Here we have to estimate the contact fraction of the renewal process under $\bP_{\bn,0}^{\gb,\quen, \textit{free}}$. We first do it under $\bP$.

\begin{lemma}\label{lem:countact} For any $\eps>0$, there exist $C_\cntc>0$ and $n_0$ such that for any $n\geq n_0$,
\begin{equation}
\bP\big(|\btau\cap\bnsquare|\geq C_{\arabic{cst}}\, f(n)\big) \,\geq\, 1-\eps\;,
\end{equation}
where $ f(n) := n^\ga/L(n)$ if $\ga\in(0,1)$ and $f(n):= n/\mu(n)$ if $\ga\geq 1$, with $\mu(n)= \bE[  \btau_1^{(1)}]$ if $\bE[  \btau_1^{(1)}]<+\infty$ and $\mu(n):= \bE[  \btau_1^{(1)} \ind_{\{\btau_1^{(1)} \leq n\}}]$ if $\bE[  \btau_1^{(1)}]=+\infty$.
\end{lemma}
\begin{proof}
When $\ga\in(0,1)$, this result is already proved in \cite[Lem. A.2]{BGK}. When $\ga \geq 1$, we have that
\begin{equation}
\frac{|\btau^{(1)} \cap [1,n] |}{f(n)} \,\to\, 1 \quad \text{ in probability} \, .
\end{equation}
Indeed, when $\bE[\btau^{(1)}_1]<+\infty$ then this is simply the law of large numbers, whereas when $\ga=1$ and $\bE[\btau^{(1)}_1]=+\infty$, this is for example in \cite{DHR79}.
Hence, provided that  $C_{\arabic{cst}}<1$, then $\bP\big(|\btau\cap\bnsquare|<  C_{\arabic{cst}} f(n) \big) \to 0$ as $n\to\infty$, which proves the lemma.
\end{proof}

Let us denote $A_n:=\{|\btau\cap\bnsquare|\geq C_{\arabic{cst}}\, f(n) \}$. We can estimate the contact fraction under $\bP_{\bn,0}^{\gb,\quen,\textit{free}}$ with the simple inequality:
\begin{equation}
\bP_{\bn,0}^{\gb,\quen,\textit{free}} (A_n^c) \,\leq\, \ind_{\{Z_{\bn,0}^{\gb,\quen,\textit{free}}\leq 1/2\}} \,+\, 2 \bE\big[\ind_{A_n^c} \, e^{\sum_{\bi\in\bnsquare}(\gb\go_\bi-\gl(\gb))\ind_{\{\bi\in\btau\}}}\big]\;,
\end{equation}
so we obtain for any $\eps>0$ and $n\geq n_0$,
\begin{equation}\begin{aligned}
\bbE \big[\bP_{\bn,0}^{\gb,\quen,\textit{free}} (A_n^c)\big] \,&\leq\, \Pb\big(Z_{\bn,0}^{\gb,\quen,\textit{free}}\leq 1/2\big) + 2 \bP(A_n^c)\\
&\leq\, 1 - \frac{1}{4\bbE\big[(Z_{\bn,0}^{\gb,\quen,\textit{free}})^2\big]} + 2\eps,
\end{aligned}\end{equation}
where we used Paley-Zygmund inequality and Lemma \ref{lem:countact}. If we choose $\eps$ small enough (more precisely $\eps<(8C)^{-1}$), and $n\leq n_\gb$, this implies
\begin{equation}
\bbE\bE_{\bn,0}^{\gb,\quen, \textit{free}}\big[\,|\btau\cap\bnsquare|\,\big] \,\geq\, C_{\arabic{cst}} \, f(n) \,\bbE \big[\bP_{\bn,0}^{\gb,\quen,\textit{free}} (A_n)\big] \,\geq\, C_\cntc \, f(n)\;,
\end{equation}
(where we get rid of the condition $n\geq n_0$ by adjusting $C_{\arabic{cst}}$ for finitely many terms). Going back to \eqref{eq:tech:irrel:1} and \eqref{eq:tech:irrel:1bis}, we finally obtain the following lower bound for any $n\leq n_\gb$:
\begin{equation}\label{eq:tech:irrel:2}
\tf(\gb,h)\,\geq\, \frac{ C_\cntc }{n} \big( h f(n)  - C_\cntc  \log n \big) \;,
\end{equation}
where we recall $f(n) := n/\mu(n)$ if $\ga\geq 1$ (with $\mu(n)$ either a constant, or going to $+\infty$ as a slowly varying function in the case $\ga=1$, $\bE[  \btau_1^{(1)}]=+\infty$.), and $f(n) := n^\ga/L(n)$ if $\ga\in(0,1)$. If we take $h=h_c(\gb)$ in \eqref{eq:tech:irrel:2},  then $\tf(\gb,h_c(\gb))=0$, so we get
\begin{equation}
h_c(\gb)\, \leq\, C_\cntc \frac{\log n}{f(n)}\;,
\end{equation}
which concludes the proof of \eqref{eq:hcupperbound}, with $L_{\arabic{svf}}(n)=C_\arabic{cst}\, L(n)\log(n)$ if $\ga\in(0,1)$; $L_{\arabic{svf}}(n)=C_\arabic{cst} \mu(n) \log n$ if $\ga \geq 1$.

\smallskip
If $(Z_{\bn,0}^{\gb,\quen,\textit{free}})_{n\in\N}$ is bounded in $L^2(\Pb)$, then we can choose $C$ such that $n_\gb=\infty$, so \eqref{eq:hcupperbound} holds for any $n\in\N$ and we obviously have $h_c(\gb)=0$.
Furthermore \eqref{eq:tech:irrel:2} also holds for any $n\in\N$, so if we take $h>0$ and $n=C_\cntc\, V_\ga(1/h)$ where $V_\ga$ is the asymptotical inverse of $b\mapsto b^{\min(\ga,1)} L_{\arabic{svf}}(b)^{-1}$ as $b\to\infty$ and $C_{\arabic{cst}}$ is a suitable constant, we finally obtain \eqref{eq:flowerbound} from \eqref{eq:tech:irrel:2}.
\end{proof}

\subsection{Second moment estimates}
\label{subsection:secondmoment}
With regard to Proposition \ref{prop:irrel}, it suffices to estimate $\bbE\big[(Z^{\gb,\quen,\textit{free}}_{\bn,0})^2\big]$ to prove disorder irrelevance, or at least an upper bound on the shift of the critical point. To do so we introduce two independent copies $\btau$, $\btau'$ of a renewal process with law $\bP$.

\begin{proposition}\label{prop:L2bound} 
For any $\gb\in[0,\gb_0/2)$ one has $Z^{\gb,\quen,\textit{free}}_{\bn,0} \in L^2(\Pb)$ for all $n\in\N$ and:
\begin{align}\label{eq:proof:L2bound:6}
\limsup_{n\rightarrow\infty}\bbE\Big[\big(Z^{\gb,\quen,\textit{free}}_{\bn,0}\big)^2\Big] 
&\;\leq\; \bE_{(\btau,\btau')}\Big[e^{\frac{3}{2}\big(\gl(2\gb)-2\gl(\gb)\big)\big(|\btau^{(1)}\cap\btau'^{(1)}|+|\btau^{(2)}\cap\btau'^{(2)}|\big)}\Big]\;,
\intertext{
When $\Pb=\Pbspinx$ for some $x>0$, one has for any $\gb\in\R_+$:
}\label{eq:proof:L2bound:8}
\lim_{n\rightarrow\infty}\bbEspinx\Big[\big(Z^{\gb,\quen,\textit{free}}_{\bn,0}\big)^2\Big] &\;=\; \bE_{(\btau,\btau')}\Big[e^{\big(\gl(2\gb)-2\gl(\gb)\big)|\btau\cap\btau'|}\Big]\;.
\end{align}
\end{proposition}

Note that this proposition also applies for any sequence of indices $\bn\in\N^2$ such that $\bn^{(1)}, \bn^{(2)}\to\infty$. Plugging those estimates into Proposition \ref{prop:irrel}, this proves Theorems~\ref{thm:irrel} and \ref{thm:irrel:spin}.

\begin{proof}[Proof of Theorems~\ref{thm:irrel}-\ref{thm:irrel:spin}]
In the general case we use the upper bound \eqref{eq:proof:L2bound:6}. If $\btau^{(1)}\cap\btau'^{(1)}$ is terminating, then $|\btau^{(1)}\cap\btau'^{(1)}|$ follows a geometric distribution (recall that $\btau^{(1)}$ and $\btau'^{(1)}$ are independent univariate renewal processes), so it has some finite exponential moment. The same holds for $\btau^{(2)}\cap\btau'^{(2)}$, and for their sum. Because $\gl(2\gb)-2\gl(\gb)\sim c\gb^2$ for some $c>0$ when $\gb\rightarrow0$, there exists $\gb_1$ such that the right hand side of \eqref{eq:proof:L2bound:6} is finite for any $\gb\in[0,\gb_1)$, so $(Z^{\gb,\quen,\textit{free}}_{\bn,0})_{n\in\N}$ is bounded in $L^2(\bbP)$. Applying Proposition~\ref{prop:irrel}, this proves Theorem~\ref{thm:irrel}.

When $\Pb=\Pbspinx$ for some $x>0$, the same argument applies with \eqref{eq:proof:L2bound:8} and $\btau\cap\btau'$ ---if $\btau\cap\btau'$ is terminating then $|\btau\cap\btau'|$ also follows a geometric distribution--- so this proves Theorem~\ref{thm:irrel:spin}.
\end{proof}

Before proving Proposition~\ref{prop:L2bound}, we need to introduce some new notations. Using a replica trick and Fubini-Tonelli theorem, we can write the second moment of the partition function as
\begin{equation}\label{eq:proof:L2bound:1}
\bbE\Big[\big(Z^{\gb,\quen,\textit{free}}_{\bn,0}\big)^2\Big] \,=\, \bE_{(\btau,\btau')}\bbE\Big[\exp\Big(\sum_{\bi\in\bnsquare}(\gb\go_\bi-\lambda(\gb))(\ind_{\{\bi\in\btau\}}+\ind_{\{\bi\in\btau'\}})\Big)\Big]\;.
\end{equation}
For any trajectories $\btau,\btau'\subset\N^2$, $\btau$ can be written as a sequence $\btau=\{\btau_1,\btau_2,\ldots\}$ with $\btau^{(r)}_k<\btau^{(r)}_{k+1}$ for every $r\in\{1,2\}, k\in\N$, and the same holds for $\btau'$. One can rewrite \eqref{eq:proof:L2bound:1} by taking the sum on $\bi\in(\btau\cup\btau')\cap\bnsquare$ (other terms are 0). We claim that there are three kind of points in $\btau\cup\btau'$ contributing to this sum:
\begin{itemize}
\item $\bi\in\btau\cap\btau'$. We will call such $\bi$ \emph{double points} of $\btau\cup\btau'$.
\item $\bi\in\btau\cup\btau'\setminus\btau\cap\btau'$, and $\bi^{(1)}\notin\btau^{(1)}\cap\btau'^{(1)}$ and $\bi^{(2)}\notin\btau^{(2)}\cap\btau'^{(2)}$: that is, $\bi$ is in either $\btau$ or $\btau'$, and no other point from $\btau\cup\btau'$ is on the same line or column than $\bi$. We will call those \emph{isolated points} of $\btau\cup\btau'$.
\item $\bi\in\btau\cup\btau'\setminus\btau\cap\btau'$, and $\bi^{(1)}\in\btau^{(1)}\cap\btau'^{(1)}$ or $\bi^{(2)}\in\btau^{(2)}\cap\btau'^{(2)}$. We will call those \emph{chained points} of $\btau\cup\btau'$.
\end{itemize}
Let us explain the denomination \emph{chained points}. Let $\bi_1\in\btau\cup\btau'$ be the first chained point for the lexical order on $\N^2$. Assume $\bi_1\in\btau\setminus\btau'$ without loss of generality. Then there exists $\bi_2\in\btau'\setminus\btau$ such that $\bi_1\leftrightarrow\bi_2$ and $\bi_1^{(r)}=\bi_2^{(r)}, \bi_1^{(3-r)}<\bi_2^{(3-r)}$ for one $r\in\{1,2\}$ (because $\bi_1$ comes first in the lexical order). Moreover there is no other point $\bi\in\btau\cup\btau', \bi\neq\bi_2$ such that $\bi\leftrightarrow\bi_1, \bi\neq\bi_1$: indeed, it would imply $\bi\in\btau'\setminus\btau$, which is impossible since $\btau'$ is strictly increasing on each coordinate.

Let us now assume that there exists another $\bi_3\in\btau\cup\btau', \bi_3\notin\{\bi_1,\bi_2\}$ such that $\bi_3\leftrightarrow\bi_2$. Then we obviously have $\bi_3\in\btau\setminus\btau'$ (because $\bi_2\in\btau'$ and the sequence $\btau'$ has strictly increasing coordinates), and $\bi_2^{(3-r)}=\bi_3^{(3-r)}, \bi_2^{(r)}<\bi_3^{(r)}$ with $r\in\{1,2\}$ given above, since $\btau$ has strictly increasing coordinates. Moreover, this $\bi_3$ is unique (if it exists), because of the same argument that proved $\bi_2$ is unique.

By repeating this process until there is no more $\bi\in\btau\cup\btau'$ satisfying $\bi\leftrightarrow\bi_k, \bi\notin\{\bi_1,\ldots,\bi_k\}$, we define a sequence $(\bi_1,\bi_2,\ldots,\bi_k)$ of points in $\btau\cup\btau'$ with $k\geq2$, such that:
\begin{itemize}
\item $\bi_1, \bi_3, \ldots \in \btau\setminus\btau'$ and $\bi_2, \bi_4, \ldots \in \btau'\setminus\btau$ (or the other way arround).
\item There is $r\in\{1,2\}$ such that for any $1\leq l\leq k-1$, one has $\bi_l^{(r)}=\bi_{l+1}^{(r)}, \bi_l^{(3-r)}<\bi_{l+1}^{(3-r)}$ if $l$ is odd, and $\bi_l^{(3-r)}=\bi_{l+1}^{(3-r)}, \bi_l^{(r)}<\bi_{l+1}^{(r)}$ if $l$ is even. In particular, $\bi_l\leftrightarrow\bi_{l+1}$ for any $1\leq l\leq k-1$.
\item There is no $\bi\in\btau\cup\btau'$, $\bi\notin\{\bi_1,\ldots, \bi_k\}$ such that $\bi\leftrightarrow\bi_l$ for some $1\leq l \leq k$.
\end{itemize}
We call $\bgs_1=(\bi_1,\ldots,\bi_k)$ a \emph{chain of points} in $\btau\cup\btau'$. Note that $\bi_1,\ldots,\bi_k$ are all \emph{chained points} as defined previously. Note also that this construction may lead to an infinite chain in $\N^2$, but is always finite if we restrict $\btau\cup\btau'$ to $\bnsquare$.

Furthermore, if we apply the same construction process to $\btau\cup\btau'\setminus\bgs_1$, we can define another chain of points $\bgs_2$ satisfying the same properties ; and by repeating it again, we obtain a sequence $(\bgs_1, \bgs_2,\ldots)$ of chains of points in $\btau\cup\btau'$ satisfying the same properties (once again this sequence is finite for $\btau\cup\btau'$ restricted to $\bnsquare$, and may be infinite in $\N^2$). Moreover, this sequence covers all chained points (indeed, any chained point $\bi\in\btau\cup\btau'$ is only preceded by a finite number of points in $\btau\cup\btau'$ for the lexical order, so the construction process always reaches $\bi$).

\begin{figure}[h]
 \centering
  \includegraphics[width=0.5\linewidth]{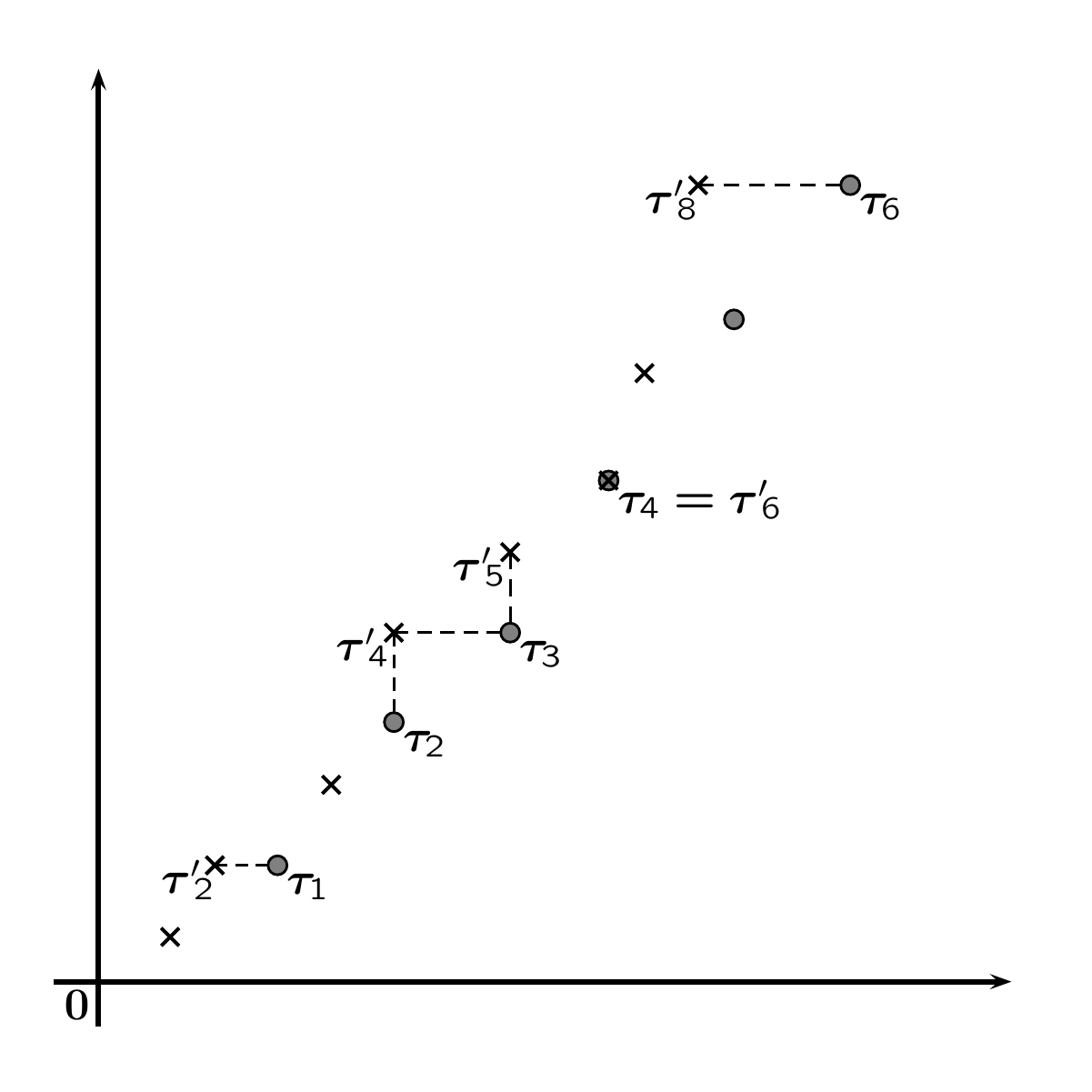}
  \caption{\footnotesize Representation of the union of two renewal sets $\btau,\btau'\subset\N^2$. It has one double point $\btau_4=\btau'_6$, several isolated points ($\btau'_1,\btau'_3,\btau'_7$ and $\btau_5$ in lexical order), and three chains of points $(\btau'_2,\btau_1)$, $(\btau_2,\btau'_4,\btau_3,\btau'_5)$ and $(\btau'_8,\btau_6)$.}
  \label{fig:Chains}
\end{figure}

\smallskip
Using this construction, we can partition $\btau\cup\btau'$ as stated in the following proposition.
\begin{proposition}\label{prop:decomp}
Let $\btau$ and $\btau'$ be two subsets of $\N^2$ such that they can be written as sequences $\btau=\{\btau_1,\btau_2,\ldots\}$ with $\btau^{(r)}_k<\btau^{(r)}_{k+1}$ for every $r\in\{1,2\}, k\in\N$ (same for $\btau'$). Then the subset $\btau\cup\btau'$ can be partitioned like this:
\begin{equation}\label{eq:prop:decomp}
\btau\cup\btau' \,=\, \bgn\cup\bgr\cup\Big(\bigcup_{m\in\N}\bgs_m\Big)\;,
\end{equation}
where $\bgn:=\btau\cap\btau'$, $\bgr$ is the set of isolated points of $\btau\cup\btau'$, $\bfrakS:=\bigcup_{m\in\N}\bgs_m$ is the set of chained points of $\btau\cup\btau'$, and $\bgs_m, m\in\N$ are chains of points. All the sets $\bgn, \bgr$ and $\bgs_m, m\in\N$ are separated.\\
Moreover, if $\bi, \bj\in\btau\cup\btau', \bi\neq\bj$ and $\bi\leftrightarrow\bj$, then there exists $m\in\N$ such that $\bi, \bj\in\bgs_m$.
\end{proposition}
The decomposition in \eqref{eq:prop:decomp} comes from the construction above. The latter claim is obvious: if $\bi\leftrightarrow\bj$ and $\bi\neq\bj$, then neither $\bi$ or $\bj$ can be a double point or isolated. So they are chained points, and they must be in the same chain (because of the last property of chains). This claim also ensures that disorder values $(\go_\bi)_{\bi\in\btau\cup\btau'}$ are independent from one set of the partition to another. We now have all the needed tools to prove Proposition~\ref{prop:L2bound}.

\begin{proof}[Proof of Proposition~\ref{prop:L2bound}]
We partition $(\btau\cup\btau')\cap\bnsquare$ into $\bnu_n, \bgr_n, \bgs_{n,m},m\in\N$ using Proposition~\ref{prop:decomp}, and we use it to separate the right hand side of \eqref{eq:proof:L2bound:1} in independent products:
\begin{equation}\begin{aligned}\label{eq:proof:L2bound:3}
&\bbE\Big[\big(Z^{\gb,\quen,\textit{free}}_{\bn,0}\big)^2\Big] \,=\, \bE_{(\btau,\btau')}\bbE\bigg[\prod_{\bi\in\bnu_n} e^{2(\gb\go_\bi-\gl(\gb))} \prod_{\bi\in\bgr_n}e^{\gb\go_\bi-\gl(\gb)} \prod_{m\in\N}\prod_{\bi\in\bgs_{n,m}}e^{\gb\go_\bi-\gl(\gb)}\bigg]\\
&\qquad=\, \bE_{(\btau,\btau')}\bigg[\bbE\Big[\prod_{\bi\in\bnu_n} e^{2(\gb\go_\bi-\gl(\gb))}\Big] \bbE\Big[\prod_{\bi\in\bgr_n}e^{\gb\go_\bi-\gl(\gb)}\Big] \prod_{m\in\N}\bbE\Big[\prod_{\bi\in\bgs_{n,m}}e^{\gb\go_\bi-\gl(\gb)}\Big]\bigg]\,.
\end{aligned}\end{equation}
We used that $\ind_{\{\bi\in\btau\}}+\ind_{\{\bi\in\btau'\}}=2$ if and only if $\bi\in\btau\cap\btau'$, and $\ind_{\{\bi\in\btau\}}+\ind_{\{\bi\in\btau'\}}=1$ for any other $\bi\in\btau\cup\btau'\setminus\btau\cap\btau'$. Note that an isolated point in $(\btau\cup\btau')\cap\bnsquare$ can be chained for a higher $n$, thus we write the dependance on $n$ of the decomposition in \eqref{eq:proof:L2bound:3} with a subscript.

\smallskip
Now we have to compute those expectations. Note that $(\go_\bi)_{\bi\in\bgn_n}$ is an independent family, and so is $(\go_\bi)_{\bi\in\bgr_n}$. Thus $\bbE\big[\prod_{\bi\in\bgr_n}e^{\gb\go_\bi-\gl(\gb)}\big]=\prod_{\bi\in\bgr_n} \bbE[e^{\gb\go_\bi-\gl(\gb)}]=1$ (recall that $e^{\gl(\gb)}=\bbE[e^{\gb\go_\bone}]$ for any $\gb\in[0,\gb_0)$), and $\bbE\big[\prod_{\bi\in\bgn_n}e^{2(\gb\go_\bi-\gl(\gb))}\big]=e^{|\bgn_n|(\gl(2\gb)-2\gl(\gb))}$. 

Let $\bgs=\{\bi_1,\bi_2,\ldots,\bi_k\}$, $k\in\N$ be a (finite) chain of $(\btau\cup\btau')\cap\bnsquare$ as defined previously, and let us estimate $\bbE\big[\prod_{\bi\in\bgs}e^{\gb\go_\bi-\gl(\gb)}\big]$.

\smallskip
\paragraph{\textit{General case}}
This expectation cannot be computed in the general case, but we can obtain an upper bound with Cauchy-Schwarz inequality. In Remark~\ref{rem:bornesup} we will discus why we think this upper bound is not sharp and can be improved, but it is fully sufficient to prove disorder irrelevance.

Let us note $\bgs^{\textit{odd}}=\{\bi_1,\bi_3,\ldots\}$ and $\bgs^{\textit{even}}=\{\bi_2,\bi_4,\ldots\}$. Notice that $\bgs^{\textit{odd}}\subset\btau$ and $\bgs^{\textit{even}}\subset\btau'$ (or the other way around), and recall that $\btau$ and $\btau'$ are strictly increasing on each coordinate. Hence $(\go_\bi)_{\bi\in\bgs^{\textit{even}}}$ is a family of independent variables, and so is $(\go_\bi)_{\bi\in\bgs^{\textit{odd}}}$.
Applying Cauchy-Schwarz inequality, we obtain
\begin{equation}\label{eq:proof:L2bound:4zero}\begin{aligned}
\bbE\Big[\prod_{\bi\in\bgs}e^{\gb\go_\bi-\gl(\gb)}\Big] \,&\leq\, \bbE\Big[\prod_{\bi\in\bgs^{\textit{odd}}}e^{2(\gb\go_\bi-\gl(\gb))}\Big]^{\frac{1}{2}}  \bbE\Big[\prod_{\bi\in\bgs^{\textit{even}}}e^{2(\gb\go_\bi-\gl(\gb))}\Big]^{\frac{1}{2}}\\
&\leq\, \prod_{\bi\in\bgs^{\textit{odd}}}\bbE\big[e^{2(\gb\go_\bi-\gl(\gb))}\big]^{\frac{1}{2}}  \prod_{\bi\in\bgs^{\textit{even}}}\bbE\big[e^{2(\gb\go_\bi-\gl(\gb))}\big]^{\frac{1}{2}}\\
&\leq\, e^{(\gl(2\gb)-2\gl(\gb)) \frac{|\bgs^{\textit{odd}}|+|\bgs^{\textit{even}}|}{2}}\,=\, e^{(\gl(2\gb)-2\gl(\gb))\frac{|\bgs|}{2}}\,,
\end{aligned}\end{equation}
where we recall the definition of $\gl$ (see \eqref{def:lambda}), and that those expectations are finite if $\gb<\gb_0/2$. Therefore we obtain the following upper bound on \eqref{eq:proof:L2bound:3}:
\begin{equation}\label{eq:proof:L2bound:4}
\bbE\Big[\big(Z^{\gb,\quen,\textit{free}}_{\bn,0}\big)^2\Big] \,\leq\, \bE_{(\btau,\btau')}\Big[e^{\big(\gl(2\gb)-2\gl(\gb)\big)\big(|\bgn_n|+\frac{|\bfrakS_n|}{2}\big)}\Big]\,.
\end{equation}
(Recall that $\bfrakS_n:=\bigcup_{m\in\N}\bgs_{n,m}$). In particular this proves that $Z^{\gb,\quen,\textit{free}}_{\bn,0}\in L^2(\Pb)$ for any $n\in\N$ and $\gb<\gb_0/2$.

\smallskip
Let us note $\btau_{\preceq \bn}:=\btau\cap\bnsquare$ (resp. $\btau'_{\preceq \bn}:=\btau'\cap\bnsquare$). Using the decomposition of $\btau_{\preceq \bn}\cup\btau'_{\preceq \bn}$ from Proposition~\ref{prop:decomp}, one can prove by induction:
\begin{equation}
\label{eq:comparechain}
\frac{1}{2}\,(|\btau^{(1)}_{\preceq \bn}\cap\btau'^{(1)}_{\preceq \bn}|+|\btau_{\preceq \bn}^{(2)}\cap\btau'^{(2)}_{\preceq \bn}|) \,\leq\, |\bgn_n|+|\bfrakS_n| \,\leq\, 2\,(|\btau^{(1)}_{\preceq \bn}\cap\btau'^{(1)}_{\preceq \bn}|+|\btau_{\preceq \bn}^{(2)}\cap\btau'^{(2)}_{\preceq \bn}|)\,. 
\end{equation}
(The constants are optimal: the left hand size is an equality if $|\bgn_n|>0, \bfrakS_n=\emptyset$, and the right hand size is an equality if $\bgn_n=\emptyset$ and $\bfrakS_n$ only contains chains of length 2). So \eqref{eq:proof:L2bound:4} becomes
\begin{equation}\begin{aligned}\label{eq:proof:L2bound:5}
\bbE\Big[\big(Z^{\gb,\quen,\textit{free}}_{\bn,0}\big)^2\Big] \,&\leq\, \bE_{(\btau,\btau')}\Big[e^{\big(\gl(2\gb)-2\gl(\gb)\big)\big(\frac{|\bgn_n|}{2}+|\btau^{(1)}_{\preceq \bn}\cap\btau'^{(1)}_{\preceq \bn}|+|\btau_{\preceq \bn}^{(2)}\cap\btau'^{(2)}_{\preceq \bn}|\big)}\Big]\\
&\leq\, \bE_{(\btau,\btau')}\Big[e^{\frac{3}{2}\big(\gl(2\gb)-2\gl(\gb)\big)\big(|\btau^{(1)}_{\preceq \bn}\cap\btau'^{(1)}_{\preceq \bn}|+|\btau_{\preceq \bn}^{(2)}\cap\btau'^{(2)}_{\preceq \bn}|\big)}\Big]\,,
\end{aligned}\end{equation}
where we used $|\bgn_n|\leq|\btau_{\preceq \bn}^{(1)}\cap\btau_{\preceq \bn}'^{(1)}|+|\btau_{\preceq \bn}^{(2)}\cap\btau_{\preceq \bn}'^{(2)}|$, and $\gl(2\gb)-2\gl(\gb)\geq0$ for any $\gb\geq0$. Moreover $|\btau^{(1)}_{\preceq \bn}\cap\btau'^{(1)}_{\preceq \bn}|+|\btau_{\preceq \bn}^{(2)}\cap\btau'^{(2)}_{\preceq \bn}|$ is non-decreasing in $\bn$, and converges to $|\btau^{(1)}\cap\btau'^{(1)}|+|\btau^{(2)}\cap\btau'^{(2)}|$. 
Using the monotone convergence theorem, this finishes the proof of \eqref{eq:proof:L2bound:6}.

\smallskip
\paragraph{\textit{Case $\go\sim\Pbspinx$, $x>0$}}
In that case we can compute $\bbEspinx\big[\prod_{\bi\in\bgs}e^{\gb\go_\bi-\gl(\gb)}\big]$ exactly for any finite chain $\bgs=\{\bi_1,\bi_2,\ldots,\bi_k\}$, $k\in\N$. Indeed, we can write without loss of generality
\begin{equation}
\go_{\bi_1} = \hgo_1\bgo_1,\; \go_{\bi_2} = \hgo_1\bgo_2,\; \go_{\bi_3} = \hgo_2\bgo_2,\;\ldots
\end{equation}
and so on, 
with $(\hgo_1, \bgo_1,\hgo_2,\ldots)$ a family of i.i.d.\ random variables. Then we can compute
\begin{align*}
\bbEspinx\Big[\prod_{\bi\in\bgs}e^{\gb\go_\bi-\gl(\gb)}\Big] \,&=\, e^{-k\gl(\gb)}\bbEspinx\Big[\bbEspinx\big[e^{\gb \hgo_1 \bgo_1}\big|(\hgo_j)_{j\geq1}, (\bgo_j)_{j\geq2}\big]\prod_{l=2}^ke^{\gb\go_{\bi_l}}\Big] \\
&=\, e^{-k\gl(\gb)}\bbEspinx\Big[\frac{e^{\gb x \hgo_1}+e^{-\gb x \hgo_1}}{2}\prod_{l=2}^ke^{\gb\go_{\bi_l}}\Big] = e^{-(k-1)\gl(\gb)}\bbEspinx\Big[\prod_{l=2}^ke^{\gb\go_{\bi_l}}\Big]\,,
\end{align*}
where we used that $\tfrac12(e^{\gb x \hgo_1} + e^{-\gb x \hgo_1}) = \cosh(\gb x \hgo_1) = \cosh(\gb x^2) = e^{\gl(\gb)}$ (because $|\hgo_1|=x$ $\bbPspinx$-a.s. and $\cosh$ is symmetric). By iteration, we finally obtain for any finite chain $\bgs$ and $\gb\in\R_+$,
\begin{equation}\label{eq:proof:L2bound:7}
\bbEspinx\Big[\prod_{\bi\in\bgs}e^{\gb\go_\bi-\gl(\gb)}\Big] \,=\, 1\;.
\end{equation}
Recalling \eqref{eq:proof:L2bound:3}, we get the exact formula:
\begin{equation}\label{eq:proof:L2bound:8zero}
\bbEspinx\Big[\big(Z^{\gb,\quen,\textit{free}}_{\bn,0}\big)^2\Big] \,=\, \bE_{(\btau,\btau')}\Big[e^{\big(\gl(2\gb)-2\gl(\gb)\big)|\bgn_n|}\Big]\,.
\end{equation}
And we obtain \eqref{eq:proof:L2bound:8} by applying the monotone convergence theorem.
\end{proof}

Notice that the particular behavior of distributions $\Pbspinx$ relies solely on \eqref{eq:proof:L2bound:7}. When $\bgs$ is a chain of length two, it corresponds to computing the correlations of the field $(e^{\gb\go_\bi})_{\bi\in\N^2}$, as discussed in Section~\ref{subsection:comments}.


\subsection{Upper bounds for the shift of the critical point}\label{sec:hcub:nonopt}
When we cannot bound the second moment of the partition function for all $n\in\N$, we can still estimate $n_\gb$ with the finite volume equivalent of Proposition~\ref{prop:L2bound} (that is inequality~\eqref{eq:proof:L2bound:5} when $\Pb\neq\Pbspinx$ for all $x>0$, and identity~\eqref{eq:proof:L2bound:8zero} when $\Pb=\Pbspinx$ for some $x>0$): thanks to Proposition~\ref{prop:irrel}, we are able to obtain an upper bound for the shift of the critical point.

\subsubsection{Case $\Pb\neq\Pbspinx$ for all $x>0$: Proposition~\ref{prop:hcub:nonopt}}
Let us consider \eqref{eq:proof:L2bound:5}: we can split the two intersections of the projections with a Cauchy-Schwarz inequality
\begin{equation}\label{eq:hcubnonopt:cs}
\bbE\Big[\big(Z^{\gb,\quen,\textit{free}}_{\bn,0}\big)^2\Big] 
\,\leq\, \bE_{(\btau,\btau')}\Big[e^{3(\gl(2\gb)-2\gl(\gb))|\btau^{(1)}_{\preceq \bn}\cap\btau'^{(1)}_{\preceq \bn}|}\Big]\,,
\end{equation}
where we also used $\bn=(n,n)$ and that both projections have the same law. Recall that $\btau^{(1)}, \btau'^{(1)}$ are independent univariate renewal processes with inter-arrival distribution $\bP(\btau^{(r)}=a)=\tilde L(a)a^{-(1+\ga)}$, and recall that $\gl(2\gb)-2\gl(\gb)\sim c\gb^2$ as $\gb\searrow0$ for some $c>0$. In particular this upper bound has already been studied in \cite{A08, Ton08} for $\ga\in(1/2,1)$, and gives the estimate \stepcounter{svf}$n_\gb\geq L_{\cntf}(1/\gb) \gb^{\frac{-2}{2\ga-1}}$ (we do not write the details here).
{\blue
In the case $\ga \ge 1$, one easily gets from~\eqref{eq:hcubnonopt:cs} that the second moment is bounded by $\exp( c \gb^2 n)$, giving the estimate $n_\gb\geq  c' \gb^{-2}$. Recollecting \eqref{eq:hcupperbound}, this fully proves Proposition~\ref{prop:hcub:nonopt} (notice that when $m_1=0$ and $\ga\geq1$, this does not lead to a better bound for the critical point than the trivial one $h_c(\gb)\leq\gl(\gb)=\gb m_1^2+\gb^2m_2^2/2 (1+o(1))$).
}

\begin{remark}\rm
\label{rem:bornesup}
{\blue As far as Theorem~\ref{thm:rel:shift} is concerned, Proposition~\ref{prop:hcub:nonopt} yields a satisfactory upper bound when $m_1\neq 0$, but not when $m_1=0$, $m_4>m_2^2$. We strongly believe that our lower bound is of the right order (\textit{i.e.} $\gb^{\max(\frac{4\ga}{2\ga-1},4)}$ when $m_1=0$, $m_4>m_2^2$), and that our upper bound is too rough when we estimate the second moment ---even though it is sufficient to prove disorder irrelevance and  give a satisfactory upper bound when $m_1\neq0$.} Furthermore, a work in progress with Q. Berger \cite{BL20} is leading us to the following claim:\smallskip

\begin{displayquote}\emph{
\hspace{-2mm}Suppose $\ga\in(1/2,1)$, $m_1=0$ and $m_4>m_2^2$. Let $\gb>0$ and $\bm_\gb:=(m_\gb,m_{\gb})$ with $m_\gb\sim \gb^{\frac{-4}{2\ga-1}}L(1/\gb)^{\frac{2}{2\ga-1}}$ as $\gb\searrow0$. Then $Z^{\gb,\quen,\textit{free}}_{\bm_\gb,0}$ converges in distribution to some non-trivial random variable ${\bf{Z}}^{\textit{free}}$ as $\gb\searrow0$. \vspace{0.8mm}
This convergence also holds in $L^2$ under some appropriate coupling.}
\end{displayquote}\smallskip

In particular this claim implies that $n_\gb\sim L_{\cntf}(1/\gb) \gb^{-4/(2\ga-1)}$ when $\ga\in(1/2,1)$: once this result is proven, it will directly give an upper bound $h_c(\gb)\leq L_\cntf(1/\gb)\gb^{4\ga/(2\ga-1)}$, which is fully satisfactory with regards to Theorem~\ref{thm:rel:shift} in the case $m_1=0$, $m_4>m_2^2$.

Moreover, when $\hgo,\bgo$ are two sequences of centered, i.i.d. Gaussian variables, one can fully compute the contribution of chains of points to the second moment of the partition function, which eventually leads to (for $\gb$ sufficiently small)
\begin{equation}
\label{momentGauss}
\bbE\Big[\big(Z^{\gb,\quen,\textit{free}}_{\bn,0}\big)^2\Big] 
\,\leq\,  2 \, \bE_{(\btau,\btau')}\Big[e^{C_\cntc\gb^4 |\btau^{(1)}_{\preceq \bn}\cap\btau'^{(1)}_{\preceq \bn}|}\Big]\,.
\end{equation}
This is the same as \eqref{eq:hcubnonopt:cs} with an exponent of order $\gb^4$ instead of $\gl(2\gb)-2\gl(\gb)\sim c\gb^2$. In particular it gives the expected upper bound $h_c(\gb)\leq L_\cntf(1/\gb)\gb^{\max(\frac{4\ga}{2\ga-1},4)}$ on the shift of the critical point, supporting again the idea that Theorem~\ref{thm:rel:shift} gives the correct order of the lower bound. We carry out the full computation in Appendix~\ref{section:appendixC}.
\end{remark}

\subsubsection{Case $\Pb=\Pbspinx$ for some $x>0$} 
In that case we have an exact computation of the second moment \eqref{eq:proof:L2bound:8zero}, where the right hand side is the same as in the proof of \cite[Prop. 3.3]{BGK} for the gPS model with i.i.d.\ disorder. Thus we obtain the same estimate as in \cite{BGK}, that is $n_\gb\sim L_\cntf(1/\gb) \gb^{-\max(\frac{2\ga}{\ga-1},4)}$ for any $\ga>1$ (see the proof of \cite[Prop. 3.3]{BGK} for the details). Therefore Proposition~\ref{prop:irrel} immediately gives an upper bound on $h_c(\gb)$ of order $L_\cntf(1/\gb)\gb^{\max(\frac{2\ga}{\ga-1},4)}$ which proves the right inequality in Theorem~\ref{thm:rel:shift:m4=1}.

\section{Disorder relevance: smoothing of the phase transition when $\Pb\neq\Pbspinx$}\label{section:smoothing}

The inequality of Theorem~\ref{thm:rel:smoothing} is proven via a rare-stretch strategy, as done in  \cite{GTsmooth} (or more recently \cite{CdH}). 
We introduce some notations that we  use in this section and the next one to lighten upcoming formulae: we write $Z^{\quen}_{\bn}:=Z_{\bn,h}^{\gb,\quen}$ for the (constrained) partition function with quenched disorder (the choice of parameters $\gb$ and $h$ will always be explicit), and $Z_{\bn,h}:=Z_{\bn,h}^{0,q} $ for the homogeneous (or annealed) partition function. Moreover we will denote by $Z^{\quen}_{\ba,\bb}$ the partition function conditioned to start from $\ba$ and constrained to end in $\bb$. More precisely for any $\bzero\preceq\ba\prec\bb$,
\begin{equation}\label{eq:defZcond}
Z^{\quen}_{\ba, \bb} \,:=\, \bE\Big[\exp\Big(\sum_{\bi \in \llbracket \ba+\bone , \bb \rrbracket } (\gb\go_{\bi} -\gl(\gb)+ h)\ind_{\{\bi\in\btau\}}\Big)\ind_{\{\bb\in\btau\}} \Big| \ba \in\btau \Big]\,,
\end{equation}
and $Z^{\quen}_{\ba, \bb}:=0$ if $\ba\nprec\bb$. Note that $Z^{\quen}_{\ba, \bb}$ has same law as $Z^{\quen}_{\bb-\ba}$, and for any $\ba\preceq\bb\preceq\bc\preceq\bd$, $Z^{\quen}_{\ba, \bb}$ and $Z^{\quen}_{\bc, \bd}$ are independent.

\subsection{Rare-stretch strategy}

The rare-stretch strategy consists in obtaining a lower bound on the partition function by considering the contribution of only one type of trajectories, which target favorable (but sparse) regions in the environment.

Fix $\gb>0$ and $h\in\R$, and let $\A_l\subset\R^{l^2}, l\in\N$ be a sequence of Borel sets. We will write $\bl:=(l,l)$,
 and assume that there is $\cG\geq0$, $\cC\geq0$ such that for any $l\in\N$ (or at least infinitely many $l\in\N$):
 \begin{equation}
 \label{def:gaincost}
\begin{aligned}
&\bullet\ \frac{1}{l} \log Z_{\bl}^{\quen} \geq \cG\,, \quad \text{ for any $\go$ such that $\go_{\boldsquare[\bl]}:=(\go_\bi)_{\bi\in\bsquare[\bl]}\in\A_l$, }\\
&\bullet\ \frac{1}{l} \log \bbP(\go_{\bsquare[\bl]}\in\A_l)\geq-\C\,.
\end{aligned}  
 \end{equation}
Here $\G$ stands for \emph{gain}, and $\cC$ for \emph{cost}. 

\begin{lemma}\label{lem:gain-cost}
For any $\gb>0$ and $h\in\R$, if $\G$, $\C$ are defined as in \eqref{def:gaincost} (for some sequence of $(\A_l)_{l\in \bbN}$), then the following holds:
\begin{equation}\label{eq:lem:gain-cost}
\G-(2+\ga)\C \,>\,0 \quad \Longrightarrow \quad \tf(\gb,h)\,>\,0\;.
\end{equation}
\end{lemma}
\begin{proof}
We replicate here the proof of \cite{CdH}, but with our disorder indexed by $\N^2$. Fix $l$ such that the above conditions hold, and let $T_i(\go), i\in\N$ be the indices of blocks of size $l\times l$ on the diagonal satisfying the event $\A_l$. More precisely,  let $T_0(\go):=0$, and
\begin{equation}
T_i(\go)\,:=\,\inf\{n>T_{i-1}(\go)\;;\, \go_{\bsquaretwo{(n-1)\bl+\bone}{n\bl}} \in\A_l \}\,,
\end{equation}
for any $i\geq 1$. Note that $\big(T_{i+1}(\go)-T_i(\go)\big)_{i\in\N}$ is an i.i.d.\ sequence with law \emph{Geom}$\big(\Pb(\A_l)\big)$, so $\bbE[T_1]=\Pb(\A_l)^{-1}\leq e^{\C l}$. We can give a lower bound of the partition function on the block $T_k\bl$,
\begin{equation}
Z_{T_k\bl}^{\quen} \,\geq\, \prod_{i=1}^k Z_{T_{i-1}\bl,\,(T_i-1)\bl}^{\quen} \; Z_{(T_i-1)\bl,\, T_i\bl}^{\quen}\,.
\end{equation}
Note that $\go_{\llbracket(T_i-1)\bl,\, T_i\bl\rrbracket}\in\A_l$ for all $i$ by definition of $T_i$, so that $Z_{T_{i-1}\bl,\,(T_i-1)\bl}^{\quen}\geq e^{\G l}$ by definition of $\G$. We also have the obvious bound $Z_{\ba,\bb}^{\quen}\geq K(\|\bb-\ba\|)e^{\gb\go_\bb-\gl(\gb)+h}$ for any $\ba\prec\bb$. Therefore we have
\begin{equation}
\begin{aligned}
Z_{T_k\bl}^{\quen} \,&\geq\, \prod_{i=1}^k \Big(e^{\gb \go_{(T_i-1)\bl}-\gl(\gb)+h}K(\|(T_i-1-T_{i-1})\bl\|) \times e^{l\G}\Big)\\
&\geq\, e^{kl\G} \prod_{i=1}^k \frac{c_{\ga_+} e^{\gb \go_{(T_i-1)\bl}-\gl(\gb)+h}}{\big(l(T_i-T_{i-1})\big)^{2+\ga_+}}\,,
\end{aligned}
\end{equation}
where the last inequality holds for any $\ga_+>\ga$ and a convenient $c_{\ga_+}>0$, by Potter's bound (cf.~\cite{BGT87}).  We can now estimate from below the free energy, using the strong law of large numbers:
\begin{equation}
\begin{aligned}
\tf  & (\gb,h)  \,=\, \lim_{k\rightarrow\infty} \frac{1}{l\,T_k} \log Z_{T_k\bl}^{\quen}\\
&\geq\, \lim_{k\rightarrow\infty} \frac{k}{T_k} \frac{1}{kl}\Big(kl\G + \sum_{i=1}^k\big[ c_{\ga_+, \gb, h}  + \gb \go_{(T_i-1)\mathbf{l}}- (2+\ga_+) \log\big(l(T_i-T_{i-1})\big)\big]\Big)\\
&\geq\, \frac{1}{\bbE[T_1]}\Big(\G + \frac{c_{\ga_+, \gb, h} + \gb \bbE[\go_{\bone}]}{l} - (2+\ga_+) \frac{\log l + \bbE[\log (T_1)]}{l} \Big)\\
&\geq\, e^{-\C l} \Big(\G - (2+\ga_+) \C  + \frac{c'_{\ga_+, \gb, h}-  (2+\ga_+)\log l}{l} \Big) \,,
\end{aligned}
\end{equation}
where we set $c_{\ga_+, \gb, h}:= \log c_{\ga_+} -\lambda(\gb)+h$. For the last inequality, we used Jensen's inequality to get that $\bbE[\log (T_1)] \leq \log (1/\bbP(\cA_l)) \leq \C l$.
Finally, if $\G - (2+\ga) \C>0$, then it also holds for some $\ga_+>\ga$, and the right hand side is strictly positive for $l$ large enough, which implies $\tf(\gb,h)>0$ and concludes the proof.
\end{proof}

\subsection{Smoothing of the phase transition in $\gb$}\label{section:smoothing:2}
Let us discuss the strategy of the proof first. For the PS model (with i.i.d. disorder), the method used in \cite{GTsmooth} to prove a smoothing is to fix $h=h_c(\gb)$, so that $\tf(\gb,h_c(\gb))=0$ and Lemma~\ref{lem:gain-cost} implies $\cG\leq(2+\ga)\cC$. Then one chooses a gain $\cG$ close to $\tf(\gb,h_c(\gb)+u)$---which matches the free energy of the model with a \emph{shifted} disorder, \textit{i.e.}\ with $\go$ replaced by $\go + u/\gb$--- and expresses the corresponding cost $\cC$ with the cost of the change of measure from $\go$ to $\go + u/\gb$ (which can be estimated via a relative entropy inequality), therefore obtaining an upper bound on the free energy near the critical point.

However this method doesn't apply well to our model, mostly because of the dimension of the field $\go$ (this is also discussed in \cite{BGK} for an i.i.d. disorder). A direct shift of the disorder field $\go$ is too costly (we shift $n^2$ variables in a model of size $n$). On the other hand an i.i.d. shift of the sequences $\hgo$ and $\bgo$ by $u/\sqrt{\gb}$ ---although involving only $2n$ variables---is not easily related to a free energy with different parameters. Therefore, we needed to adapt this method. We first prove a smoothing inequality \emph{with respect to $\gb$} instead of $h$, using a \emph{dilation} of the disorder instead of a shift, {\it i.e.} we replace the sequence $\hgo$ (resp. $\bgo$) by $(1+\gd)\hgo$ (resp. $(1+\gd)\bgo$). This change of measure matches the same model with disorder intensity $\gb(1+\gd)^2\approx\gb(1+2\gd)$ instead of $\gb$, and is not too costly (we change the law of $2n$ variables in a system of size $n$).

Let us introduce the ``shifted'' free energy $\tilde \tf(\gb,h):= \tf(\gb, h+\lambda(\gb))$ (\textit{i.e.}\ if we omit the term $-\gl(\gb)$ in the definition of the partition function): in view of Proposition~\ref{prop:existF}, we get that $(\gb,h)\mapsto \tilde \tf(\gb,h)$ is convex, so the critical line $\gb\mapsto \tilde h_c(\gb) := \inf \{ h \colon \tilde \tf(\gb,h) >0\} = h_c(\gb) -\lambda(\gb)$ is concave. Actually it is decreasing and continuous (recall that the upper bound in \eqref{eq:inh} is strict for $\gb>0$), so one can consider the inverse map of $\gb\mapsto \tilde h_c(\gb)$, that we denote $h \mapsto \tilde  \gb_c(h)$: for each $h>0$, the value $\tilde \gb_c(h)$ is the critical value for the map $\gb \mapsto \tilde \tf(\gb,h)$ corresponding to the localization transition. One can therefore consider the transition as $\gb$ varies, and the next proposition tells this phase transition is smooth, \textit{i.e.}\ for fixed $h$, the growth of $\gb \mapsto \tilde \tf(\gb,h)$ close to $\tilde \gb_c(h)$ is at most quadratic.

\begin{proposition}\label{prop:proof:rel:smoothing}
Under Assumption~\ref{hyp:entropy}, for any $h>0$, there exists $c_h>0$ such that for any $\gd \in (0,1)$, one has
\begin{equation}\label{e0q:prop:proof:rel:smoothing}
\tilde \tf\big( (1+\gd) \tilde \gb_c(h) , h \big) \,\leq\, c_h \gd^2\,.
\end{equation}
\end{proposition}

\begin{proof}
For any $\gd>0$, we define $\tP_{l,\gd}$ the law of the disorder in $\bsquare[\bl]$ \emph{dilated} by $(1+\gd)$ on each coordinate ({\it i.e.} $\hgo_{\llbracket1,l\rrbracket}$ is replaced with $(1+\gd)\hgo_{\llbracket1,l\rrbracket}$, same for $\bgo_{\llbracket1,l\rrbracket}$). We denote $\tP_{\gd}$ the infinite product law. Note that $H(\tP_{l,\gd}|\Pb)= l H(\tP_{1,\gd}|\Pb)$, and that there is some $c>0$ such that $H(\tP_{1,\gd}|\Pb)\leq c\,\gd^2$ by Assumption~\ref{hyp:entropy}.
For any $\gb>0$, $h\in\R$ and $l\in\N$, let us define
\begin{equation}
\A_l\,:=\,\Big\{ \go_{\bsquare[\bl]} \in \R^{l^2}\;;\; \frac{1}{l} \log \tilde Z_{\bl}^{\quen} \geq  \frac12 \tilde \tf\big(\gb(1+\gd)^2,h\big)\Big\} \,,
\end{equation}
where $\tilde Z_{\bl}^{\quen} := Z_{\bl, h+\lambda(\gb)}^{\gb, \quen}$, so that $\lim_{l\to +\infty} \frac1l \log \tilde Z_{\bl}^{\quen} := \tilde \tf(\gb,h)$.

The set $\A_l$ is chosen so that the \emph{gain} (as defined for Lemma~\ref{lem:gain-cost}) is $\G=\tfrac12 \tilde \tf\big(\gb(1+\gd)^2,h\big)$, where $\tilde \tf\big(\gb(1+\gd)^2,h\big)$ is exactly the free energy of the gPS model with partition function $\tilde  Z_{\bl}^{\quen}$, where we changed the disorder law from $\Pb$ to $\tP_\gd$ (because multiplying $\hgo$ and $\bgo$ by $(1+\gd)$ is the same thing as multiplying $\gb$ by $(1+\gd)^2$). Thus $\frac{1}{l} \log \tilde Z_{\bl}^{\quen}$ converges $\tP_\gd$-a.s. to $\tf\big(\gb(1+\gd)^2,h\big)$ when $l\rightarrow\infty$, so we obviously have that 
\begin{equation}\label{eq:proof:entropy:1}
\tP_{l,\gd}(\A_l) \underset{l\rightarrow\infty}{\longrightarrow}1\;.
\end{equation}
Now we can estimate $\Pb(\A_l)$ via a standard relative entropy inequality, which gives
\begin{equation}
\begin{aligned}\label{eq:proof:entropy:2}
\Pb(\A_l)
\,&\geq\, \tP_{l,\gd}(\A_l)  \exp\Big( -\frac{1}{\tP_{l,\gd}(\A_l)} \big(H(\tP_{l,\gd} | \Pb) + e^{-1}\big) \Big)\;.
\end{aligned}
\end{equation}
We therefore get that $\Pb(\A_l) \geq \tfrac{e^{-2/e}}{2} e^{- 2 c \gd^2 l  }$, for $l$ large enough (so that $\tP_{l,\gd}(\A_l)\geq 1/2$). Therefore, for $l$ large (how large depends on $\gd$), we get that 
\begin{equation}
\frac{1}{l} \log \Pb(\A_l)\, \geq\,  - 3c \gd^2\;.
\end{equation}
Now, for any $h>0$, fix $\gb = \tilde \gb_c(h)$, so that $h = \tilde h_c(\gb)$. Then $\tilde \tf(\gb, h_c(\gb)) =\tilde \tf ( \tilde \gb_c(h),h)=0$, so Lemma~\ref{lem:gain-cost} implies $\G-(2+\ga)\C\leq0$, thereby
\begin{equation}
\tfrac12 \tilde \tf\big(  (1+\gd)^2 \tilde \gb_c(h),h)\big)- 3c (2+\ga) \gd^2 \,\leq\, 0 \; ,
\end{equation}
which gives $\tilde \tf\big((1+\gd)^2 \tilde \gb_c(h)  ,h)\big) \leq c' \gd^2$.
This concludes the proof by simply recalling that $\tilde \tf$ is non-decreasing in its first coordinate, so $\tilde \tf\big( (1+\gd) \tilde \gb_c(h),h)\big) \leq\tilde \tf\big( (1+\gd)^2 \tilde \gb_c(h) ,h)\big)$.
%
%
\end{proof}

\subsection{Conclusion: smoothing of the phase transition in $h$}

Once we have the smoothing inequality with respect to $\gb$ of Proposition~\ref{prop:proof:rel:smoothing}, we are able transcribe it to a smoothing in $h$ using the convexity properties of $\tilde \tf(\gb,h)$.
We now conclude the proof of Theorem~\ref{thm:rel:smoothing}, thanks to Proposition~\ref{prop:proof:rel:smoothing}.

\begin{proof}[Proof of Theorem~\ref{thm:rel:smoothing}]
We fix $\gb>0$, and we take $t>0$ small enough such that $\tilde h_c(\gb) + t < 0$. Recall that $\tilde h_c(\cdot)$ is a concave, non-increasing, continuous function: there exists $\gb_{t}\in(0,\gb)$ such that $\tilde h_c(\gb_t)= \tilde h_c(\gb)+t$.
Hence,
\begin{equation}\label{eq:proof:rel:smooth}
\tf(\gb,h_c(\gb)+t) \,=\, \tilde \tf(\gb, \tilde h_c(\gb)+t) \,=\, \tilde \tf( \gb_t + u, \tilde h_c(\gb_t))\,,
\end{equation}
where $u=\gb-\gb_t>0$. Using Proposition~\ref{prop:proof:rel:smoothing} and the fact that $\tilde \gb_c( \tilde h_c(\gb_t))= \gb_t$, we have
\begin{equation}\label{eq:proof:rel:smooth:2}
\tf(\gb,h_c(\gb)+t) \,=\, \tilde \tf \big( \gb_t (1+u/\gb_t), \tilde h_c(\gb_t) \big) \,\leq\, c_{\tilde h_c( \gb_t)} (u/\gb_t)^2\,.
\end{equation}
Note that $\gb_t\nearrow\gb>0$ as $t\searrow0$, so the factor $c_{\tilde h_c(\gb_t)} \gb_t^{-2}$ is bounded by a constant that depends only on $\gb$, uniformly in $t$ sufficiently small.
Using that $\tilde h_c(\cdot)$ is concave, we also have
\begin{equation}
-\frac{t}{u}\,=\,\frac{\tilde h_c(\gb)-\tilde h_c(\gb_t)}{\gb-\gb_t} \,\leq\, \tilde h'_c(\gb_t^+) \,\leq\, \tilde h'_c((\gb/2)^+)\,<\,0\,,
\end{equation}
where $\tilde h'_c(\cdot^+)$ is the right derivative of $\tilde h_c$, and the second inequality holds for any $t$ sufficiently small (because $\tilde h'_c(\gb_t^+)<0$ as soon as $\gb_t>0$, and it decreases as $t\searrow0$). 
We deduce that $t\geq c_{\gb} \,u$ for some $c_{\gb} >0$, and plugging it in \eqref{eq:proof:rel:smooth:2}, we finally obtain
\begin{equation}
\tf(\gb,h_c(\gb)+t) \leq c'_\gb \, t^2,
\end{equation}
where $c'_\gb>0$ depends only on $\gb$. This concludes the proof of Theorem~\ref{thm:rel:smoothing}.
\end{proof}

\section{Disorder relevance: shift of the critical point when $\Pb\neq\Pbspinx$} \label{section:shift}

In this section, we prove Theorem~\ref{thm:rel:shift}, that is a  the lower bound for the shift of the critical point when $\Pb\neq\Pbspinx$ for all $x>0$. We will discuss the case $\Pb=\Pbspinx$ in Section~\ref{section:shift:m4=1}.

\subsection{Coarse-graining and fractional moment method}
Our proof is based on a fractional moment method, introduced in \cite{DGLT09} for the original PS model, and slightly adapted to the gPS model with independent disorder in \cite{BGK}. The first part of the proof (the coarse-graining procedure) is identical to that of \cite{BGK}, but the different estimates are specific to our setting.

\smallskip
Recall the notation $Z_{\bn}^{\quen}:=Z_{\bn,h}^{\gb,\quen}$ from Section~\ref{section:smoothing}. Let us define for any $\bn \in\bbN^2$ (note that we do not assume $\bn=(n,n)$ in this section) the \emph{fractional moment} of the partition function: 
\begin{equation}\label{eq:frac:mom}
A_{\bn}\,:=\,\bbE\big[\big(Z_{\bn}^{\quen}\big)^\gh \big]\,,
\end{equation}
where $\gh\in(0,1)$ is a constant we will fix later on (notice that $A_\bn<\infty$ because $Z_{\bn,h}^{\gb,\quen}\in L^1(\bbP)$). For any $k\in\N$ we set $\bk := (k,k)$, and for $\bn \succeq \bk$, we decompose the partition function $Z^{\quen}_{\bn}$ according to the first point $\bn-\bi$ of $\btau$ which lies in the square $\llbracket \bn-\bk+\bone, \bn\rrbracket$ (in particular $\bzero\preceq\bi\prec\bk$), and the point $\bn-\bi-\bj$ before: in particular it is the last point of $\btau$ which isn't in the previous square, so it can only be in one of the three boxes $\llbracket 0,\bn^{(1)}-k \rrbracket \times \llbracket 0,\bn^{(2)}-k \rrbracket$, $ \rrbracket \bn^{(1)}-k,\bn^{(1)} \rrbracket \times  \llbracket 0,\bn^{(2)}-k \rrbracket$ or $\llbracket 0,\bn^{(1)}-k \rrbracket \times \rrbracket \bn^{(2)}-k,\bn^{(2)} \rrbracket$. This decomposition gives us for any $\bn\succeq \bk$,
\begin{equation}\label{eq:decomp:Z}
Z^{\quen}_{\bn} \,=\, Z^{\quen,1}_{\bn} + Z^{\quen,2}_{\bn} + Z^{\quen,3}_{\bn}\,,
\end{equation}
where we write for any $s\in\{1,2,3\}$,
\begin{equation}
Z^{\quen,s}_{\bn} \,:=\, \sum_{(\bi,\bj)\in D_{\bk,\bn}^s} Z_{\bn-\bi-\bj}^\go \times \bP(\btau_1=\bj) e^{\gb\go_{\bn-\bi}-\gl(\gb)+h} \times Z_{\bn-\bi,\bn}^\go\,,
\end{equation}
where $Z_{\bn-\bi,\bn}^\go$ is defined in \eqref{eq:defZcond}, and 
the sets $D_{\bk,\bn}^s, s\in\{1,2,3\}$ are defined as follow (with $\N_0:=\N\cup\{0\}$):
\begin{equation}
\begin{aligned}
D_{\bk,\bn}^1\,&:=\,\{(\bi,\bj)\in\N_0^2\times\N^2\;;\; \bi\prec \bk\preceq \bi+\bj\preceq\bn\}\,,\\
D_{\bk,\bn}^2\,&:=\,\{(\bi,\bj)\in\N_0^2\times\N^2\;;\; \bi\prec \bk,\, \bi^{(1)}+\bj^{(1)} <k,\, \bi^{(2)}+\bj^{(2)} \geq k,\, \bi+\bj\preceq\bn\}\,,\\
D_{\bk,\bn}^3\,&:=\,\{(\bi,\bj)\in\N_0^2\times\N^2\;;\; \bi\prec \bk,\, \bi^{(1)}+\bj^{(1)} \geq k,\, \bi^{(2)}+\bj^{(2)} < k,\, \bi+\bj\preceq\bn\}\,.
\end{aligned}
\end{equation}

\vspace{-0.3cm}

\begin{figure}[h]
 \centering
  \includegraphics[width=0.5\linewidth]{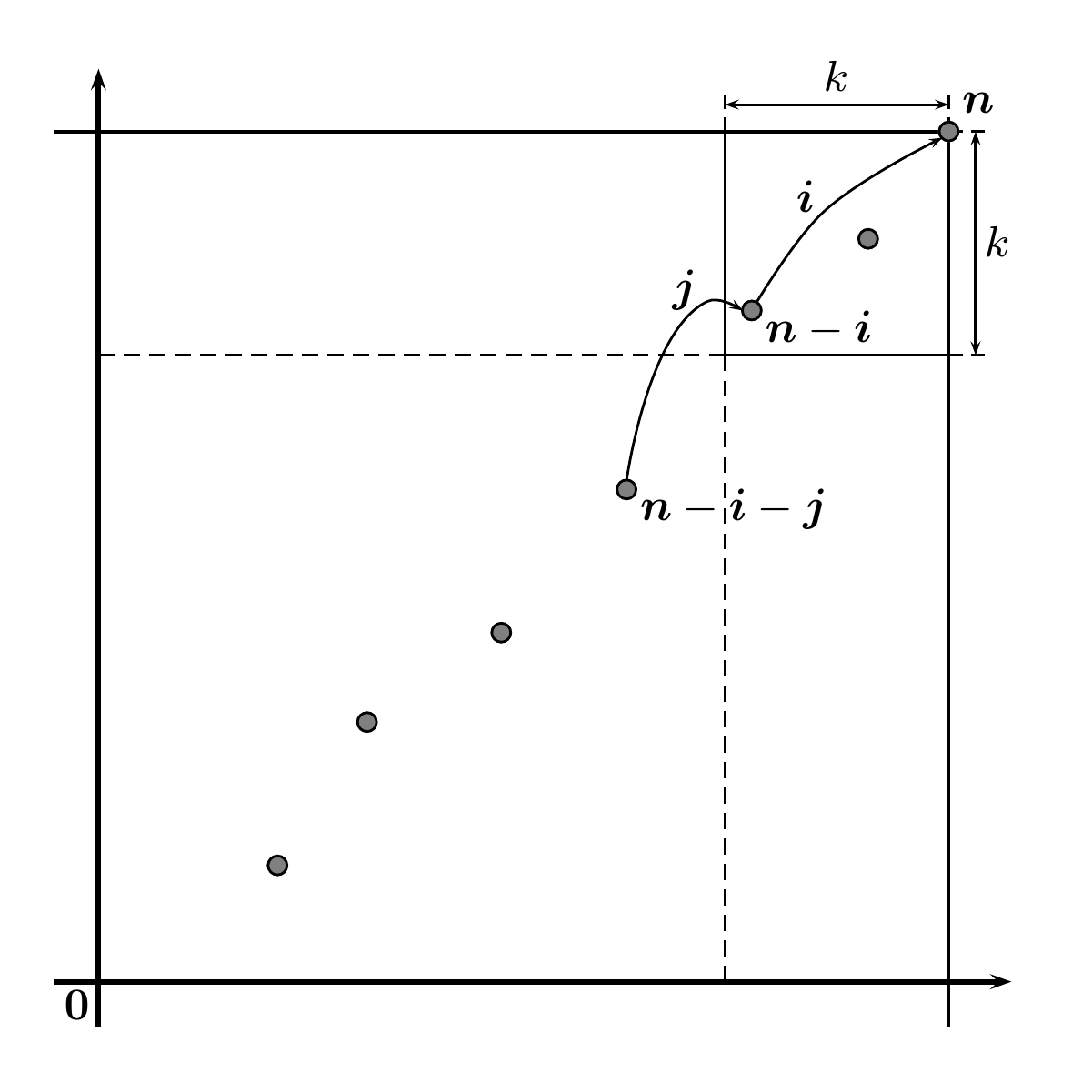}
  \caption{\footnotesize Decomposition of the partition function for the coarse-graining procedure. $\bn-\bi$ denotes the first renewal epoch in the square $\llbracket \bn-\bk+\bone, \bn\rrbracket$, and $\bn-\bi-\bj$ is the renewal epoch before: it is either in the bottom left rectangle ({\it e.g.} the representation in the figure; those are the trajectories contributing to $Z^{\quen,1}_{\bn}$), bottom right ($Z^{\quen,2}_{\bn}$) or top left ($Z^{\quen,3}_{\bn}$).}
  \label{fig:Coarse_graining}
\end{figure}

We also define $D_{\bk,\binfty}^s=\cup_{\bn\in\N^2}D_{\bk,\bn}^s$ (which means we drop the condition $\bi+\bj\preceq \bn$). Recall that $Z^{\quen}_{\bn - \bi,\bn}$ and $Z^{\quen}_{\bi}$ have same law, and that $Z^{\quen}_{\bn-\bi-\bj}$, $e^{\gb\go_{\bn- \bi}-\gl(\gb)+h}$ and $Z^{\quen}_{\bn- \bi,\bn}$ are independent. Using this together with the definition of $A_{\bn}$ and the standard inequality $(\sum_i a_i)^\gh \leq \sum_i a_i^\gh$ for any $\gh\in(0,1)$ and $a_i\geq0, i\in\N$, we obtain
\begin{equation}\label{eq:decomp:A:1}
A_{\bn} \,\leq\, A^1_{\bn} + A^2_{\bn} + A^3_{\bn},
\end{equation}
where for any $s\in\{1,2,3\}$,
\begin{equation}\label{eq:decomp:A:2}
A^{s}_{\bn} \,:=\, c_{\gb,h,\gh} \sum_{(\bi,\bj)\in D_{\bk,\bn}^s} A_{\bn-\bi-\bj} \times K(\|\bj\|)^\gh \times A_{\bi}\,,
\end{equation}
and $c_{\gb,h,\gh}:=\bbE[e^{(\gb\go_{\bone}-\gl(\gb)+h)\gh}]=e^{\gl(\gh\gb)-\gh\gl(\gb)+\gh h}$. Note that $c_{\gb,h,\gh}$ is uniformly bounded for $\gb,h$ small and $0\leq\gh\leq1$, so we can bound it by a constant $C_\cntc$.

\smallskip
As in \cite{BGK}, the proof of Theorem~\ref{thm:rel:shift} relies on the following claim.
\begin{lemma}\label{lem:rel:shift}
For fixed $\gb>0$ and $h\in\R$, if there exist $\gh \in(0,1)$ and $k\in\N$ such that $\gr_1+\gr_2+\gr_3\leq1$, where
\begin{equation}\label{eq:rho}
\gr_s\,:=\,C_{\arabic{cst}} \sum_{(\bi,\bj)\in D_{\bk,\binfty}^s} K(\|\bj\|)^\gh A_\bi\,,\qquad s\in\{1,2,3\}\,,
\end{equation}
then $\tf(\gb,h)=0$.
\end{lemma}
\begin{proof}
The proof is straightforward. Define $\overline A:= \sup \{A_\bi,\, \bi^{(1)}<k\text{ or }\bi^{(2)}<k\}$. By Jensen's inequality, one obviously has $A_\bi\leq\bbE[Z^{\quen}_\bi]^\gh\leq e^{\gh  h\min(\bi^{(1)},\bi^{(2)})}$, because there are at most $\min(\bi^{(1)},\bi^{(2)})$ renewal points in $\llbracket1,\bi\rrbracket$. Thus we have $\overline A \leq e^{\gh hk}$. Using the decomposition \eqref{eq:decomp:A:1}, \eqref{eq:decomp:A:2}  of $A_\bn$ and $\gr_1+\gr_2+\gr_3\leq 1$, we deduce (by induction) $A_\bn\leq \overline A$ for any $\bn\in\N^2$. By applying Jensen's inequality, we conclude
\begin{equation}
\tf(\gb,h)\,=\,\limtwo{\bn\to\infty}{\bn=(n,n)}\frac{1}{\gh\, n}\bbE\log\big[(Z^{\quen}_\bn)^\gh\big] \,\leq\, \limtwo{\bn\to\infty}{\bn=(n,n)}\frac{1}{\gh \,n}\log A_\bn \,=\, 0\,.
\end{equation}
\end{proof}

\begin{proof}[Proof of Theorem~\ref{thm:rel:shift}]
We now assume $\Pb\neq\Pbspinx$ for all $x>0$. We fix $h$ as in Theorem~\ref{thm:rel:shift}:
{\blue
\begin{equation}\label{eq:gD}
h\;=\; h(\gb)\;:=\;
\left\{\begin{aligned} \gb^{\max(\frac{2\ga}{2\ga-1},2)\,+\,\eps} \qquad & \text{if}\quad m_1\neq 0, \\
\gb^{\max(\frac{4\ga}{2\ga-1},4)\,+\,\eps} \qquad& \text{if}\quad m_1= 0,\, m_4>m_2^2, \end{aligned}\right. 
\end{equation}
}
where $\eps>0$ is arbitrarily small, but fixed. Our goal is to choose $\gh\in(0,1)$ and $k\in\N$ such that $\gr_1$, $\gr_2$ and $\gr_3$ (which is symmetric to $\gr_2$) are small, so that Lemma~\ref{lem:rel:shift} implies $\tf(\gb,h)=0$ and $h_c(\gb)\geq h$. First we pick
\begin{equation}\label{eq:k}
k\,=\,k(\gb)\,:=\,\frac{1}{\tf(0,h)}\,,
\end{equation}
which is the correlation length of the annealed system (actually we take the integer part of it, but we omit to write it for clarity purpose). Notice that $k\to\infty$ as $h\to0$, and recall that Theorem~\ref{thm:hom} allows us to write $k=L_{\ga}(1/h) h^{-1/\ga}$ with $L_{\ga}$ a slowly varying function when $\ga\leq1$; and $k\sim c_\ga h^{-1}$ for some $c_\ga>0$ as $h\searrow0$ when $\ga>1$ (here we slightly changed the notations from the theorem, to lighten upcoming computations).

\smallskip
Note that, if $\gh$ is picked such that $(2+\ga)\gh>2$, then we have 
\begin{equation}
\sum_{\bj\succeq \bk-\bi} K(\|\bj\|) \,\leq\, \frac{L_\cntf(\|\bk-\bi\|)}{\|\bk-\bi\|^{(2+\ga)\gh-2}}\,.
\end{equation} Therefore,
\begin{equation}\label{eq:gr1}
\gr_1\,\leq\, \sum_{\bi\prec\bk} \frac{L_\arabic{svf}(\|\bk-\bi\|)}{\|\bk-\bi\|^{(2+\ga)\gh-2}} A_\bi \,\leq\, \sum_{\bi\prec\bk} \frac{L_\cntf(k-\bi^{(2)})}{(k-\bi^{(2)})^{(2+\ga)\gh-2}} A_\bi\,,
\end{equation}
(because $\|\bk-\bi\| \geq k-\bi^{(2)}$). Similarly,
\begin{equation}
\sumtwo{\bj^{(1)}< k-\bi^{(1)}}{\bj^{(2)}\geq k-\bi^{(2)}} K(\|\bj\|) \,\leq\,  \sum_{\bj^{(2)}\geq k-\bi^{(2)}}\frac{L_\cntf(\bj^{(2)})}{(\bj^{(2)})^{(2+\ga)\gh-1}} \,\leq\,  \frac{L_\cntf(k-\bi^{(2)})}{(k-\bi^{(2)})^{(2+\ga)\gh-2}}\,.
\end{equation}
Thus,
\begin{equation}\label{eq:gr2}
\gr_2\,\leq\,\sum_{\bi\prec\bk} \frac{L_\arabic{svf}(k-\bi^{(2)})}{(k-\bi^{(2)})^{(2+\ga)\gh-2}} A_\bi\,=:\,S\,.
\end{equation}

Let us denote this sum $S$. Because $\gr_3$ is symmetric to $\gr_2$, and because $\gr_1$ and $\gr_2$ are bounded by the same sum $S$ (up to the slowly varying function, which doesn't change the behavior of the sum), Theorem~\ref{thm:rel:shift} will be proven as soon as we show that $S$ can be made small, by applying Lemma~\ref{lem:rel:shift}. For that we need some precise estimates on $A_\bi$ for $\bi\prec \bk$.
We will handle $A_{\bi}$ differently depending on whether $\bi$ is small or not, using Lemmas~\ref{lem:Afrac1}-\ref{lem:Afrac2} below.

\begin{lemma}\label{lem:Afrac1}
There exist constants $C_\cntc>0$, $h_1>0$ and a slowly varying function $L_\cntf$ such that for any $h\in(0,h_1)$ and $\bn$ with $1\leq\|\bn\|\leq 1/\tf(0,h)$, one has
\begin{equation}\label{eq:Afrac1}
Z_{\bn,h}\;\leq\, \left\{\begin{array}{ll} 
C_{\arabic{cst}} \,L(\|\bn\|)^{-1} \,\|\bn\|^{-(2-\ga)}& \text{if }\ga\in(0,1)\;,\vspace{1mm}\\
L_{\arabic{svf}}(\|\bn\|) \, \|\bn\|^{-1/\min(\ga,2)} &\text{if }\ga>1\;;
\end{array}\right.\end{equation}
where $Z_{\bn,h}$ is the partition function of the homogeneous model with parameter $h$.\\
When $\ga=1$, the upper bound 
$Z_{\bn,h}\;\leq\, L_{\arabic{svf}}(1/\|\bn\|) \, \|\bn\|^{-1}
$ 
holds for $1\leq\|\bn\|\leq \tf(0,h)^{-(1-\eps^2)}$.
\end{lemma}
This lemma implies that $Z_{\bn,h}$ remains of the same order as $Z_{\bn,0}$ as long as $n$ does not exceed the correlation length $1/\tf(0,h)$ (in the case $\alpha=1$ we added the assumption $\|\bn\|\leq \tf(0,h)^{-(1-\eps^2)}$ only to avoid technicalities). Then we have by Jensen's inequality that $A_{\bi}\leq (\bbE Z_{\bi}^{\quen})^\gh= (Z_{\bi,h})^\gh$, therefore Lemma~\ref{lem:Afrac1} gives us a first bound on $A_{\bi}$ valid for all~$\bi$ not too large.


When $\bi$ is of greater order, we can obtain an improved bound on $A_\bi$ as soon as $\Pb\neq\Pbspinx$ for all $x>0$.

{\blue
\begin{lemma}\label{lem:Afrac2}
Assume that $\Pb\neq\Pbspinx$ for all $x>0$. If $\eps$ has been fixed small enough, then there exist some constants $C_{\rcntc{cst:lemnum}}, C_{\rcntc{cst:lemexpeps}}, C_{\rcntc{cst:lemexpeps:ga>1}}>0$ such that for any $k^{(1-\eps^2)}\leq  \| \bn \|\leq 2 k$, and for $h$ defined as in \eqref{eq:gD}, one has
\begin{equation}\label{eq:Afrac2}
A_{\bn}\;\leq\,\bigg\{ \begin{array}{ll}
C_{\ref{cst:lemnum}} \|\bn\|^{-(2-\ga\,+\,\eps\,C_{\ref{cst:lemexpeps}})\gh} &\text{if }\ga\in(1/2,1],\vspace{1mm}\\
C_{\ref{cst:lemnum}} \|\bn\|^{-(\ga\,+\,\eps\,C_{\ref{cst:lemexpeps:ga>1}})\gh}  &\text{if }\ga>1. \end{array}
\end{equation}
\end{lemma}
}
We postpone the proof of Lemma~\ref{lem:Afrac1} to Appendix~\ref{section:appendixA} (because it only concerns the homogeneous gPS model), and we prove Lemma~\ref{lem:Afrac2} later in this section.
We now have all the tools we need to finish the proof. Fix $\eps>0$ sufficiently small for Lemma~\ref{lem:Afrac2} to apply, then define $\gh$ depending on $\ga>1/2$.\smallskip

\noindent \emph{Case $\ga\in(1/2,1)$.} Choose $\gh$ such that
\begin{equation}\label{eq:gga0}
\max\left(\frac{4}{4+\eps\,C_{\ref{cst:lemexpeps}}}\,,\,\frac{4-2\eps^2}{4-2\eps^2+\eps^2 \ga}\,,\, \frac{2}{2+\ga} \,,\, \frac{1}{2-\ga}\right) \;<\;\gh\;<\;1\;.
\end{equation}
In particular $(2+\ga)\gh>2$ so \eqref{eq:gr1} and \eqref{eq:gr2} hold and we only have to control the sum $S$ defined in \eqref{eq:gr2}. We introduce the notation $u_k:=k^{(1-\eps^2)}$, and we part the sum in two:
\begin{equation}\label{eq:splitS}
S:= S_1+S_2:=  \sumtwo{\bi\prec\bk,}{\|\bi\|\geq u_k} \frac{L_4(k-\bi^{(2)})}{(k-\bi^{(2)})^{(2+\ga)\gh-2}} A_\bi \; + \sumtwo{\bi\prec\bk,}{\|\bi\| < u_k} \frac{L_4(k-\bi^{(2)})}{(k-\bi^{(2)})^{(2+\ga)\gh-2}} A_\bi.
\end{equation}
To bound $S_1$, we use Lemma~\ref{lem:Afrac2} to estimate $A_\bi$.
\begin{equation}
\begin{aligned}
S_1 &\leq  \sumtwo{\bi\prec\bk,}{\|\bi\|\geq u_k} \frac{L_4(k-\bi^{(2)})}{(k-\bi^{(2)})^{(2+\ga)\gh-2}} \frac{C_{\ref{cst:lemnum}}}{\|\bi\|^{(2-\ga+\eps\,C_{\ref{cst:lemexpeps}})\gh}}\\
&\leq \sum_{\bi^{(2)}=0}^{k-1} \frac{L_4(k-\bi^{(2)})}{(k-\bi^{(2)})^{(2+\ga)\gh-2}} \frac{C_\cntc}{(\max(\bi^{(2)},u_k))^{(2-\ga+\eps C_{\ref{cst:lemexpeps}})\gh-1}},
\end{aligned}
\end{equation}
where we used $(2-\ga+\eps \,C_{\ref{cst:lemexpeps}})\gh \geq (2-\ga)\gh>1$ (recall \eqref{eq:gga0} and $\gh>\gh_0$), and summed over $\bi^{(1)}$. We split this sum once more according to whether $\bi^{(2)}\leq k/2$ or $\bi^{(2)}>k/2$, which gives
\begin{equation}
\begin{aligned}
S_1 &\leq \Bigg(\sum_{\bi^{(2)}=0}^{k/2}+\sum_{\bi^{(2)}=k/2+1}^{k-1}\Bigg)\frac{L_4(k-\bi^{(2)})}{(k-\bi^{(2)})^{(2+\ga)\gh-2}} \frac{C_\cntc}{(\max(\bi^{(2)},u_k))^{(2-\ga+\eps C_{\ref{cst:lemexpeps}})\gh-1}}\\
&\leq \frac{L_\cntf(k)}{k^{(2+\ga)\gh-2}}\frac{C_\cntc}{k^{(2-\ga+\eps\,C_{\ref{cst:lemexpeps}})\gh-2}} + \frac{L_\cntf(k)}{k^{(2+\ga)\gh-3}}\frac{C_\cntc}{k^{(2-\ga+\eps\,C_{\ref{cst:lemexpeps}})\gh-1}}
\leq \frac{L_\cntf(k)}{k^{(4+\eps\,C_{\ref{cst:lemexpeps}})\gh-4}},
\end{aligned}
\end{equation}
where, in the first term we used $(2-\ga+\eps\,C_{\ref{cst:lemexpeps}})\gh<2$ and bounded uniformly in the first factor, and in the second we used $(2+\ga)\gh<3$ and bounded uniformly in the second factor. Finally we have $(4+\eps\,C_{\ref{cst:lemexpeps}})\gh-4>0$ because of \eqref{eq:gga0}, so $S_1$ can be made small if $k$ is large enough (\textit{i.e.}\ if $\gb$ is small enough). Note that it is crucial here to have $C_{\ref{cst:lemexpeps}}>0$ thanks to Lemma~\ref{lem:Afrac2}, which relies on the assumption that $\bbP\neq\Pbspinx$ for all $x>0$.
\smallskip

For $S_2$, we bound the first factor uniformly in $\bi^{(2)}$ (note that $(k-\bi^{(2)})\geq k/2$ for $\|\bi\|<u_k$), and we estimate $A_\bi$ with Lemma~\ref{lem:Afrac1}.
\begin{equation}
\begin{aligned}
S_2 &\leq \frac{L_\cntf(k)}{k^{(2+\ga)\gh-2}}  \sumtwo{\bi\prec\bk,}{1\leq \|\bi\| < u_k} \frac{C_1}{\|\bi\|^{(2-\ga)\gh}L(\|\bi\|)^\gh} \; + \; \frac{L_4(k)}{k^{(2+\ga)\gh-2}} A_\bzero\\
&\leq \frac{L_\cntf(k)}{k^{(2+\ga)\gh-2}\,u_k^{(2-\ga)\gh-2}} +\frac{L_4(k)}{k^{(2+\ga)\gh-2}}\\
&\leq \frac{L_{\arabic{svf}}(k)}{k^{(2+\ga)\gh-2 + (1-\eps^2)((2-\ga)\gh-2)}} + \frac{L_\cntf(k)}{k^{(2+\ga)\gh-2}},
\end{aligned}
\end{equation}
where we used $(2-\ga)\gh\in(1,2)$, and we recall $u_k=k^{(1-\eps^2)}$ (note that we had to write separately the term $\bi=\bzero$). The exponent in the denominator of the first term can be written $(4-2\eps^2+\eps^2\ga))\gh-4+2\eps^2$ which is positive because of \eqref{eq:gga0}, and the other exponent is also positive. Thus $S_2$ is also small for $k$ large (\textit{i.e.}\ $\gb$ small).\smallskip

Therefore, there is some $\gb_\eps>0$ such that for any $\gb\in(0,\gb_\eps)$, $S\leq1/3$. Thereby $\gr_1+\gr_2+\gr_3\leq 1$ for our choice of $k$ and $\gh$. Applying Lemma~\ref{lem:rel:shift}, this concludes the proof of Theorem~\ref{thm:rel:shift} in the case $\ga\in(1/2,1)$. \smallskip

\noindent \emph{Case $\ga>1$.} This is very similar to the previous case. Choose $\gh$ such that
\begin{equation}\label{eq:gga0:ga>1}
\max\left(\frac{3}{2+\ga}\,,\,\frac{1}{\ga + \eps \,C_{\ref{cst:lemexpeps:ga>1}}} \,,\, \frac{4-2\eps^2}{2+\ga+\frac{1-\eps^2}{\min(\ga,2)}}\right) \;<\;\gh\;<\;1\;.
\end{equation}
The first two terms in the maximum are obviously strictly smaller than $1$; regarding the third one, one should first notice that $2-\frac{1}{\min(\ga,2)}$ is strictly smaller than $\ga$ (it is obvious if $\ga\geq2$, and follows from $(\ga-1)^2>0$ if $1<\ga<2$). This implies  
$(1-\eps^2)(2-\frac{1}{\min(\ga,1)}) < \ga$,
which is equivalent to
$4-2\eps^2\,<\,2+\ga+\frac{(1-\eps^2)}{\min(\ga,2)}$, thus $\gh$ is well defined. Moreover $(2+\ga)\gh>2$ so \eqref{eq:gr1} and \eqref{eq:gr2} hold again and we only have to control $S$.

As in the case $\ga\in(1/2,1)$, we split it according to $\|\bi\|\geq u_k:=k^{(1-\eps^2)}$ and $\|\bi\|<u_k$ to obtain \eqref{eq:splitS}. We bound $S_1$ with Lemma~\ref{lem:Afrac2}.
\begin{equation}
\begin{aligned}
S_1 &\leq  \sumtwo{\bi\prec\bk,}{\|\bi\|\geq u_k} \frac{L_4(k-\bi^{(2)})}{(k-\bi^{(2)})^{(2+\ga)\gh-2}} \frac{C_{\ref{cst:lemnum}}}{\|\bi\|^{(\ga+\eps\,C_{\ref{cst:lemexpeps:ga>1}})\gh}}\\
&\leq \Bigg(\sum_{\bi^{(2)}=0}^{k/2}+\sum_{\bi^{(2)}=k/2+1}^{k-1}\Bigg)\frac{L_4(k-\bi^{(2)})}{(k-\bi^{(2)})^{(2+\ga)\gh-2}} \frac{C_\cntc}{(\max(\bi^{(2)},u_k))^{(\ga+\eps C_{\ref{cst:lemexpeps:ga>1}})\gh-1}}\\
&\leq \frac{L_\cntf(k)}{k^{(2+\ga)\gh-2}}\frac{C_\cntc}{k^{(\ga+\eps\,C_{\ref{cst:lemexpeps:ga>1}})\gh-2}} + \frac{L_\cntf(k)}{k^{(2+\ga)\gh-3}}\frac{C_\cntc}{k^{(\ga+\eps\,C_{\ref{cst:lemexpeps:ga>1}})\gh-1}}
\leq \frac{L_\cntf(k)}{k^{(2+2\ga+\eps\,C_{\ref{cst:lemexpeps:ga>1}})\gh-4}},
\end{aligned}
\end{equation}
where we used $(2+\ga)\gh>3$ and $(\ga+\eps\,C_{\ref{cst:lemexpeps:ga>1}})\gh>1$ to estimate the sums, and the last bound goes to $0$ as $k\to\infty$ because $(2+2\ga+\eps\,C_{\ref{cst:lemexpeps:ga>1}})\gh>4$.\smallskip

For $S_2$, we bound the first factor uniformly in $\bi^{(2)}$ and estimate $A_\bi$ with Lemma~\ref{lem:Afrac1}.
\begin{equation}
\begin{aligned}
S_2 &\leq \frac{L_\cntf(k)}{k^{(2+\ga)\gh-2}}  \sumtwo{\bi\prec\bk,}{1\leq \|\bi\| < u_k} \frac{L_\cntf(\|\bi\|)}{\|\bi\|^{\gh/\min(\ga,2)}} \; + \; \frac{L_4(k)}{k^{(2+\ga)\gh-2}} A_\bzero\\
&\leq \frac{L_\cntf(k)}{k^{(2+\ga)\gh-2}\,u_k^{\gh/\min(\ga,2)-2}} + \frac{L_4(k)}{k^{(2+\ga)\gh-2}}\\
&\leq \frac{L_{\arabic{svf}}(k)}{k^{(2+\ga)\gh-2 + (1-\eps^2)(\gh/\min(\ga,2)-2)}} + \frac{L_\cntf(k)}{k^{(2+\ga)\gh-2}}.
\end{aligned}
\end{equation}
The exponent in the denominator of the first term can be written $\big(2+\ga+\frac{1-\eps^2}{\min(\ga,2)}\big)\gh-(4-2\eps^2)$, which is positive. Finally $S_2$ vanishes too as $k\to\infty$, and this concludes the proof of Theorem~\ref{thm:rel:shift} in the case $\ga>1$ with Lemma~\ref{lem:Afrac2}.\smallskip

\noindent \emph{Case $\ga=1$.} Choose $\gh$ such that
\begin{equation}\label{eq:gga0:ga=1}
\max\left(\frac{2}{3}\,,\,\frac{4}{4+\eps\,C_{\ref{cst:lemexpeps}}}\right)\;<\;\gh\;<\;1\;.
\end{equation}
We have $(2+\ga)\gh=3\gh>2$ again, so we only have to control $S$, which we split it again in $S_1+S_2$ as in \eqref{eq:splitS}. First we control $S_1$ with Lemma~\ref{lem:Afrac2}.
\begin{equation}
\begin{aligned}
S_1 &\leq  \sumtwo{\bi\prec\bk,}{\|\bi\|\geq u_k} \frac{L_4(k-\bi^{(2)})}{(k-\bi^{(2)})^{3\gh-2}} \frac{C_{\ref{cst:lemnum}}}{\|\bi\|^{(1+\eps\,C_{\ref{cst:lemexpeps}})\gh}}\\
&\leq \Bigg(\sum_{\bi^{(2)}=0}^{k/2}+\sum_{\bi^{(2)}=k/2+1}^{k-1}\Bigg)\frac{L_4(k-\bi^{(2)})}{(k-\bi^{(2)})^{3\gh-2}} \frac{C_\cntc}{(\max(\bi^{(2)},u_k))^{(1+\eps C_{\ref{cst:lemexpeps}})\gh-1}}\\
&\leq \frac{L_\cntf(k)}{k^{3\gh-2}}\frac{C_\cntc}{k^{(1+\eps\,C_{\ref{cst:lemexpeps}})\gh-2}} + \frac{L_\cntf(k)}{k^{3\gh-3}}\frac{C_\cntc}{k^{(1+\eps\,C_{\ref{cst:lemexpeps}})\gh-1}}
\leq \frac{L_\cntf(k)}{k^{(4+\eps\,C_{\ref{cst:lemexpeps}})\gh-4}},
\end{aligned}
\end{equation}
and this sum decays because of \eqref{eq:gga0:ga=1}. Then we control $S_2$ with Lemma~\ref{lem:Afrac1}.
\begin{equation}
\begin{aligned}
S_2 &\leq \sumtwo{\bi\prec\bk,}{1\leq \|\bi\| < u_k} \frac{L_4(k-\bi^{(2)})}{(k-i_2)^{3\gh-2}} \frac{L_5(\|\bi\|)}{\|\bi\|} \; + \; \frac{L_4(k)}{k^{3\gh-2}} A_\bzero\\
&\leq \sum_{i_2=0}^{u_k}  \frac{L_4(k-\bi^{(2)})}{(k-i_2)^{3\gh-2}} C_\cntc \log(u_k) \; + \; \frac{L_4(k)}{k^{3\gh-2}} 
\leq  \frac{L_\cntf(k)}{k^{3\gh-2}} ,
\end{aligned}
\end{equation}
so $S_2$ vanishes too as $k\to\infty$. This concludes the proof of Theorem~\ref{thm:rel:shift} for $\ga=1$.
\end{proof}

{\blue

\subsection{Fractional moment estimate: proof of Lemma~\ref{lem:Afrac2}}
\label{section:shift:m4>1}

We prove Lemma \ref{lem:Afrac2} via a change of measure procedure. Even though we obtain the same estimate on the fractional moment in both cases $m_1\neq0$ and $m_1=0$, $m_4>m_2^2$, we have to handle them separately and with different values of $h$.

\medskip
{\bf \noindent Case $m_1\neq0$.} For any $\bn\in\N^2$, let us define a new probability measure by tilting both sequences $(\hgo_i)_{i\leq \bn^{(1)}}$, $(\bgo_i)_{i\leq \bn^{(2)}}$ by some $\gd\in(-\gd_0,\gd_0)$:
\begin{equation}
\frac{\dd \bbP_{\bn,\gd}}{\dd \bbP} (\go) \,=\, \frac{e^{\gd\sum_{i_1=1}^{\bn^{(1)}} \hgo_{i_1} + \gd\sum_{i_2=1}^{\bn^{(2)}} \bgo_{i_2}}}{Q(\gd,0)^{\|\bn\|/2}}\,,
\end{equation}
where
\begin{equation}\label{eq:def:Q:m1neq0}
Q(\gd,\gb)\,:=\, \bbE\big[e^{\gb \,\go_{\bone} + \gd \,\hgo_1+\gd\,\bgo_1}\big]\,.
\end{equation}
It is well-defined for $\gb\in[0,\gb_1)$ and $\gd\in(-\gd_0,\gd_0)$, for some $\gb_1, \gd_0>0$, because
\begin{align*}
Q(\gd,\gb)&\leq\bbE\big[e^{(\gb+2|\gd|)\max(|\go_\bone|,\,\hgo_1,\,\bgo_1)}\big] \leq \bbE\big[e^{(\gb+2|\gd|)|\go_\bone|}\big] +2\, \bbE\big[e^{(\gb+2|\gd|)\hgo_1}\big] \\& \leq \bbE\big[e^{(\gb+2|\gd|)(\hgo_1^2+\bgo_1^2)/2}\big] +2\, \bbE\big[e^{(\gb+2|\gd|)(\hgo_1^2+1)}\big]<\infty\;,
\end{align*}
where we used $|\go_\bone|\leq \frac12 (\hgo_1^2+\bgo_1^2)$ and $\hgo_1\leq \hgo_1^2+1$ for all $\hgo_1, \bgo_1\in\R$. Recall $A_{\bn}=\bbE\big[\big(Z^{\quen}_{\bn}\big)^\gh\big]$ with $\gh\in(0,1)$. Using H\"{o}lder's inequality, we write for any $\bn\in\N^2$, $\gd\in(-\gd_0,\gd_0)$,
\begin{align}
\nonumber A_{\bn}=\bbE\big[\big(Z^{\quen}_{\bn}\big)^\gh\big] &= \bbE_{\bn,\gd}\bigg[\big(Z^{\quen}_{\bn}\big)^\gh\bigg(\frac{\dd \bbP_{\bn,\gd}}{\dd \bbP} (\go)\bigg)^{-1}\bigg]\\
\nonumber &\leq \big(\bbE_{\bn,\gd}\big[Z^{\quen}_{\bn}\big]\big)^\gh \Bigg(\bbE_{\bn,\gd}\Bigg[\bigg(\frac{\dd \bbP_{\bn,\gd}}{\dd \bbP} (\go)\bigg)^{-1/(1-\gh)}\Bigg]\Bigg)^{1-\gh}\\
\label{eq:holder:m1neq0} &= \big(\bbE_{\bn,\gd}\big[Z^{\quen}_{\bn}\big]\big)^\gh \Big(Q(\gd,0)^\gh Q\big(-\gd\gh/(1-\gh),0\big)^{1-\gh}\Big)^{\|\bn\|/2}.
\end{align}
Using a Taylor expansion of $Q(\gd,0)=\bbE[e^{\gd\,\hgo_1}]^2$ as $\gd\to0$ (which we will detail below in \eqref{eq:taylor:Q:m1neq0}), we have
\begin{equation}
Q(\gd,0)^\gh Q\big(-\gd\gh/(1-\gh),0\big)^{1-\gh} \,=\, 1 + \gd^2(m_2-m_1^2)\frac{\gh}{1-\gh} + O(|\gd|^3)\,,
\end{equation}
where $m_2-m_1^2=\Var(\hgo_1)>0$ (otherwise disorder is constant a.s.), so this expression is bounded by $\exp(C_\cntc\,\gd^2)$ for some $C_{\arabic{cst}}>0$, uniformly in $\gd\in(-\gd_1,\gd_1)$ for some $\gd_1>0$. Therefore,
\begin{equation}\label{eq:Afracexp:m1neq0}
A_{\bn} \;\leq\; \big(\bbE_{\bn,\gd}\big[Z^{\quen}_{\bn}\big]\big)^\gh \exp\big(C_{\arabic{cst}}\,\gd^2 \|\bn\|/2\big)\,.
\end{equation}
We fix right away $|\gd|:=\|\bn\|^{-1/2}>0$, so that the exponential is bounded by a constant ---the sign of $\gd$ will be determined below.

Let us now compute $\bbE_{\bn,\gd}[Z_{\bn}^{\quen}]$, and estimate it with our choice of $\gd$. Using Fubini's theorem, one has
\begin{equation*}
\begin{aligned}
\bbE_{\bn,\gd}\big[Z_{\bn}^{\quen}\big] 
\,&=\, \bE\bigg[\bbE_{\bn,\gd}\Big[e^{\sum_{\bi\in \llbracket 1, \bn \rrbracket } (\gb\go_\bi-\gl(\gb)+h)\ind_{\{\bi\in\btau\}}}\Big]\ind_{\{\bn \in\btau\}}\bigg]\\
&=\, \bE\,\bbE\Bigg[ e^{(-\gl(\gb)+h)|\btau\cap\bnsquare|} e^{\sum_{\bi\in\btau\cap\bnsquare}\gb\go_\bi} \prod_{i_1=1}^{\bn^{(1)}}\frac{e^{\gd \,\hgo_{i_1}}}{Q(\gd,0)^{1/2}} \prod_{i_2=1}^{\bn^{(2)}}\frac{e^{\gd \,\bgo_{i_2}}}{Q(\gd,0)^{1/2}} \, \ind_{\{\bn \in\btau\}} \Bigg]\\
&=\, \bE\Bigg[ e^{(-\gl(\gb)+h)|\btau\cap\bnsquare|}\!\! \prod_{\bi\in\btau\cap\bnsquare} \frac{\bbE\big[e^{\gb\go_\bi+\gd\hgo_{\bi^{(1)}}+\gd\bgo_{\bi^{(2)}}}\big]}{Q(\gd,0)} \\
&\hspace{5.5cm} \times \prodtwo{i_1\leq \bn^{(1)},}{i_1\notin\btau^{(1)}} \frac{\bbE[e^{\gd \,\hgo_{i_1}}]}{Q(\gd,0)^{1/2}}  \prodtwo{i_2\leq \bn^{(2)},}{i_2\notin\btau^{(2)}} \frac{\bbE[e^{\gd \,\bgo_{i_2}}]}{Q(\gd,0)^{1/2}} \; \ind_{\{\bn \in\btau\}} \Bigg]\,,
\end{aligned}
\end{equation*}
where we used that $\hgo_i$, $\bgo_j$, $i,j\in\N$ are independent variables. Noticing that the factors with $i_1\notin\btau^{(1)}$, $i_2\notin\btau^{(2)}$ simplify to 1, and recalling $e^{\gl(\gb)}=Q(0,\gb)$, this finally gives
\begin{equation}\label{eq:etilte:1:m1neq0}
\bbE_{\bn,\gd}\big[Z_{\bn}^{\quen}\big] \,=\, \bE\bigg[\Big(e^{h}\frac{Q(\gd,\gb)}{Q(\gd,0)Q(0,\gb)}\Big)^{|\btau\cap\bnsquare|} \, \ind_{\{\bn \in\btau\}}\bigg]\, .
\end{equation}
Now we have to estimate $\mathrm{Frac}(\gd,\gb):=\frac{Q(\gd,\gb)}{Q(\gd,0)Q(0,\gb)}$  as $\gb,\gd \to 0$ (recall that they are related to $h$ by \eqref{eq:gD}, \eqref{eq:k} and $|\gd|=\|\bn\|^{-1/2}$, with $k^{1-\gep^2}\leq \|\bn\| \leq 2k$). A Taylor expansion of $Q(\gd,\gb)$ to the second order in $(\gd,\gb)$, combined with the finite exponential moment, gives
\begin{equation}\label{eq:taylor:Q:m1neq0}
Q(\gd,\gb)=1+\gb m_1^2 + 2\gd m_1 + \frac12 \gb^2m_2^2 + 2\gd\gb m_2m_1 + \gd^2(m_2+m_1^2) + O(\,\cdot^3)
\,,
\end{equation}
with $O(\,\cdot^3):=O(\gb^3)+O(\gb^2\gd)+O(\gb\gd^2)+O(\gd^3)$. This leads to the following expansion:
\begin{equation}\label{eq:taylor:Q2:m1neq0}
\mathrm{Frac}(\gd,\gb)= 1 + \gd\gb m_1(m_2-m_1^2) + O(\gb^p) + O(\gd^2) + o(\gd\gb)\,,
\end{equation}
where we can push the expansion to any order $p\in\N$ in $\gb$ because $\mathrm{Frac}(0,\gb)=1$ for all $\gb\geq0$ and $\mathrm{Frac}(\gd,\gb)$ is $\cC^\infty$ on $(-\gd_0,\gd_0)\times[0,\gb_1)$. Recall that $m_2-m_1^2>0$ and that we are in the case $m_1\neq0$. Let us now prove that, with our choice of $\gd$, the order $\gd\gb$ is dominating in~\eqref{eq:taylor:Q2:m1neq0},
and deduce a bound on~$\bbE_{\bn,\gd}[Z_{\bn}^{\quen}]$.

\smallskip
\noindent
\emph{Case $\ga\in(1/2,1]$}: Because of our assumptions $k^{(1-\eps^2)}\leq\|\bn\|\leq 2k$, \eqref{eq:gD}, \eqref{eq:k} and Theorem~\ref{thm:hom}, and because we can bound the slowly varying function from Theorem~\ref{thm:hom} by $h^{\eps^{3/2}}\leq L_\ga(1/h) \leq h^{-\eps^{3/2}}$ for $h$ sufficiently small, there exists some $\gb_\eps>0$ such that for any $\gb<\gb_\eps$,
\begin{align}
\nonumber h^{\frac12(\frac{1}{\ga}+\eps^{3/2})} &\,\leq\,|\gd| \,\leq\, h^{\frac12(\frac{1}{\ga}-\eps^{3/2})(1-\eps^2)},\\  \label{eq:gdtilt:m1neq0}
\gb^{\frac{1}{2\ga-1}+\frac{1}{2\ga}\eps+c_1\,\eps^{3/2}} &\,\leq\, |\gd| \,\leq\, \gb^{\frac{1}{2\ga-1}+\frac{1}{2\ga}\eps-c_1\,\eps^{3/2}}\,,
\end{align}
for some constant $c_1>0$ provided that $\eps$ was taken sufficiently small. Notice that $\gd$ decays faster than~$\gb$ (especially if $\ga$ is close to $1/2$), so $\gd^2 \ll |\gd|\gb$. Moreover we can fix $p$ sufficiently large (depending on $\ga$) such that $\gb^{p} \ll |\gd|\gb$. Finally, letting $\gd:=-\mathrm{sign}(m_1)\|\bn\|^{-1/2}$,
there exist a constant $c>0$ such that for any $\gb<\gb_\eps$, we have
\begin{equation}\label{eq:Afrachom:taylorexp}
\mathrm{Frac}(\gd,\gb) \leq 1 - c |\gd|\gb \leq e^{-c |\gd|\gb}\,.
\end{equation}
Eventually, noticing that $h=\gb^{\frac{2\ga}{2\ga-1}+\eps}\ll \gb^{\frac{2\ga}{2\ga-1}+\frac{1}{2\ga}\eps+c_1\,\eps^{3/2}} \leq |\gd|\gb$ if $\eps$ is sufficiently small, and recalling $|\gd|=\|\bn\|^{-1/2}$, we can bound \eqref{eq:etilte:1:m1neq0} from above by
\begin{equation}\label{eq:Afrachom:m1neq0}
\bbE_{\bn,\gd}\big[Z_{\bn}^{\quen}\big] \,\leq\, \bE\Big[\big(e^{-c'\|\bn\|^{-1/2}\gb}\big)^{|\btau\cap\bnsquare|} \ind_{\{\bn \in\btau\}}\Big]\,,
\end{equation}
for some $c'>0$.

\smallskip
\noindent
\emph{Case $\ga>1$}: This is very similar to the previous case, with simpler bounds because the slowly varying function from Theorem~\ref{thm:hom} is replaced by a constant. So putting together~\eqref{eq:k}, Theorem~\ref{thm:hom} and our assumption $k^{(1-\eps^2)}\leq\|\bn\|\leq 2k$, we deduce that there exist $c_1,c_2>0$ and $\gb_\eps>0$ such that for any $\gb<\gb_{\eps}$ we have $c_1 h^{\frac{1}{2}} \leq |\gd|\leq c_2 h^{\frac{1-\eps^2}{2}}$. Recalling \eqref{eq:gD}, this yields
\begin{equation}\label{eq:gdtilt:ga>1:m1neq0}
c'_1 \, \gb^{1+\frac{\eps}{2}}\,\leq\,|\gd| \,\leq\, c'_2 \, \gb^{\frac{1}{2}(2+\eps)(1-\eps^2)}\leq c'_2 \gb^{1+\frac{\eps}{4}}.
\end{equation}
Here again $|\gd| \ll \gb$, and we can fix $p\in\N$ sufficiently large so that $\gb^{p} \ll |\gd|\gb$. Fixing $\gd:=-\mathrm{sign}(m_1)\|\bn\|^{-1/2}>0$, there is a constant $c''>0$ such that for any $\gb<\gb_\eps$, we have $\mathrm{Frac}(\gd,\gb) \,\leq\, e^{-c'' |\gd|\gb}$. Finally we have $|\btau\cap\bnsquare| \leq \|\bn\|\leq C_\cntc h^{-1}$ (recall Theorem~\ref{thm:hom}), so $e^{h|\btau\cap\bnsquare|}$ is bounded uniformly by some constant. Thus:
\begin{equation}\label{eq:Afrachom:ga>1:m1neq0}
\bbE_{\bn,\gd}\big[Z_{\bn}^{\quen}\big] \,\leq\, C_\cntc \bE\Big[\big(e^{-c''\|\bn\|^{-1/2}\gb}\big)^{|\btau\cap\bnsquare|} \ind_{\{\bn \in\btau\}}\Big]\,.
\end{equation}

\smallskip
In both cases \eqref{eq:Afrachom:m1neq0}-\eqref{eq:Afrachom:ga>1:m1neq0}, we recognize the partition function of a homogeneous gPS model with parameter $-c\|\bn\|^{-1/2}\gb$, for some $c>0$. The following result (proven in Appendix~\ref{section:appendixA}) gives an estimate of this partition function.
\begin{lemma}\label{lem:zhom}
For any $\ga>0$, fix some $0<\ga_-<\ga$. Then there exist $C_{\rcntc{cst:propshift1}}, C_\cntc>0$ and $u_0>0$ such that for any $u\in(0,u_0)$ and $\bn\in\N^2$, one has
\begin{equation}\label{eq:lem:zhom}
Z_{\bn,-u} \,\leq\, C_{\ref{cst:propshift1}} \frac{K(\|\bn\|)}{u^2} \;+\; \bP(\bn\in\btau)\,e^{-C_{\arabic{cst}}u\|\bn\|^{\min(\ga_-,1)}}\,.
\end{equation}
\end{lemma}
We now conclude the proof of Lemma~\ref{lem:Afrac2} by applying this result with $u:=c\|\bn\|^{-1/2}\gb$ (which decays to 0 as $\gb\searrow0$), separating again the cases $\ga\in (1/2,1]$ and $\ga>1$.

\smallskip
\noindent \emph{Case $\ga\in(1/2,1]$.} We apply Lemma~\ref{lem:zhom} to \eqref{eq:Afrachom:m1neq0}. Notice that, thanks to \eqref{eq:gdtilt:m1neq0}, we have
\begin{equation}
u \,\|\bn\|^{\ga_-} \,\geq\, \gb^{2\frac{\ga-\ga_-}{2\ga-1} \,- \,\eps \frac{2\ga_--1}{2\ga}\,+\,c_1\,\eps^{3/2}}\,.
\end{equation}
Provided that $\eps$ is sufficiently small (so that $\eps \frac{2\ga-1}{2\ga} > 2c_1\,\eps^{3/2}$), we can choose $\ga_-$ sufficiently close to $\ga$ (depending on $\eps$) so that the exponent is negative; thus $u\|\bn\|^{\ga_-}$ goes to infinity as a power of $\gb$. In particular the second term in \eqref{eq:lem:zhom} decays much faster than the first one (which decays at most polynomially in $\gb$, recall \eqref{eq:gdtilt:m1neq0}). Hence, Lemma~\ref{lem:zhom} and  \eqref{eq:gdtilt:m1neq0} give for any $\gb$ sufficiently small,
\begin{equation}
Z_{\bn,-c'\|\bn\|^{-1/2}\gb} \,\leq\, C_{\cntc} \frac{K(\|\bn\|)}{\|\bn\|^{-1}\,\gb^2} \,\leq\, \frac{C_\cntc}{\|\bn\|^{2-\ga+\eps C_\cntc}}\,.
\end{equation}
Recollecting \eqref{eq:Afracexp:m1neq0} and \eqref{eq:Afrachom:m1neq0}, this finally proves the lemma in the case $\ga\in(1/2,1]$.

\smallskip
\noindent
\emph{Case $\ga>1$.} 
We apply Lemma~\ref{lem:zhom} to \eqref{eq:Afrachom:ga>1:m1neq0}. Recalling \eqref{eq:gdtilt:ga>1:m1neq0}, we notice that  $u\|\bn\|\geq C_\cntc \gb^{-\eps/4}$ diverges as a power of $\gb$: the second term in \eqref{eq:lem:zhom} decays again much faster than the first one as $\gb\searrow0$, so Lemma~\ref{lem:zhom} and \eqref{eq:gdtilt:ga>1:m1neq0} give for any $\gb$ sufficiently small,
\begin{equation}
Z_{\bn,-c''\|\bn\|^{-1/2}\gb} \,\leq\, C_{\cntc} \frac{K(\|\bn\|)}{\|\bn\|^{-1}\,\gb^2} \,\leq\, \frac{C_\cntc}{\|\bn\|^{\ga+\eps C_\cntc}}\,.
\end{equation}
 With \eqref{eq:Afracexp:m1neq0} and \eqref{eq:Afrachom:ga>1:m1neq0}, this concludes the proof of the lemma in the case $\ga>1$.

\medskip
{\bf \noindent Case $m_1=0$, $m_4>m_2^2$.} When $m_1=0$, one could push the expansion in \eqref{eq:taylor:Q2:m1neq0} to the third order to find a dominating term $\gd\gb^2 m_3^2$. Assuming $m_3\neq0$ and adjusting the definition of~$h$ (thus of $\gd$), this would prove the fractional moment estimates with $h=\gb^{\max(\frac{4\ga}{2\ga-1},4)+\eps}$, hence a shift of the critical point of order $\gb^{\max(\frac{4\ga}{2\ga-1},4)+\eps}$. However when $m_1=0$, we can define another change of measure which yields the same estimates on $A_n$ without the additional assumption $m_3\neq0$. Indeed, let us define for all $\bn\in\N^2$ and $\gd\in[0,\gd_0)$,
\begin{equation}
\frac{\dd \bbP'_{\bn,\gd}}{\dd \bbP} (\go) \,=\, \frac{e^{-\gd\sum_{i_1=1}^{\bn^{(1)}} \hgo_{i_1}^2 - \gd\sum_{i_2=1}^{\bn^{(2)}} \bgo_{i_2}^2}}{R(\gd,0)^{\|\bn\|/2}}\,,
\end{equation}
where we define
\begin{equation}\label{eq:def:Q}
R(\gd,\gb)\,:=\, \bbE\big[e^{\gb \,\go_{\bone} - \gd \,\hgo_1^2-\gd\,\bgo_1^2}\big]\,,
\end{equation}
This change of measure is a \emph{quadratic} i.i.d tilt of the sequences $\hgo$ and $\bgo$. Similarly to \eqref{eq:def:Q:m1neq0}, it is well-defined for $\gb\in[0,\gb_1)$ and $\gd\in[0,\gd_0)$, for some $\gb_1, \gd_0>0$ (notice that here we directly choose a non-positive tilt). As in the previous case, H\"{o}lder's inequality yields for any $\bn\in\N^2$,
\begin{equation}\label{eq:holder}
A_{\bn}
\,\leq\, \big(\bbE'_{\bn,\gd}\big[Z^{\quen}_{\bn}\big]\big)^\gh \Big(R(\gd,0)^\gh R\big(-\gd\gh/(1-\gh),0\big)^{1-\gh}\Big)^{\|\bn\|/2}.
\end{equation}
Moreover a Taylor expansion of $R(\gd,\gb)$ to the third order in $(\gd,\gb)$ yields

\begin{equation}\label{eq:taylor:Q}
\begin{aligned}R(\gd,\gb)&=1-2\gd m_2+\frac{\gb^2m_2^2}2 + \gd^2(m_4+m_2^2) + \frac{\gb^3m_3^2}6\\
&\quad-\gb^2\gd m_2m_4 + \gb\gd^2 m_3^2 - \gd^3(\frac{m_6}3+m_2m_4) + O(\,\cdot^4)\,,
\end{aligned}
%
\end{equation}
with $O(\,\cdot^4):=O(\gb^4)+O(\gb^3\gd)+\cdots+O(\gd^4)$, where we used $m_1=0$. Thus,
\begin{equation}
R(\gd,0)^\gh R\big(-\gd\gh/(1-\gh),0\big)^{1-\gh} \,=\, 1 + \gd^2(m_4-m_2^2)\frac{\gh}{1-\gh} + O(\gd^3)\,,
\end{equation}
which is bounded by $\exp(C_\cntc\,\gd^2)$ for some $C_{\arabic{cst}}>0$, uniformly in $\gd\in(-\gd_1,\gd_1)$ (recall that we assumed $m_4>m_2^2$). Therefore, letting $\gd:=\|\bn\|^{-1/2}>0$, we have
\begin{equation}\label{eq:Afracexp}
A_{\bn} \;\leq\; C_\cntc \big(\bbE'_{\bn,\gd}\big[Z^{\quen}_{\bn}\big]\big)^\gh  \,,
\end{equation}
and similarly to \eqref{eq:etilte:1:m1neq0}, we write with a direct computation
\begin{equation}\label{eq:etilte:1}
\bbE'_{\bn,\gd}\big[Z_{\bn}^{\quen}\big] \,=\, \bE\bigg[\Big(e^{h}\frac{R(\gd,\gb)}{R(\gd,0)R(0,\gb)}\Big)^{|\btau\cap\bnsquare|} \, \ind_{\{\bn \in\btau\}}\bigg]\, .
\end{equation}
A Taylor expansion of $\mathrm{Frac_2}(\gd,\gb):=\frac{R(\gd,\gb)}{R(\gd,0)R(0,\gb)}$ to the third order in $(\gd,\gb)$ yields
\begin{equation}\label{eq:taylor:Q2} 
\mathrm{Frac_2}(\gd,\gb)= 1 - \gd\gb^2m_2(m_4-m_2^2) + O(\gb^p) + O(\gd\gb^3) + O(\gd^2\gb) + O(\gd^3)\,,
\end{equation}
where we can push the expansion to any order $p\in\N$ in $\gb$. Notice that this is very similar to \eqref{eq:taylor:Q2:m1neq0} above, with a leading term of order $\gb\gd^2$ instead (and $m_2(m_4-m_2^2)>0$ under our assumptions). By duplicating all arguments above, and recalling $k^{(1-\eps^2)}\leq\|\bn\|\leq 2k$, \eqref{eq:k}, Theorem~\ref{thm:hom} and our specific choice of $h$ in \eqref{eq:gD}, we finally obtain for some $c>0$ and any $\ga>1/2$,
\begin{equation}\label{eq:Afrachom}
\bbE'_{\bn,-\|\bn\|^{-1/2}}\big[Z_{\bn}^{\quen}\big] \,\leq\, \bE\Big[\big(e^{-c\|\bn\|^{-1/2}\gb^2}\big)^{|\btau\cap\bnsquare|} \ind_{\{\bn \in\btau\}}\Big] \,=\, Z_{\bn,-c\|\bn\|^{-1/2}\gb^2} \,.
\end{equation}
Eventually, we apply Lemma~\ref{lem:zhom} (with $u=c\|\bn\|^{-1/2}\gb^2$) to estimate this homogeneous partition function, and we conclude the proof by recollecting \eqref{eq:Afracexp} as in the case $m_1\neq0$ (we do not write the details again).

} 

\section{Disorder relevance: shift of the critical point when $\Pb=\Pbspinx$} \label{section:shift:m4=1}

In this section we prove the lower bound in Theorem~\ref{thm:rel:shift:m4=1} ---recall that we assume $m_1=0$ and $m_4=m_2^2$, so $\bbP=\Pbspinx$ for some $x>0$--- and in particular we discuss how the estimate of the fractional moment ({\it i.e.} Lemma~\ref{lem:Afrac2}) have to be adapted. Notice that we can assume $\bbP=\Pbspin$, that is $x=1$, without loss of generality ---one only has to replace $\gb$ with $\gb x^2$ in the partition function.

We can reproduce exactly the first part of the proof above, namely the coarse-graining procedure whose core is Lemma~\ref{lem:rel:shift}, and we can use Lemma~\ref{lem:Afrac1} as it is. On the other hand, the change of measure argument needs important adaptation. In the case of an i.i.d.\ disorder, what plays the role of Lemma~\ref{lem:Afrac2} is \cite[Prop.~4.2]{BGK}: there, the estimates of fractional moments $A_{\bi}$ use an i.i.d.\ tilt of the disorder, but only along an \emph{extended diagonal}, {\it i.e.}
\begin{equation}\label{eq:def:Jbn}
J_\bn\,:=\,\big\{\bi\in \llbracket \mathbf{1}, \bn \rrbracket \,;\,|\bi^{(1)}-\bi^{(2)}|\leq 2 \ell_\bn\big\}.
\end{equation}
The width $\ell_\bn$ is chosen depending on $\ga>1$, so that the renewal process $\btau$ is very unlikely to deviate from the diagonal by more than $\ell_\bn$ (see \cite[Thm.~A.5]{BGK}, or \cite[Thm.~4.2]{BL18} for a more general statement):
\begin{equation}\label{eq:def:lbn}
\ell_\bn \,:=\, \bigg\{\!\begin{array}{ll}
(\bn^{(1)})^{(1+\eps^2)/\ga} & \text{if } \ga\in(1,2]\,,\\
C_\cntc \sqrt{\bn^{(1)}\log \bn^{(1)}} & \text{if } \ga>2\,,
\end{array}
\end{equation}
(notice that $|J_\bn|\leq 2 \bn^{(1)} \ell_\bn \ll \bn^{(1)}\bn^{(2)}$ when $\bn^{(1)}\approx\bn^{(2)}$). In our setting the non-independent disorder adds some technicalities to this method, but they can be handled if we restrain ourselves to $\Pbspin$. 
We prove the following result, which plays the role of Lemma~\ref{lem:Afrac2} in the case $\bbP=\Pbspin$.
\begin{proposition}\label{prop:BGK4.2} Assume $\ga>1$, $\bbP=\Pbspin$, recall $k=1/\tf(0,h)$, and define $\ell_\bn$ as in \eqref{eq:def:lbn}. Then there exist $h_0>0$ and $L_{\cntf}$ such that for any $h\in(0,h_0)$ and $\sqrt{k}\leq\bn^{(1)}\leq k$, $\bn^{(1)}\leq\bn^{(2)}\leq\bn^{(1)}+\ell_\bn$, one has
\begin{equation}
A_\bn\;\leq\; \left\{\!\begin{array}{ll}
L_{\arabic{svf}}(k) \, k^{-(1+\eps^2)\gh/\ga} & \text{if }\ga\in(1,2]\,, \\
L_{\arabic{svf}}(k) \, k^{-(\ga-1)\gh/2} & \text{if }\ga>2\,.
\end{array}\right.
\end{equation}
\end{proposition}
Note that these are exactly the same estimates on $A_\bn$ as in \cite[Prop.~4.2]{BGK}. Once plugged in the computations of $\gr_1$, $\gr_2$ and $\gr_3$ from Lemma~\ref{lem:rel:shift}, they give the same lower bound for the shift of the critical point as in \cite[Thm 1.4]{BGK}, which is
\begin{equation}
h_c(\gb)\,\geq\,\bigg\{\!\begin{array}{ll}
\gb^{\frac{2\ga}{\ga-1}+\eps} & \text{if } \ga\in(1,2]\,,\\
\gb^4|\log\gb|^{-6} & \text{if }\ga>2\,.
\end{array}
\end{equation}
This proves the left inequality in Theorem~\ref{thm:rel:shift:m4=1}. We do not write the details here, because once we have the estimates on $A_\bn$ from Proposition~\ref{prop:BGK4.2}, the computations of $\gr_1,\gr_2$ and $\gr_3$ are the same as in \cite{BGK} and do not depend on the setting of disorder.\smallskip

\begin{proof}[Proof of Proposition~\ref{prop:BGK4.2}]
This follows the same scheme as \cite[Prop. 4.2]{BGK}. 
We define the same change of measure on $J_\bn$ as \cite[(4.18)]{BGK}, that is
\begin{equation}
\frac{\dd \ol \bbP_{\bn,\gd}}{\dd \bbP} (\go) \,:=\, \frac{\prod_{\bi\in J_\bn} e^{-\gd\go_\bi}}{\Qb_{J_\bn}(\gd)}\,,\qquad\text{where } \;\Qb_{J_\bn}(\gd)\,:=\, \bbE\Big[\prod_{\bi\in J_\bn} e^{-\gd\go_\bi}\Big]\,.
\end{equation}
Applying H\"{o}lder's inequality to $A_\bn$ similarly to \eqref{eq:holder}, we have
\begin{align}
\nonumber A_{\bn}&\leq \big(\ol \bbE_{\bn,\gd}\big[Z^{\quen}_{\bn}\big]\big)^\gh \Bigg(\ol \bbE_{\bn,\gd}\Bigg[\bigg(\frac{\dd \ol \bbP_{\bn,\gd}}{\dd \bbP} (\go)\bigg)^{-1/(1-\gh)}\Bigg]\Bigg)^{1-\gh}\\
\label{eq:holder:m4=1} &= \big(\ol \bbE_{\bn,\gd}\big[Z^{\quen}_{\bn}\big]\big)^\gh \; \Qb_{J_\bn}(\gd)^\gh \; \Qb_{J_\bn}\big(-\gd\gh/(1-\gh)\big)^{1-\gh}.
\end{align}
Let us fix $\gd:=(\bn^{(1)}\ell_\bn)^{-\frac{1}{2}(1+\eps^3)}>0$ ---this is the same as in \cite{BGK} where we added the power $(1+\eps^3)$ to avoid technicalities. We claim the following:

\begin{lemma}\label{lem:QJn}
There exists $C_\cntc>0$ and $\gd_1>0$ such that if $\gd^2\bn^{(1)} \ell_{\bn} \leq \gd_1$, then
\begin{equation}
1\,\leq\, \Qb_{J_\bn}(\gd) \,\leq\, C_{\arabic{cst}}\;.
\end{equation}
\end{lemma}

\begin{proof}[Proof of Lemma~\ref{lem:QJn}]
The lower bound follows directly from Jensen's inequality, so let us focus on the upper bound. For all $1\leq i \leq \bn^{(1)}$, we define
\begin{equation}\label{eq:defJb}
J_\bn(i)\,:=\,\{j\in\llbracket1,\bn^{(2)}\rrbracket;\, (i,j)\in J_\bn\}\,,\qquad \text{and} \quad \gs_i=\gs_i(\bgo):=\sum_{j\in J_\bn(i)} \bgo_j\,,
\end{equation}
in particular $|\gs_i|\leq|J_\bn(i)|\leq 2\ell_\bn$. Then, computing first the expectation conditionally on $\bgo$, we have
\begin{equation}\begin{aligned}
\Qb_{J_\bn}(\gd) &\,=\, \bbE\bigg[\exp\Big( -\gd \sum_{i=1}^{\bn^{(1)}} \hgo_i \,\gs_i(\bgo) \Big)\bigg] \,=\, \bbE\bigg[ \prod_{i=1}^{\bn^{(1)}} \cosh\big( \gd \,\gs_i(\bgo) \big) \bigg],
\end{aligned}\end{equation}
where we used that $\hgo$ is a sequence of independent variables, is independent from $\bgo$, and that $\bbE[e^{-x\hat \go_i}] = \cosh(-x \hat \go_i)= \cosh(x)$ for all $x\in\R$. Notice that for all $x\in\R$, one has $\cosh(x)\leq e^{\frac{x^2}{2}}$: we therefore have
\begin{equation}\begin{aligned}
\Qb_{J_\bn}(\gd) &\;\leq\; \bbE\bigg[ \exp\bigg( \sum_{i=1}^{\bn^{(1)}} \frac{\gd^2 \gs_i^2}{2} \bigg) \bigg] \;\leq\; \prod_{i=1}^{\bn^{(1)}} \bbE\bigg[ \exp\bigg( \frac{1}{2}\bn^{(1)} \gd^2 \gs_i^2 \bigg) \bigg]^{1/\bn^{(1)}}
\end{aligned}\end{equation}
by Cauchy-Schwarz inequality. Here $\gs_i$ is a sum of $|J_\bn(i)|$ i.i.d. bounded variables, so $\frac{\gs_i}{\sqrt{|J_\bn(i)|}}$ converges in distribution to some standart gaussian $Z\sim \cN(0,1)$ as $|J_\bn(i)|\to\infty$. This implies
\begin{equation}
\Qb_{J_\bn}(\gd) \;\leq\; C_\cntc+ C'_{\arabic{cst}}\,\bbE\Big[ \exp\Big( \bn^{(1)} \gd^2 |J_\bn(i)| Z^2 \Big) \Big]\end{equation}
uniformly in $\bn$ (where $C_{\arabic{cst}}$ handles the indexes $i$ for which $|J_\bn(i)|$ is small). Noticing that $\bn^{(1)} \gd^2 |J_\bn(i)|\leq 2\,\gd^2 \bn^{(1)} \ell_\bn \leq 2\gd_1$, and that $Z^2$ has some finite exponential moments, this concludes the proof of the lemma.
\end{proof}

Note that $\gd^2 \bn^{(1)} \ell_\bn\to0$ as $\bn^{(1)}\to\infty$, so this lemma implies that the last two factors in the right-hand side of  \eqref{eq:holder:m4=1} are uniformly bounded. It remains to deal with the expectation in~\eqref{eq:holder:m4=1}. One has
\begin{equation}\label{eq:esptilt:m4=1:0}\begin{aligned}
\ol \bbE_{\bn,\gd}\big[Z^{\quen}_{\bn}\big] &\,=\, \bE\bigg[\ol \bbE_{\bn,\gd}\Big[e^{\sum_{\bi\in \llbracket 1, \bn \rrbracket } (\gb\go_\bi-\gl(\gb)+h)\ind_{\{\bi\in\btau\}}}\Big]\ind_{\{\bn \in\btau\}}\bigg]\\
&\,=\, \bE\Bigg[\frac{e^{h|\btau\cap\bnsquare|}}{\Qb_{J_\bn}(\gd)} \bbE \bigg[e^{-\gl(\gb)|\btau\cap\bnsquare|}\prod_{\bi\in \llbracket 1, \bn \rrbracket }e^{ \gb\go_\bi\ind_{\{\bi\in\btau\}} - \gd\go_\bi \ind_{\{\bi\in J_\bn\}}} \bigg]\ind_{\{\bn \in\btau\}}\Bigg],
\end{aligned}\end{equation}
Under the assumptions of Proposition~\ref{prop:BGK4.2} we have $|\btau\cap\bnsquare| \leq \|\bn\|\leq C_\cntc h^{-1}$  (recall Theorem~\ref{thm:hom}), so this implies that $e^{h|\btau\cap\bnsquare|}$ is bounded by some uniform constant (and so is $\Qb_{J_\bn}(\gd)^{-1}$). 

Let us fix a realization of $\btau$ with $\bn\in\btau$, and let $B_{\bn,\btau}^{\gb,\gd}$ be the expectation over $\Pb$ in \eqref{eq:esptilt:m4=1:0}. Recall the definition of $J_\bn(i)$ and $\sigma_i$ in Lemma~\ref{lem:QJn}, and define
\begin{equation}
\gsb_i\,=\,\gsb_i(\bgo)\,:=\,\sum_{j=1}^{\bn^{(2)}} \bgo_j \ind_{\{(i,j)\in\btau\}} \,,
\end{equation}
and notice  that $|\gsb_i|=1$ if $i\in\llbracket 1,\bn^{(1)}\rrbracket\cap\btau^{(1)}$, and $\gsb_i=0$ otherwise.
Again, as for $\Qb_{J_\bn}$, computing the expectation first conditionally on $\bgo$ gives
\begin{equation}\label{eq:esptilt:m4=1}\begin{aligned}
B_{\bn,\btau}^{\gb,\gd} &=\bbE \bigg[e^{-\gl(\gb)|\btau\cap\bnsquare|} \exp\Big( \sum_{i=1}^{\bn^{(1)}} \hgo_i( \gb \gsb_i(\bgo) - \gd \gs_i(\bgo))\Big) \bigg]\\
&=\;\bbE \bigg[e^{-\gl(\gb)|\btau\cap\bnsquare|} \prod_{i=1}^{\bn^{(1)}} \cosh\big( \gb \gsb_i(\bgo) - \gd \gs_i(\bgo)\big) \bigg].
\end{aligned}\end{equation}
Notice that $\gl(\gb)=\log \cosh(\gb)$ when $\bbP=\Pbspin$. Using that $\cosh$ is even and that $\cosh 0=1$, we get that 
\begin{equation}\begin{array}{lll}\displaystyle
e^{\gl(\gb)|\btau\cap\bnsquare|}&=\; \displaystyle \prod_{i=1}^{\bn^{(1)}} \cosh(\gb \ind_{\{i\in\btau^{(1)}\}}) \; =\; \displaystyle \prod_{i=1}^{\bn^{(1)}} \cosh(\gb \gsb_i )\,,
\end{array}\end{equation}
where we used that $\gsb=\pm1$ when it is not equal to $0$.
Then \eqref{eq:esptilt:m4=1} becomes
\begin{equation}\label{eq:esptilt:m4=1:2}\begin{aligned}
B_{\bn,\btau}^{\gb,\gd} \;& =\; \bbE \bigg[\prod_{i=1}^{\bn^{(1)}} \frac{\cosh\big( \gb \gsb_i  - \gd \gs_i \big)}{\cosh(\gb \gsb_i )} \bigg]\\
& =\; \bbE \bigg[\prod_{i=1}^{\bn^{(1)}} \Big(\cosh (\gd \gs_i) - \sinh (\gd \gs_i) \tanh (\gb \gsb_i) \Big) \bigg],
\end{aligned}\end{equation}
where we simply used trigonometric identities to expand $\cosh( \gb \gsb_i  - \gd \gs_i)$. Notice that $\sinh (\gd \gs_i) \tanh (\gb \gsb_i) = \sinh (\gd \gs_i  \gsb_i) \tanh (\gb)$ (because $\gsb\in\{-1,0,1\}$ and $\sinh,\tanh$ are antisymmetric), and recall that $|\gd\gs_i|\ll (\ell_\bn/\bn^{(1)})^{1/2}\to 0$ as $\bn^{(1)}\to\infty$. Therefore, some standard Taylor expansions give for $\bn^{(1)}$ sufficiently large,
\begin{equation}\begin{aligned}
\cosh(\gd \gs_i) &\leq 1 + (\gd\gs_i)^2,\\
- \sinh (\gd \gs_i  \gsb_i) &\leq -\gd \gs_i \gsb_i + (\gd \gs_i  \gsb_i)^2 \leq -\gd \gs_i \gsb_i + (\gd \gs_i)^2.
\end{aligned}\end{equation}
Finally, for $\bn^{(1)}$ sufficiently large, one has for all $\bgo\in\{-1,1\}^\N$,
\begin{equation}
\cosh (\gd \gs_i) - \sinh (\gd \gs_i) \tanh (\gb \gsb_i) \;\leq\; 1 + 2(\gd \gs_i)^2 -  \gd \tanh(\gb) \gs_i \gsb_i\, ,
\end{equation}
(where we also used $\tanh(\gb)\leq1$). Since all the factors in \eqref{eq:esptilt:m4=1:2} are positive for $\bn^{(1)}$ sufficiently large, 
we get
\begin{equation}\label{eq:esptilt:m4=1:3}\begin{aligned}
B_{\bn,\btau}^{\gb,\gd}\, & \leq\; \bbE \bigg[\prod_{i=1}^{\bn^{(1)}} \big(1 + 2(\gd \gs_i)^2 - \gd \tanh(\gb) \gs_i \gsb_i \big) \bigg]\\
& \leq\; \bbE \Big[\exp\Big(\sum_{i=1}^{\bn^{(1)}} 2 (\gd \gs_i)^2 - \gd \tanh(\gb) \gs_i \gsb_i\Big) \Big] \, .
\end{aligned}\end{equation}
Let us define $\gs'_i := \sum_{j\in J_\bn(i)} \bgo_j \ind_{\{(i,j)\notin\btau\}}$, which is a small modification of $\gs_i$ (we removed $\bar \sigma_i$ if it appears in the sum in $\sigma_i$),
so that we have
\begin{equation}\begin{aligned}
\gs_i\gsb_i&\;=\; \gs'_i \gsb_i + \bigg(\sum_{j\in J_\bn(i)} \bgo_i \ind_{\{(i,j)\in\btau\}} \bigg) \bigg(\sum_{j=1}^{\bn^{(2)}} \bgo_i \ind_{\{(i,j)\in\btau\}} \bigg)\\
&\;=\; \gs'_i \gsb_i + \ind_{\{\exists j\in J_\bn(i);\, (i,j)\in\btau\}}.
\end{aligned}\end{equation}
Hence, we get that 
\begin{equation}
B_{\bn,\btau}^{\gb,\gd} \;\leq\; e^{-\gd\tanh(\gb) |\btau\cap J_{\bn}|} \hspace{1pt} \bbE\bigg[\exp\bigg( 4\sum_{i=1}^{\bn^{(1)}} \gd^2 \gs_i^2 \bigg)\bigg]^{1/2}  \hspace{1pt} \bbE\bigg[\exp\bigg( -2\gd \tanh(\gb)\!\sum_{i=1}^{\bn^{(1)}} \gs'_i \gsb_i \bigg)\bigg]^{1/2}.
\end{equation}
where we applied Cauchy-Schwarz inequality.
In the proof of Lemma~\ref{lem:QJn} we already showed that the second factor is bounded. We claim that the third factor is also bounded.

\begin{lemma}\label{lem:azuma}
There exists $\gd_2>0$ such that if $|\gd| \sqrt{\bn^{(1)} \ell_{\bn} \log(\bn^{(1)})} \leq \gd_2$, and if $\beta\leq 1$, then
\begin{equation}
1\,\leq\, \bbE\bigg[\exp\bigg( 2 \gd \tanh( \gb) \sum_{i=1}^{\bn^{(1)}} \gs'_i\gsb_i \bigg)\bigg] \,\leq\,2\, .
\end{equation}
\end{lemma}

We prove this Lemma at the end of this section. 
Since $\gd \sqrt{\bn^{(1)} \ell_{\bn} \log(\bn^{(1)})} \to 0$, the assumption of the lemma is clearly satisfied by $-\gd$ if $\bn^{(1)}$ is large enough:
going back to \eqref{eq:esptilt:m4=1:0} (and writing $\tanh(\gb)>\gb/2$ for small $\gb>0$), we finally obtain
\begin{equation}
\ol \bbE_{\bn,\gd}\big[Z^{\quen}_{\bn}\big] \;\leq\; C_\cntc \, \bE\Big[ e^{-\gd\gb|\btau\cap J_\bn|/2} \ind_{\{\bn \in\btau\}}\Big],
\end{equation}
which is the same as \cite[(4.25)]{BGK}. Then we finish the proof of Proposition~\ref{prop:BGK4.2} by applying \cite[Lemma 4.3]{BGK} and Lemma~\ref{lem:zhom} (we do not write the details here because it is a replica of \cite{BGK}).

\end{proof}

\begin{proof}[Proof of Lemma~\ref{lem:azuma}]
The left inequality is a straight consequence to Jensen's inequality, so we focus on the upper bound. We introduce some notations specific to this lemma. Let $r:=|\btau\cap\bnsquare|$ and denote $(a_l,b_l)=\btau_l, 1\leq l\leq r$ (in particular $(a_r,b_r)=\btau_r=\bn$). Moreover for all $1\leq l \leq r$, denote $J_\bn(a_l) = \llbracket c_l,d_l \rrbracket$. Then we can write
\begin{equation}\begin{aligned}
\sum_{i=1}^{\bn^{(1)}} \gs'_i\gsb_i \;
& =\; \sum_{l=1}^r \bgo_{b_l} \bigg(\sum_{j=c_l}^{d_l} \bgo_j\ind_{\{(a_l,j)\notin\btau\}}\bigg)\\
&=\; \sum_{l=1}^r \bgo_{b_l} \bigg(\sum_{j=c_l}^{\min(d_l, b_l-1)} \bgo_j \;+ \sum_{j=\max(b_l+1,c_l)}^{d_l} \bgo_j\bigg),
\end{aligned}\end{equation}
where we parted the sum according to indices of $\bgo$ before and after $b_l$. Let us define $X_0=Y_0=0$ and for all $1\leq t\leq r$,
\begin{equation}
X_t := \sum_{l=1}^t \bgo_{b_l} \bigg(\sum_{j=c_l}^{\min(d_l, b_l-1)} \bgo_j\bigg) 
\qquad\text{and}\quad Y_t := \sum_{l=1}^t \bgo_{b_l} \bigg(\sum_{j=\max(b_l+1,c_l)}^{d_l} \bgo_j\bigg),
\end{equation}
so that $\sum_{i=1}^{\bn^{(1)}} \gs'_i\gsb_i = X_r+Y_r$, and by Cauchy-Schwarz inequality,
\begin{equation}\label{eq:CS:XY}
 \bbE\bigg[\exp\bigg( 2 \gd \tanh( \gb) \sum_{i=-\ell_\bn}^{\bn^{(1)}+\ell_\bn} \gs'_i\gsb_i \bigg)\bigg] \;\leq\;  \bbE\Big[\exp\big( 4 \gd \tanh( \gb)  X_r \big)\Big]^{1/2} \bbE\Big[\exp\big( 4 \gd \tanh( \gb)  Y_r \big)\Big]^{1/2}.
\end{equation}
We only treat the first factor, the second one being symmetric. We define $\cF_0$ the trivial $\gs$-algebra, and for all $1\leq t\leq r$,
\begin{equation}
\cF_t\;:=\;\gs\big(\{\bgo_j,\; 1\leq j \leq b_{t+1}-1\}\big),
\end{equation}
with $b_{r+1}:=\bn^{(2)}+1$: it is then easily checked that $(X_t)_{0\leq t\leq r}$ is a $(\cF_t)_{0\leq t\leq r}$-martingale. 
We also define a ``truncated'' version of $X$: for all $w>0$, $X_0^{(w)}:=0$ and for all $1\leq t\leq r$, we set
\begin{equation}
X^{(w)}_t \;:=\; \sum_{l=1}^t \bgo_{b_l} \bigg(\sum_{j=c_l}^{\min(d_l, b_l-1)} \bgo_j\bigg) \ind_{\big\{\big|\sum_{j=c_l}^{\min(d_l, b_l-1)} \bgo_j \big|\,\leq\, w \sqrt{\ell_n \log(\bn^{(1)})}\big\}}\,.
\end{equation}
Notice that $(X^{(w)}_t)_{0\leq t\leq r}$ is also a $(\cF_t)_{0\leq t\leq r}$-martingale, and it has bounded increments: $|X^{(w)}_t - X^{(w)}_{t-1}| \leq w \sqrt{\ell_n \log(\bn^{(1)})}$ for all $1\leq t\leq r$. Therefore we deduce from the Azuma-Hoeffding inequality that for all $u,w>0$,
\begin{equation}\label{eq:azuma}
\Pb\bigg(|X^{(w)}_r| > u \sqrt{r \ell_\bn \log(\bn^{(1)})}\bigg) \;\leq \; 2 \exp\Big(-\frac{u^2}{2 w^2}\Big).
\end{equation}
Moreover we have
\begin{equation}\begin{aligned}
\Pb\big(X_r \neq X^{(w)}_r \big) \;&=\; \Pb\bigg(\exists\, 1\leq l\leq r;\; \bigg|\sum_{j=c_l}^{\min(d_l, b_l-1)} \bgo_j \bigg| > w \sqrt{\ell_n \log(\bn^{(1)})}\bigg)\\
&\leq\; \sum_{l=1}^r \Pb\bigg(\bigg|\sum_{j=c_l}^{\min(d_l, b_l-1)} \bgo_j \bigg| > w \sqrt{\ell_n \log(\bn^{(1)})}\bigg),
\end{aligned}\end{equation}
so by Hoeffding's inequality, and because $|J_\bn(i)|\leq 2\ell_\bn$ and $r\leq \bn^{(1)}$,
\begin{equation}\label{eq:hoeffding}
\Pb\big(X_r \neq X^{(w)}_r \big) \;\leq\; \bn^{(1)} \exp\Big(-\frac{w^2}{4} \log (\bn^{(1)})\Big).
\end{equation}
We deduce from \eqref{eq:hoeffding} and \eqref{eq:azuma} that for all $u,w>0$,
\begin{equation}
\Pb\bigg(|X_r| > u \sqrt{r \ell_\bn \log(\bn^{(1)})}\bigg) \;\leq \; 2 \exp\Big(-\frac{u^2}{2 w^2}\Big) + \bn^{(1)} \exp\Big(-\frac{w^2}{4} \log (\bn^{(1)})\Big),
\end{equation}
in particular with $w^2= \frac{u\sqrt{2}}{\sqrt{\log(\bn^{(1)})}}$,
\begin{equation}\label{eq:azuma:2}\begin{aligned}
\Pb\bigg(|X_r| > u \sqrt{r \ell_\bn \log(\bn^{(1)})}\bigg) \;&\leq \; (2+\bn^{(1)}) \exp\Big(-\frac{u\sqrt{\log(\bn^{(1)})}}{2\sqrt{2}}\Big)\\
&\leq\; \exp(-Cu),
\end{aligned}\end{equation}
where $C>0$ is a constant sufficiently small such that this inequality holds for all $u\geq 1$ and $\bn^{(1)}\in\N$.
Hence, \eqref{eq:azuma:2} implies that $\frac{|X_r|}{\sqrt{r\ell_\bn\log(\bn^{(1)})}}$ is dominated under some coupling by $C'+Z$, with $C'$ a constant and $Z$ an exponential random variable with parameter $C$. In particular,
\begin{equation}\label{eq:azuma:3}
\bbE\Big[\exp\big( 4 \gd \tanh( \gb)  X_r \big)\Big] \;\leq\; \bbE\bigg[\exp\bigg( 4  |\gd|\gb\Big(C' + \sqrt{r\ell_\bn\log(\bn^{(1)})} Z  \Big)\bigg)\bigg],
\end{equation}
and  this can be made arbitrarily small if $|\gd| \gb \sqrt{r\ell_\bn\log(\bn^{(1)})}\leq |\gd|\sqrt{\bn^{(1)}\ell_\bn\log(\bn^{(1)})}$ is small, which concludes the proof of the lemma.
\end{proof}

\appendix

\section{Estimates on the homogeneous gPS model} \label{section:appendixA}

Here we prove Lemma~\ref{lem:zhom} and Lemma~\ref{lem:Afrac1}.

\begin{proof}[Proof of Lemma~\ref{lem:zhom}]
This Lemma is already proven for $\ga>1$ in \cite[Lemma 4.4]{BGK}, so we replicate the proof here for $\ga\in(0,1]$ (notice that this proof follows the lines of \cite[Prop.~A.2.]{DGLT09} in the context of the homogeneous PS model).

We decompose the partition function according to the number of renewal points until $\bn$:
\begin{equation}
Z_{\bn,-u} = \sum_{j=1}^{\|\bn\|} e^{-ju} \bP\big(\btau_j=\bn\big) = \sum_{j=1}^{\|\bn\|^{\ga_-}} e^{-ju} \bP\big(\btau_j=\bn\big) + \sum_{j=\|\bn\|^{\ga_-} +1}^{\|\bn\|} e^{-ju} \bP\big(\btau_j=\bn\big).
\end{equation}
For the first sum, (recall $\ga_-<\ga\leq1$) we use Proposition~\ref{prop:approx:btau:keme} to bound $\bP\big(\btau_j=\bn\big)$ by a constant times $j K(\|\bn\|)$, both for $\ga<1$ and $\ga=1$ (for the latter, notice that $b_j\ll \|\bn\|$). We obtain
\begin{equation}
\sum_{j=1}^{\|\bn\|^{\ga_-}\!\!} e^{-ju} \bP\big(\btau_j=\bn\big) \leq C_\cntc \frac{K(\|\bn\|)}{u^2} \sum_{j=1}^{\|\bn\|^{\ga_-}\!\!} ju^2 e^{-ju}\leq C_\cntc \frac{K(\|\bn\|)}{u^2},
\end{equation}
where we used a Riemann-sum approximation of the last sum to get that it is bounded by a constant times $\int_{\R_+} xe^{-x}\dd x=1$ for any $u\leq u_0$. For the terms $j>\|\bn\|^{\ga_-}$, we simply bound $j$ from below:
\begin{equation}
\sum_{j=\|\bn\|^{\ga_-}+1}^{+\infty} e^{-ju} \bP\big(\btau_j=\bn\big) \leq e^{-u \|\bn\|^{\ga_-}} \bP(\bn\in\btau),
\end{equation}
and the proof is complete.
\end{proof}

\begin{proof}[Proof of Lemma~\ref{lem:Afrac1}]
We first take care of the case $\ga\geq1$ with very rough estimates: if $\ga>1$, recalling Theorem~\ref{thm:hom}, there is some $c_\ga>0$ such that $\tf (0,h)\sim c_\ga h$ as $h\searrow0$. In particular $h|\btau\cap\bnsquare|$ is bounded uniformly in $h\in(0,h_1)$ and $1\leq \|\bn\| \leq 1/\tf(0,h)$. Thus,
\begin{equation}
Z_{\bn,h} \,=\, \bE[e^{h|\btau\cap\bnsquare|}\ind_{\{\bn\in\btau\}}]\,\leq\, C_\cntc \, \bP(\bn\in\btau)\,.
\end{equation}
Then we conclude with Proposition~\ref{prop:approx:btau:ga>=1}.

If $\ga=1$, then $\tf (0,h)\sim L_\ga(1/h) h$ as $h\searrow0$ where $L_\ga$ is slowly varying. Using the additional assumption $\|\bn\|\leq \tf(0,h)^{-(1-\eps^2)}$, $h|\btau\cap\bnsquare|$ is also bounded (it decays to 0) and the former upper bound still holds. Then we conclude with Proposition~\ref{prop:approx:btau:ga>=1}.\smallskip

Those rough estimates do not apply when $\ga\in(0,1)$, since we do not have uniform bounds on $h|\btau\cap\bnsquare|$. We follow the line of proof of \cite[Lem.~4.1]{DGLT09} for the PS model.
A first step in the proof consists in dealing with the indicator function in the partition function, to compare it with its free counterpart, see \eqref{eq:Znhfree} below. Let us define
\begin{equation}\label{eq:defTNM}
T_{\bn}\,:=\, \Big\{\bi \in \llbracket 1, \bn \rrbracket \,;\,\| \bi\| \leq \tfrac12 \|\bn\|\Big\}\,,
\end{equation}
the lower left half  of the rectangle $\llbracket 1, \bn \rrbracket $, where we recall that $\|\cdot\|$ is the $L^1$ norm on $\N^2$. Because, conditionally on $\bn\in \btau$, the time-reversed process $\tilde\btau$ in $\llbracket 1, \bn \rrbracket\setminus T_{\bn}$ starting from $\bn$ has same law as $\btau$ in $T_{\bn}$, the partition function is bounded with a Cauchy-Schwarz inequality by
\begin{equation}\label{eq:proof:Afrac1:2}
Z_{\bn,h} = \bE\Big[e^{ h\, |\btau \cap\llbracket 1, \bn \rrbracket |} \big| \bn\in\btau\Big] \bP\big( \bn \in\btau\big) \leq \bE\Big[e^{2 h \,|\btau \cap T_{\bn}| } \big|  \bn \in\btau\Big] \bP\big( \bn\in\btau\big) \, ,
\end{equation}
(if $\|\bn\|$ is even the anti-diagonal $\{\|\bi\|=\|\bn\|/2\}$ is counted twice, but the upper bound still holds). Let us define $X_{\bn}:=\sup\{ \bi \in T_{\bn}\cap\btau\}$ the last renewal point in $T_{\bn}$ (we take the supremum for the natural order on $\btau\subset\N^2$: recall that $\btau$ is strictly increasing on both coordinates). Because $|\btau \cap T_{\bn}| $ and $\ind_{\{\bn\in\btau\}}$ are independent conditionally to $\{X_{\bn}=\bi\}$, we can write:
\begin{equation}\label{eq:proof:Afrac1:3}
\bE\Big[e^{ 2 h  |\btau \cap T_{\bn} | } \Big| \bn \in\btau\Big]  = \!\sum_{\bi\in T_{\bn}}\! \bE\Big[e^{2 h |\btau \cap T_{\bn} | } \Big| X_{\bn}= \bi \Big] \bP\big(X_{\bn}=\bi \big| \bn\in\btau\big).
\end{equation}
then we can use the following Lemma, which is proven afterwards.

\begin{lemma}\label{lem:cond}
Assume $\ga\in (0,1)$.
There exists $C_\cntc>0$ such that for any $\bn\in\N$ and $\bi\in T_{\bn}$,
\begin{equation}
\bP\big(X_{\bn}=\bi\,,\, \bn \in\btau\big) \,\leq\,  C_{\arabic{cst}}\, L(\|\bn\|)^{-1} \, \|\bn\|^{-(2-\ga)} \bP\big(X_{\bn}=\bi \big)\,.
\end{equation}
\end{lemma}

Thanks to this Lemma and \eqref{eq:proof:Afrac1:3}, the inequality \eqref{eq:proof:Afrac1:2} becomes
\begin{equation}
\label{eq:Znhfree}
Z_{\bn,h}\,\leq\, \frac{C_{\arabic{cst}}\,L(\|\bn\|)^{-1}}{\|\bn\|^{2-\ga}}\bE\Big[e^{2 h \,|\btau \cap T_{\bn}| } \Big]\,,
\end{equation}
so Lemma~\ref{lem:Afrac1} will be proven once we show that the above expectation is bounded by a constant, uniformly for $\|\bn\| \leq 1/\tf(0,h)$. Notice that $\|\btau \| = ( \|\btau_{i}\|)_{i\geq 1}$ is a renewal process on $\N$, so we may write the upper bound
\begin{equation}
\label{eq:normtau}
\bE\Big[e^{2 h \,|\btau \cap T_{\bn}| } \Big] \,\leq\, \bE\Big[\exp \Big( 2 h \, \sum_{i=1}^{\|\bn\|/2} \ind_{\{i\in \|\btau\|\}} \Big)\Big]\,,
\end{equation}
which is the partition function of a homogeneous PS model with underlying univariate renewal $\|\btau\|$, which easily verifies $\bP(\|\btau_1\|=k)  \sim  L(k) k^{-(1+\ga)}$ as $k\to +\infty$. The right side of \eqref{eq:normtau} has already been studied in \cite[Lemma 4.1]{DGLT09} when $\ga\in(0,1)$, and we therefore get that the expectation is bounded by a constant uniformly in $\|\bn\| \leq 1/\tf(0,h)$.
\end{proof}

\begin{proof}[Proof of Lemma~\ref{lem:cond}]
Fix $\bn, \bi\in\N^2$. Recall and the definitions of $T_\bn$ in \eqref{eq:defTNM} and $X_\bn$, and that we assumed $\ga\in(0,1)$. We write:
\begin{equation}\label{eq:proof:cond:1}
\begin{aligned}
\bP\big(X_\bn=\bi\,,\, \bn\in\btau\big) 
&= \bP\big(\bi\in\btau\big) \bP \Big(\|\btau_1\|>\tfrac{1}{2}\|\bn\|-\|\bi\|\,,\,\bn-\bi\in\btau\Big)\\
&= \bP\big(\bi\in\btau\big) \sum_{\bj\in Q_\bi^\bn} \bP\big(\btau_1=\bj\big) \bP\big(\bn-\bi-\bj\in\btau\big),
\end{aligned}
\end{equation}
with
\begin{equation}
Q_\bi^\bn:= \Big\{\bj\in\N^2\,;\, \bj \preceq \bn -\bi \,,\, 
\|\bj\| > \tfrac{1}{2}\|\bn\| -\|\bi\|\Big\}.
\end{equation}
Then, we can use Proposition~\ref{prop:approx:btau:ga<1} to get that $\bP\big(\bn-\bi-\bj\in\btau\big)$ is bounded by a constant times $L(\| \bn-\bi-\bj\|)^{-1}\| \bn-\bi-\bj\|^{-(2-\ga)}$. We get that the sum in \eqref{eq:proof:cond:1} is bounded by a constant times
\begin{align*}
&\sum_{\bj\in Q_\bi^\bn} \bP\big(\btau_1 =\bj\big) L(\| \bn-\bi-\bj\|)^{-1} \| \bn-\bi-\bj\|^{-(2-\ga)} \\
&\quad\qquad\leq C_\cntc  L(\|\bn\|)^{-1} \| \bn\|^{-(2-\ga)} \sum_{\bj\in Q_\bi^\bn, \|\bj\| \leq \|n\|/4} \bP\big(\btau_1 =\bj\big)  \\
&\qquad\qquad + C_{\arabic{cst}} L(\|\bn\|) \| \bn\|^{-(2+\ga)}  \sum_{\boldsymbol{\ell} \leq \|n\|} L(\| \boldsymbol{\ell} \|)^{-1} \| \boldsymbol{\ell}\|^{-(2-\ga)} \, ,
\end{align*}
where we decomposed the sum according to whether $\|\bj\| \leq \|\bn\|/4$ or not. We used that  if $\|\bj\| \leq \|\bn\|/4$ then $\| \bn-\bi-\bj\| \geq \|\bn \|/4$, and if $\|\bj\| > \|\bn\|/4$ then we can bound $\bP(\bj\in \btau)$ uniformly thanks to \eqref{eq:interarrival:tau}, and we wrote  $\boldsymbol{\ell}:= \bn -\bi -\bj$.

Using that the last sum above is bounded by a constant times $L(\| \bn\|)^{-1} \| \bn\|^{\ga} $, we therefore get that \eqref{eq:proof:cond:1} is bounded by a constant times
\begin{align*}
&\bP ( \bi \in \btau)   L(\|\bn\|)^{-1} \| \bn\|^{-(2-\ga)} \Big( \bP\big(  \|\btau_1\| > \tfrac{1}{2}\|\bn\|-\|\bi\| \big) + L(\|\bn\|) \|\bn\|^{-\ga} \Big) \\
&\qquad \leq C_\cntc  L(\|\bn\|)^{-1} \| \bn\|^{-(2-\ga)}  \bP(\bi\in \btau) \bP\big(  \|\btau_1\| > \tfrac{1}{2}\|\bn\|-\|\bi\| \big)   \\
&\qquad \leq   C_{\arabic{cst}}  L(\|\bn\|)^{-1} \| \bn\|^{-(2-\ga)}  \bP( X_{\bn} =\bi)\, ,
\end{align*}
which concludes the proof.
\end{proof}

\section{Some properties of bivariate renewals}\label{section:appendixB}
We provide here some estimates on the bivariate renewal $\btau$ that we use (recall its inter-arrival distribution \eqref{eq:interarrival:tau}). They can be found in \cite[Appendix A]{BGK} or in \cite{B18} in a more general setting, and rely on the fact that $\btau$ is in the domain of attraction of a $\min(\ga,2)$-stable distribution. We define the \emph{scaling sequence} $a_n$,
\begin{equation}
a_n\,:=\,\psi(n)n^{1/\min(\ga,2)}\,,
\end{equation}
where $\psi$ is some slowly varying function (we do not detail it here, see \cite{B18, BGK} ; notice that if $\ga>2$, $\psi$ is a constant). We also define the \emph{recentering sequence} $b_n$,
\begin{equation}
b_n := 0  \ \ \text{if } \ga\in(0,1), \qquad
b_n:= n\, \mu(a_n) \ \ \text{if } \ga=1,   \qquad
b_n := n\,\mu \  \ \text{if } \ga>1.
\end{equation}
Here $\mu(x) = \bE[\btau^{(1)}_1\ind_{\{\btau^{(1)}_1 \leq \,x\}}]$ is the truncated moment, and $\mu=\lim_{x\to+\infty} \mu(x)$.
When $\ga=1$, notice that either $\mu<+\infty$ or $\mu(x)\to+\infty$ as a slowly varying function.

\begin{proposition}[\cite{B18}, Thm.~2.4]
\label{prop:approx:btau:keme} 
There exists some $C_\cntc>0$ such that
\begin{equation}
\bP\big(\btau_j=\bn+b_j\bone\big)\;\leq \; C_{\arabic{cst}} \, j\,\bP(\btau_1=\bn)\;,
\end{equation}
for any $j\in\N$ and $\bn\in\N^2$ such that $\|\bn\|\geq a_j$.
\end{proposition}

\begin{proposition}[\cite{B18}, Thms.~3.1-4.1]\label{prop:approx:btau:ga<1} Assume $\ga\in(0,1)$. There is some constant $C>0$ such that for all $\bn \in \N^2$,
\begin{equation}
\bP(\bn\in\btau) \leq C L(\|\bn\|)^{-1} \|\bn\|^{-(2-\ga)} \, .
\end{equation}
\end{proposition}
Note that this upper bound is sharp when $\bn^{(1)}$ and $\bn^{(2)}$ are of the same order ({\it i.e.} when $\bn$ is close to the diagonal: we refer to \cite[Cor.~3-B]{Will68} for a precise statement).

\begin{proposition}[\cite{B18}, Thms.~3.3-3.4 and 4.2-4.3]
\label{prop:approx:btau:ga>=1}
Assume $\ga\geq 1$, then there is a constant $C>0$ such that for any $\bn \in \N^2$,
\begin{equation}
\bP(\bn\in \btau) \leq  \frac{C}{ \mu(n) a_{n/\mu(n)}} \, ,
\end{equation}
with $n = \min(\bn^{(1)}, \bn^{(2)})$. When $\ga>1$, $\mu(n)$ can be replaced by $1$.
\end{proposition}
We stress that this upper bound is sharp when $\bn$ is close to the diagonal, we refer to \cite[Thms.~3.3-3.4]{B18} for the precise statements.
Also, notice that $n\mapsto \mu(n) a_{n/\mu(n)}$ is regularly varying with index $1/\min(\ga,2)$.

\smallskip
Let us also recall some results on intersections of renewals, either univariate or bivariate.

\begin{proposition}\label{prop:intertau}\mbox{}
\begin{enumerate}[leftmargin=*]
\item Let $\tau$ be a renewal in $\N$, with $\bP(\tau_1=n)=L(n)/n^{1+\ga}$, for  $\ga >0 $ and  $L$ slowly varying. Let $\tau'$ be an independent copy of $\tau$. Then $\tau\cap\tau'$ is also a renewal process, and it is terminating a.s. if $\sum_{n\geq1} \frac{n^{2\ga-2}}{L(n)^2}<\infty$ (in particular $\ga\leq1/2$ is necessary, while $\ga<1/2$ is sufficient). Otherwise it is persistent a.s..
\item Let $\btau$ be a renewal in $\N^2$, with inter-arrival distribution as in \eqref{eq:interarrival:tau}. Let $\btau'$ be an independent copy of $\btau$. Then $\btau\cap\btau'$ is also a renewal process, and it is terminating a.s. if $\ga<1$, or if $\ga=1$ and $\sum_{n\geq1}\frac1{nL(n)\mu(n)}<\infty$. Otherwise it is persistent a.s..
\end{enumerate}\end{proposition}

This comes from the following observation: $\tau\cap\tau'$ (resp. $\btau\cap\btau'$) is a univariate (resp. bivariate) renewal process, so $|\tau\cap\tau'|$ (resp. $|\btau\cap\btau'|$) is either infinite a.s., or a finite geometric variable. So one only has to estimate $\bE[|\tau\cap\tau'|]$ (resp. $\bE[|\btau\cap\btau'|]$) to determine if the intersection is persistent or terminating. For univariate processes, those estimate can be obtained thanks to the asymptotic behavior of the renewal mass function $ \bP(i\in\tau)$ (see \cite{Doney97} if $\ga\in(0,1)$, \cite{Erik70} if $\ga=1$, the case $\ga>1$ being simply the renewal theorem).  Regarding the bivariate case, it has recently been treated in \cite[Prop. A.3 and Rem. A.7]{BGK}, thanks to the estimates on the renewal mass function $ \bP(\bi\in\btau)$ collected in \cite{B18}.


\section{Computation of the second moment in the Gaussian case}\label{section:appendixC}
In this appendix we compute the upper bound of the second moment of the partition function under the assumption that $\hgo$, $\bgo$ have a centered Gaussian distribution, for $\gb\in(0,\gb_1)$ sufficiently small (recall Remark~\ref{rem:bornesup}). We can assume without loss of generality that $\hgo,\bgo\sim \cN(0,1)$ (by adjusting $\gb$). 
Recalling our computations from Section~\ref{subsection:secondmoment} and the decomposition from Proposition~\ref{prop:decomp}, we have to derive the correlations along a chain of points.
\begin{proposition}\label{prop:appC}
Let $(X_k)_{k\geq1}$ be a sequence of i.i.d. variables with distribution $\cN(0,1)$. Then we have for all $\gb\leq 1/2$, $\ell\in\N$,
\begin{equation}
\bbE\Big[e^{\gb(X_1X_2+X_2X_3+\ldots+X_\ell X_{\ell+1})}\Big]\;=\; \xi_1\xi_2\ldots\xi_\ell\;,
\end{equation}
where $\xi_0=1$ and for all $k\geq0$, $\xi_{k+1}=(1-\gb^2\xi_k^2)^{-1/2}$.
\end{proposition}
Note that the assumption $\gb\leq1/2$ is required for $\xi_k$ to be well-defined for all $k\geq1$.
\begin{proof} The proof relies on the following identity: if $X\sim\cN(0,1)$, then for all $t\in\R$, $\gb<1$ and $k\geq0$,
\begin{equation}\label{eq:appC:induction}
\bbE\Big[e^{\gb tX} e^{\gb^2\xi_k^2X^2/2}\Big] = \xi_{k+1} e^{\gb^2\xi_{k+1}^2t^2/2}\;,
\end{equation}
(this follows from a direct computation). Hence for any $\ell\geq1$, we condition the expectation over $(X_1,\ldots,X_\ell)$ to write
\begin{align*}
\bbE\Big[e^{\gb\sum_{k=1}^\ell X_kX_{k+1}}\Big]\;&=\; \bbE\Big[e^{\gb\sum_{k=1}^{\ell-1} X_kX_{k+1}} \hspace{1pt} \bbE\big[e^{\gb X_\ell X_{\ell+1}}\big| (X_1,\ldots,X_\ell)\big]\Big]\\
&=\; \bbE\Big[e^{\gb\sum_{k=1}^{\ell-1} X_kX_{k+1}} \hspace{1pt}  e^{\gb^2X_\ell^2/2}\Big]\;.
\end{align*}
Then, by conditionning over $(X_1,\ldots,X_{\ell-1})$ and applying \eqref{eq:appC:induction}, we obtain by induction,
\begin{align*}
\bbE\Big[e^{\gb\sum_{k=1}^\ell X_kX_{k+1}}\Big]\;&=\; \bbE\Big[e^{\gb\sum_{k=1}^{\ell-2} X_kX_{k+1}} \hspace{1pt} \bbE\big[e^{\gb X_{\ell-1} X_\ell} e^{\gb^2X_\ell^2/2} \big| (X_1,\ldots,X_{\ell-1}) \big]\Big]\\
&=\; \xi_1\hspace{1pt} \bbE\Big[e^{\gb\sum_{k=1}^{\ell-2} X_kX_{k+1}} e^{\gb^2\xi_1^2X_{\ell-1}^2/2}\Big]\\
&=\; \xi_1\ldots\xi_{\ell-1} \hspace{1pt} \bbE\Big[e^{\gb^2\xi_{\ell-1}^2X_{1}^2/2} \Big]\;=\; \prod_{k=1}^\ell \xi_k
\;,
\end{align*}
which concludes the proof.
\end{proof}
When $\gb\leq1/2$, one can prove that the sequence $(\xi_k)_{k\geq1}$ is non-decreasing and that it converges to $\xi_\infty:= \frac1{\gb\sqrt2}\sqrt{1-\sqrt{1-4\gb^2}}$. Moreover the application $f:x\mapsto (1-\gb^2x^2)^{-1/2}$ is convex, hence we have for all $k\geq1$,
\begin{equation}
\xi_{k+1}-\xi_k \;\leq\; f'(\xi_\infty)(\xi_k-\xi_{k-1})\;,
\end{equation}
with $f'(\xi_\infty)=\gb^2\xi_\infty^4 \leq C_\cntc \gb^2$ for some $C_{\arabic{cst}}>0$. Thus we can prove by induction,
\begin{equation}
\xi_{k} \;\leq\; \xi_1 + (\xi_1-1)(C_{\arabic{cst}}\gb^2 +\big(C_{\arabic{cst}}\gb^2)^2 +\ldots + (C_{\arabic{cst}}\gb^2)^{k-1}\big)\;.
\end{equation}
Assuming $\gb<\gb_1<C_{\arabic{cst}}^{-1/2}$, there exist $C_\cntc>0$ such that for all $k\geq1$,
\begin{equation}
\sum_{i=1}^k (C_1\gb^2)^i \;\leq\; \frac{C_1\gb^2}{1-C_1\gb^2}\;\leq\; C_{\arabic{cst}} \gb^2\;.
\end{equation}
Moreover a Taylor expansion yields $\xi_1-1=\frac12\gb^2+o(\gb^2)$. Hence there exist $\gb_1>0$ and $C_\cntc>0$ such that for all $\gb\in(0,\gb_1)$ and $k\geq1$, one has
\begin{equation}
\frac{\xi_k}{\xi_1} \;\leq\; 1 + C_{\arabic{cst}} \gb^4 \;\leq e^{C_{\arabic{cst}} \gb^4}\;.
\end{equation}
Finally, noticing that $e^{\gl(\gb)}=\xi_1$, we conclude
\begin{equation}
\label{estimchainGauss}
\bbE\Big[e^{\gb\sum_{k=1}^\ell X_kX_{k+1} -\ell\gl(\gb)}\Big] \;=\; \prod_{k=1}^{\ell}\frac{\xi_k}{\xi_1} \;\leq\; e^{C_{\arabic{cst}} \ell \gb^4}\; .
\end{equation}
The left-hand side in \eqref{estimchainGauss}
corresponds exactly to the contribution of a chain of length $\ell$ to the second moment of the partition function.
Going back to \eqref{eq:proof:L2bound:3}
and recalling that the expectation with respect
to intersection points gives a contribution $e^{|\bgn_n|(\gl(2\gb)-2\gl(\gb))}$
and that the isolated points do not contribute, we get that
\begin{align*}
\bbE\Big[\big(Z^{\gb,\quen,\textit{free}}_{\bn,0}\big)^2\Big]
\leq \bE_{(\btau,\btau')}\bigg[ e^{|\bgn_n|(\gl(2\gb)-2\gl(\gb))}  \prod_{m\in\N} e^{C_3 \gb^4 | \bgs_{n,m}|}\bigg]\, .
\end{align*}
Applying Cauchy-Schwarz inequality,
we obtain
\begin{equation}
\label{Gauss-almostthere}
\bbE\Big[\big(Z^{\gb,\quen,\textit{free}}_{\bn,0}\big)^2\Big]
\leq \bE_{(\btau,\btau')}\bigg[ e^{ 2 (\gl(2\gb)-2\gl(\gb)) |\bgn_n|} \bigg]^{1/2}  \bE_{(\btau,\btau')}\bigg[  e^{ 4 C_3 \gb^4 (|\btau^{(1)}_{\preceq \bn}\cap\btau'^{(1)}_{\preceq \bn}|+|\btau_{\preceq \bn}^{(2)}\cap\btau'^{(2)}_{\preceq \bn}|)}\bigg]^{1/2}\,,
\end{equation}
where we also used that  $\sum_{m} | \bgs_{n,m}| = |\mathfrak{S}_n| \leq  2\,(|\btau^{(1)}_{\preceq \bn}\cap\btau'^{(1)}_{\preceq \bn}|+|\btau_{\preceq \bn}^{(2)}\cap\btau'^{(2)}_{\preceq \bn}|)$,
recall~\eqref{eq:comparechain}.
Now, for $\alpha<1$, Proposition~\ref{prop:intertau}-(2)
gives that $\bgn = \btau \cap \btau'$ is a.s. finite;
in fact, $|\bgn|$ is a geometric random variable. 
Hence, if $\gb$ is small enough
the first factor in the r.h.s.\ of~\eqref{Gauss-almostthere}
is bounded by $2$.
Applying Cauchy-Schwarz inequality for the
other factor, we therefore end up with
\begin{align*}
\bbE\Big[\big(Z^{\gb,\quen,\textit{free}}_{\bn,0}\big)^2\Big]
\le 2\, \bE_{(\btau,\btau')}\bigg[  e^{ 8 C_3 \gb^4 |\btau^{(1)}_{\preceq \bn}\cap\btau'^{(1)}_{\preceq \bn}|}\bigg]^{1/2} \, ,
\end{align*}
and one can remove the exponent $1/2$ since the expectation on the r.h.s.\ is larger than $1$.
This proves the inequality~\eqref{momentGauss} claimed in Section~\ref{sec:hcub:nonopt} when $\ga<1$, and thus gives the expected $n_\gb$ and upper bound on the critical point shift.

If $\ga\geq1$, one concludes by proving that the first factor in the r.h.s. of \eqref{Gauss-almostthere} remains bounded as long as $\gb^2 N^{1/\min(\ga,2)}\sim1$, and the second factor does as long as $\gb^4 N\sim1$ (up to slowly varying factors). Hence we have $n_\gb\sim\gb^{-4}$, which gives the expected upper bound on the critical point shift. All the required estimates can be found in \cite{BGK, Erik70}, so we do not write all the details here.

\bibliographystyle{plain}
\bibliography{biblio.bib}

\end{document}